\documentclass[12pt]{amsart}

\usepackage{amsmath} 
\usepackage{amssymb} 
\usepackage{amsthm} 

\usepackage{microtype} 
\usepackage{pinlabel} 

\usepackage{MnSymbol} 
\usepackage[scaled=0.9]{sourcecodepro}

\usepackage[hidelinks,pagebackref,pdftex]{hyperref}

\renewcommand*{\backref}[1]{}
\renewcommand*{\backrefalt}[4]{
  \ifcase #1 
  [No citations.]
  \or [#2]
  \else [#2]
  \fi }

\let\originalleft\left
\let\originalright\right
\renewcommand{\left}{\mathopen{}\mathclose\bgroup\originalleft}
\renewcommand{\right}{\aftergroup\egroup\originalright}

\newcommand{\calD}{\mathcal{D}}

\newcommand{\calO}{\mathcal{O}}

\newcommand{\calT}{\mathcal{T}}

\newcommand{\CC}{\mathbb{C}}

\newcommand{\EE}{\mathbb{E}}

\newcommand{\HH}{\mathbb{H}}

\newcommand{\NN}{\mathbb{N}}

\newcommand{\RR}{\mathbb{R}}

\newcommand{\ZZ}{\mathbb{Z}}

\newcommand{\image}{\operatorname{image}}

\newcommand{\supp}{\operatorname{supp}} 

\newcommand{\from}{\colon} 

\newcommand{\homeo}{\mathrel{\cong}} 

\newcommand{\isom}{\cong} 
\newcommand{\cross}{\times}

\newcommand{\cover}[1]{{\widetilde{#1}}}

\newcommand{\closure}[1]{{\overline{#1}}}

\newcommand{\bdy}{\partial} 

\newcommand{\CP}{\mathbb{CP}} 

\newcommand{\PSL}{\operatorname{PSL}} 
\newcommand{\SO}{\operatorname{SO}} 

\input{header_article.tex}

\newcommand{\fakeenv}{} 

\newenvironment{restate}[2]  
{ 
 \renewcommand{\fakeenv}{#2} 
 \theoremstyle{plain} 
 \newtheorem*{\fakeenv}{#1~\ref{#2}} 
 \begin{\fakeenv}
}
{
 \end{\fakeenv}
}

\newenvironment{restated}[2]  
{ 
 \renewcommand{\fakeenv}{#2} 
 \theoremstyle{definition} 
 \newtheorem*{\fakeenv}{#1~\ref{#2}} 
 \begin{\fakeenv}
}
{
 \end{\fakeenv}
}

\usepackage{mathtools}
\usepackage{enumitem}
\usepackage{color}
\usepackage{subfig, caption}
\usepackage{wrapfig}
\usepackage{pdflscape}
\usepackage{array}
\usepackage{tikz-cd}
\usepackage{array}
\usepackage{longtable, booktabs}

\newcommand{\acw}{\rcurvearrowup}

\newcommand{\UT}[2]{\operatorname{UT}_{#1}{#2}} 

\newcommand{\PSU}{\operatorname{PSU}}
\newcommand{\PSO}{\operatorname{PSO}}

\newcommand{\sfu}{{\sf{u}}}
\newcommand{\sff}{{\sf{f}}}
\newcommand{\sfs}{{\sf{s}}}

\newcommand{\xu}{x_{\sfu}}
\newcommand{\xf}{x_{\sff}}
\newcommand{\xs}{x_{\sfs}}

\newcommand{\Haar}{\operatorname{Haar}}

\newcommand{\calV}{\mathcal{V}}

\newcommand{\sob}{\mathcal{S}}

\title[Cohomology fractals]{Cohomology fractals, Cannon--Thurston maps, and the geodesic flow}

\author[Bachman, Goerner, Schleimer, and Segerman]{David Bachman, Matthias Goerner, \\ Saul Schleimer and Henry Segerman} 
\date{\today}

\address{\hskip-\parindent
  Pitzer College,
1050 N. Mills Ave.
Claremont, CA 91711, USA}
\email{bachman@pitzer.edu}   

\address{\hskip-\parindent
        Pixar Animation Studios\\
1200 Park Avenue\\
Emeryville, CA 94608, USA}
\email{enischte@gmail.com}      

\address{\hskip-\parindent
        Mathematics Institute\\
        University of Warwick\\
        Coventry CV4 7AL, United Kingdom}
\email{s.schleimer@warwick.ac.uk}                  

\address{\hskip-\parindent
        Department of Mathematics\\
        Oklahoma State University\\
        Stillwater, OK 74078, USA}
\email{segerman@math.okstate.edu}                               
                                               
\thanks{This work is in the public domain except for Figures~\ref{Fig:Thurston} and \ref{Fig:LightningMatchUp}.}

\begin{document}

\begin{abstract}
Cohomology fractals are images naturally associated to cohomology classes in hyperbolic three-manifolds. 
We generate these images for cusped, incomplete, and closed hyperbolic three-manifolds in real-time by ray-tracing to a fixed visual radius.  
We discovered cohomology fractals while attempting to illustrate Cannon--Thurston maps without using vector graphics; we prove a correspondence between these two, when the cohomology class is dual to a fibration. 
This allows us to verify our implementations by comparing our images of cohomology fractals to existing pictures of Cannon--Thurston maps.

In a sequence of experiments, we explore the limiting behaviour of cohomology fractals as the visual radius increases. 
Motivated by these experiments, we prove that the values of the cohomology fractals are normally distributed, but with diverging standard deviations.  
In fact, the cohomology fractals do not converge to a function in the limit.
Instead, we show that the limit is a distribution on the sphere at infinity, only depending on the manifold and cohomology class.
\end{abstract}



\maketitle

\section{Introduction}

\begin{figure}[htb]
\centering
\subfloat[Cannon--Thurston map.]{
\label{Fig:CTVector}
\includegraphics[width=0.47\textwidth]{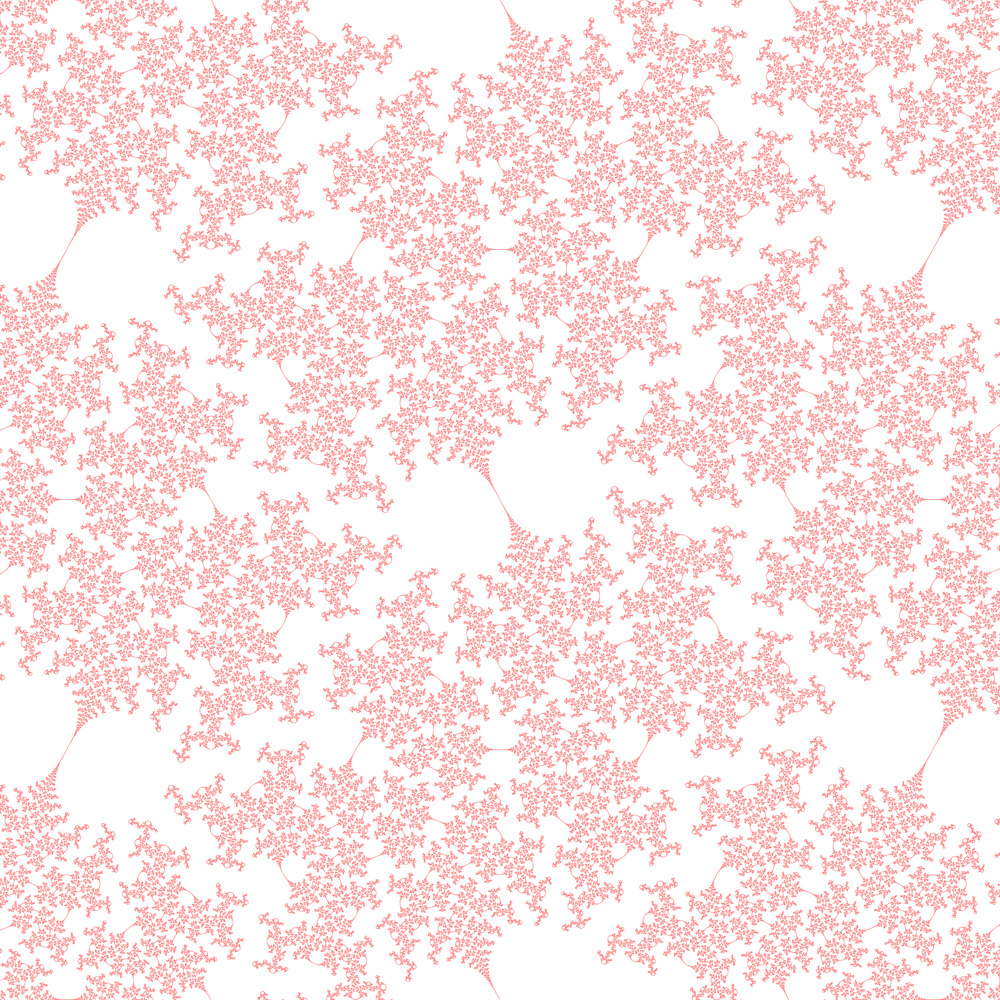}
}
\thinspace
\subfloat[Cohomology fractal.]{
\label{Fig:CTPixelColour}
\includegraphics[width=0.47\textwidth]{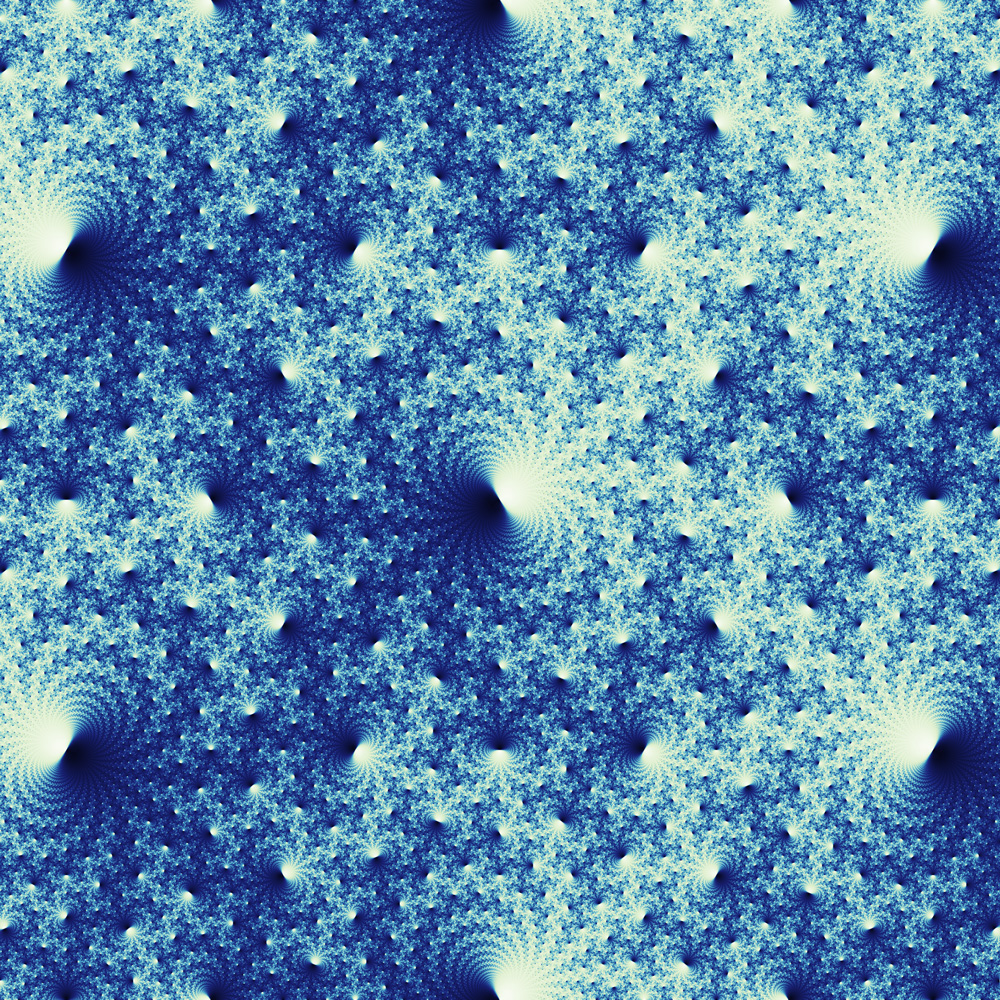}
}

\subfloat[\reffig{CTPixelColour} in black and white.]{
\label{Fig:CTPixelBW}
\includegraphics[width=0.47\textwidth]{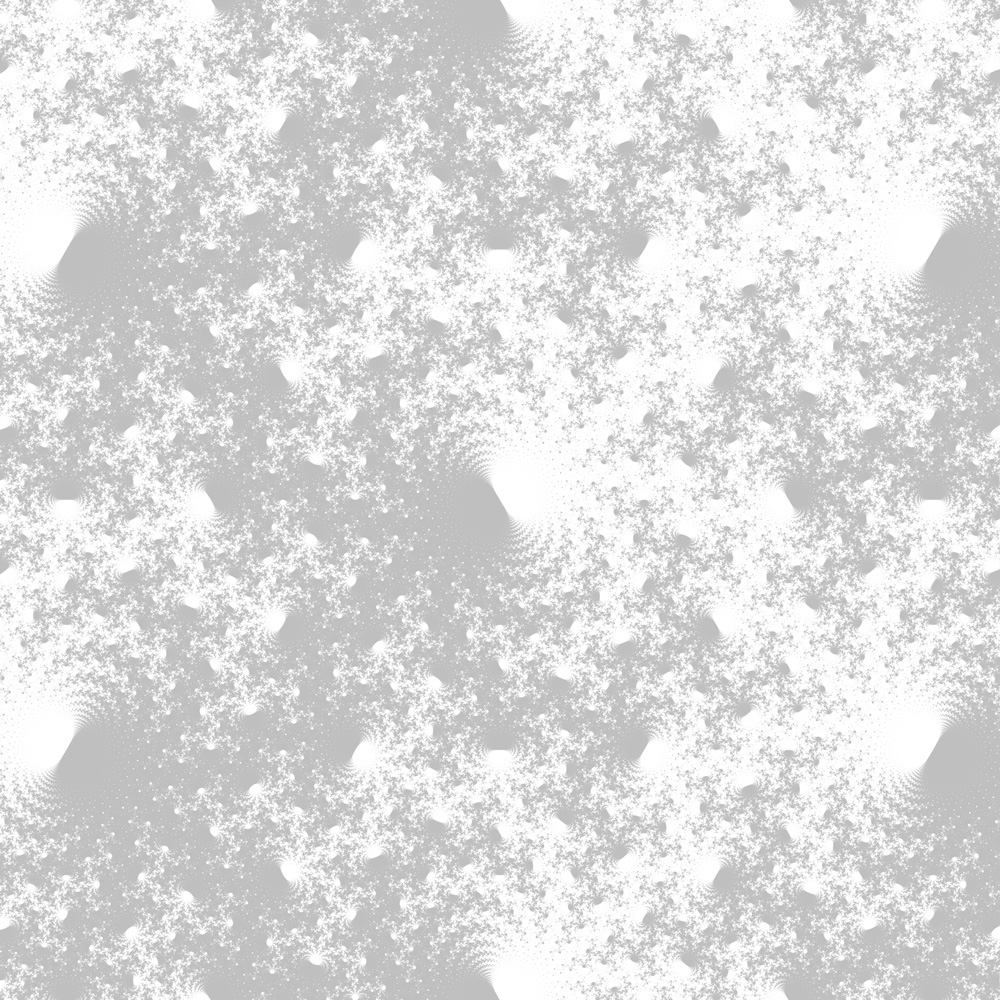}
}
\thinspace
\subfloat[\reffig{CTVector} overlaid on \reffig{CTPixelBW}.]{
\label{Fig:CTPixelBWandVector}
\includegraphics[width=0.47\textwidth]{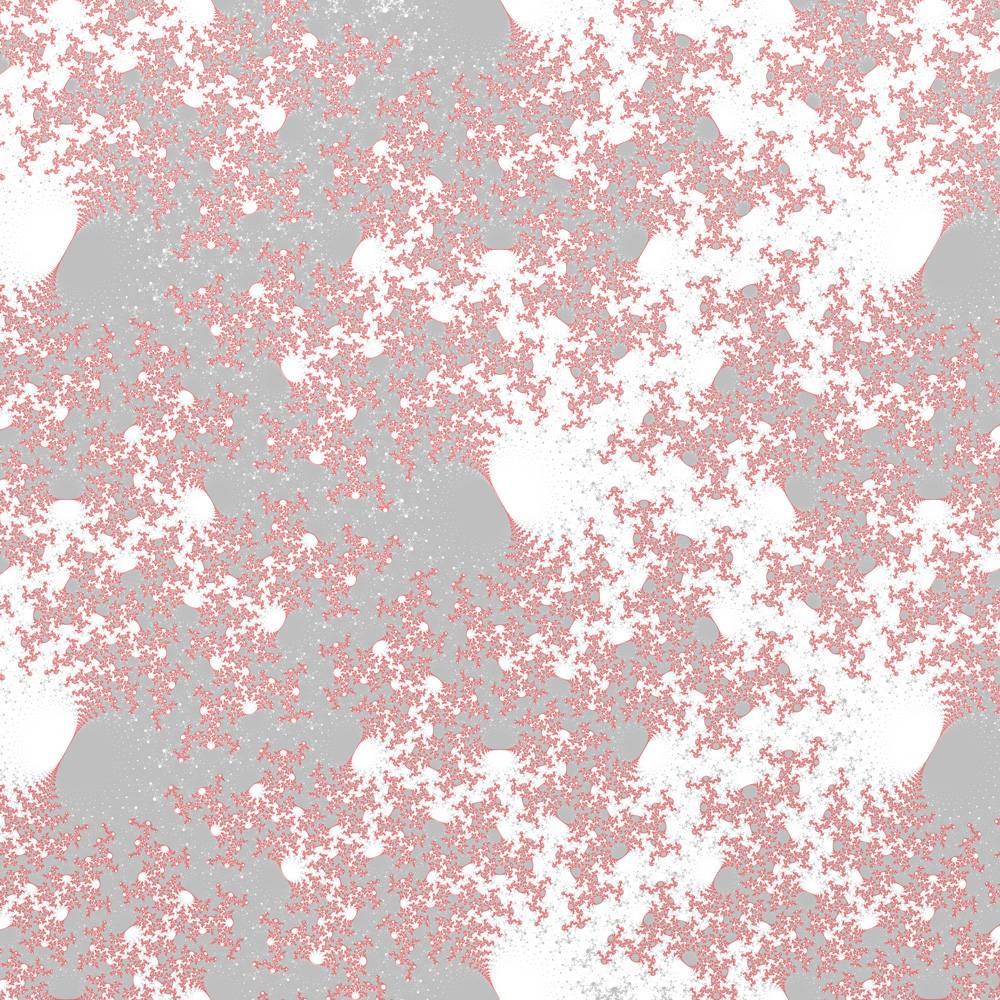}
}

\caption{Matching up Cannon--Thurston map images with the cohomology fractal for \texttt{m004}, the figure-eight knot complement. Compare our \reffig{CTPixelBW} with Figure~10.11 of \emph{Indra's Pearls}~\cite[page 335]{IndrasPearls}, which was produced by paint-filling a vector graphics image~\cite{Wright19}.}
\label{Fig:MatchUp}
\end{figure}

Cannon and Thurston discovered that Peano curves arise naturally in hyperbolic geometry~\cite{CannonThurston07}.  
They proved that for every closed hyperbolic three-manifold, equipped with a fibration over the circle, 
there is a map from the circle to the sphere that is continuous, finite to one, and surjective.  
Furthermore this \emph{Cannon--Thurston map} is equivariant with respect to the action of the fundamental group.  
We review their construction in \refsec{CannonThurston}; \reffig{CTVector} shows an approximation.

In a previous expository paper~\cite{BachmanSchleimerSegerman20}, we introduced \emph{cohomology fractals}; 
these are images arising from a hyperbolic three-manifold $M$ equipped with a cohomology class $[\omega] \in H^1(M; \RR)$.  See~\reffig{CTPixelColour}.
In that paper we gave an overview of the construction; 
we also discussed some of the features of the three-manifold and cohomology class that can be seen in its cohomology fractal.  
We have also written an open-source~\cite{github_cohomology_fractals} real-time web application for exploring these fractals. 
This is available at \url{https://henryseg.github.io/cohomology_fractals/}. 

In the present work we give rigorous definitions of cohomology fractals, we relate them to Cannon--Thurston maps (see \reffig{CTPixelBWandVector}), we give technical details of our implementation, and we discuss their limiting behaviour.  

We now outline the contents of each section of the paper. Note that we include a glossary of notation in \refapp{Notation}.
We begin by reviewing the definitions of ideal and material triangulations, and their hyperbolic geometry in \refsec{Triangulations}. In \refsec{CannonThurston} we define Cannon--Thurston maps. In \refsec{Graphics} we discuss the differences between vector and raster graphics. We also recall a vector graphics algorithm (\refalg{CTApprox}) used in previous work to illustrate Cannon--Thurston maps.

In \refsec{CohomologyFractals} we give several equivalent definitions of the cohomology fractal.
It depends on choices beyond the manifold $M$ and the cohomology class $[\omega]$: 
there is a choice of viewpoint $p \in M$ and a choice of a visual radius $R$.  
The cohomology fractal is a function $\Phi^{\omega,p}_{R} \from \UT{p}{M} \to \RR$.  
Roughly, for each vector $v \in \UT{p}{M}$ we build the geodesic arc $\gamma$ of length $R$ from $p$ in the direction of $v$ and compute $\Phi^{\omega,p}_{R}(v) = \omega(\gamma)$.  
(Note that we repeatedly generalise the definition of the cohomology fractal throughout the paper; 
the decorations alter to remind the reader of the desired context.)

In \reffig{MatchUp}, we see a cohomology fractal closely matching an approximation of a Cannon--Thurston map, as  produced by \refalg{CTApprox}. 
In \refsec{Matching} we prove the following.

\begin{restate}{Proposition}{Prop:LightDark}
Cohomology fractals are dual to approximations of the Cannon--Thurston map.
\end{restate}

Thus we have a new representation of Cannon--Thurston maps. 
We also compare cohomology fractals with the \emph{lightning curves} of Dicks and various coauthors.  
(The name is due to Wright~\cite[page~324]{IndrasPearls}.)
We experimentally observe that the lightning curve corresponds to some of the brightest points of the cohomology fractal. 
  
In \refsec{Implement} we describe the algorithms we use to produce images of cohomology fractals. 
Adding the ability to move through the manifold leads us to separate the viewpoint $p$ from a \emph{basepoint}, denoted $b$, of the cohomology fractal. 
We still trace rays starting at $p$, but then evaluate $\omega$ on any path in $\cover{M}$ from $b$ to the endpoint of $\gamma$. 

We also generalise the above \emph{material} view (with vectors $v$ in $\UT{p}{\cover{M}}$) to the \emph{ideal} and \emph{hyperideal} views (with vectors $v$ being perpendicular to a horosphere or geodesic plane, respectively). Each view is a subset $D \subset \UT{}{\cover{M}}$; our notation for the cohomology fractal becomes $\Phi^{\omega,b,D}_{R} \from D \to \RR$. 

In \refsec{Cone} we discuss cohomology fractals for incomplete and closed manifolds.  
We draw cohomology fractals in the closed case in two ways. First, we deform the cohomology fractal for a surgery parent through Thurston's \emph{Dehn surgery space}. Second, we reimplement our algorithms using material triangulations. We also discuss possible sources of numerical error in our implementations.

In \refsec{Experiments} we give a sequence of experiments exploring the dependence of cohomology fractals on the visual radius $R$. For any fixed $R$, the cohomology fractal is constant on regions with sizes roughly proportional to $\exp(-R)$. As $R$ increases, these regions subdivide, and intricate patterns come into focus. This suggests that there is a limiting object. The following shows that such a limit cannot be a function.

\begin{restate}{Theorem}{Thm:NoPicture}
Suppose that $M$ is a finite volume, oriented hyperbolic three-manifold.  Suppose that $F$ is a transversely oriented surface.  Then the limit 
\[
\lim_{R \to \infty} \Phi_R(v) 
\]
does not exist for almost all $v \in \UT{}{\cover{M}}$.
\end{restate}
Indeed, experimentally, increasing $R$ leads to noisy pictures. However, this is due to undersampling. A heuristic argument (see \refrem{Embiggen}) shows that we can avoid noise if we increase the screen resolution as we increase $R$.  We simulate this by computing \emph{supersampled} images.
These, and further experiments, indicate that in contrast with \refthm{NoPicture}, the \emph{mean} of the cohomology fractal, taken over a pixel, converges.
Its values appear to be normally distributed with standard deviation growing like $\sqrt{R}$.

Motivated by this, in \refsec{CLT}, we show that the cohomology fractal obeys a central limit theorem.

\begin{restate}{Theorem}{Thm:CLT}
Fix a connected, orientable, finite volume, complete hyperbolic three-manifold $M$ and a closed, non-exact, compactly supported one-form $\omega \in \Omega_c^1(M)$. There is $\sigma > 0$ such that for all basepoints $b$, all views $D$ with area measure $\mu_D$, for all probability measures $\nu_D\ll \mu_D$, and for all $\alpha\in\RR$, we have
\[
\lim_{T\to\infty} \nu_D\left[ v\in D : \frac{\Phi_T(v)}{\sqrt{T}} \leq \alpha \right] = \int_{-\infty}^\alpha \frac{1}{\sigma\sqrt{2\pi}} e^{-(s/\sigma)^2/2}\, ds
\]
where $\Phi_T=\Phi^{\omega,b,D}_{T}$ is the associated cohomology fractal.
\end{restate}

\noindent 
That is, if we regard the cohomology fractal across a pixel as a random variable, divide it by $\sqrt{T}$, and take the limit, the result is a normal distribution of mean zero.
The standard deviation of the normal distribution only depends on the manifold and cohomology class. 
The proof uses Sinai's central limit theorem for geodesic flows.

In \refsec{Pixel}, we prove that treating the cohomology fractals as \emph{distributions} gives a well-defined limit.
In this introduction, for simplicity, we focus on the case where $D$ is a material view. 
The \emph{pixel theorem} (\refthm{Pixel}) states that the limit
\[
\Phi^{\omega,b,D}(\eta) = \lim_{T\to\infty} \int_{D} \Phi^{\omega,b,D}_{T} \cdot \eta	
\]
is well-defined for any two-form $\eta \in \Omega^2(D)$.
 \refthm{Pixel} also states various transformation laws relating, for example, the distributions corresponding to different views. 
Thus there is a view-independent distribution related to the view-dependent distributions via the conformal isomorphism $i_D$ from $D$ to $\bdy \cover{M}$.

\begin{restate}{Corollary}{Cor:BoundaryDistribution}
Suppose that $M$ is a connected, orientable, finite volume, complete hyperbolic three-manifold.  
Fix 
a closed, compactly supported one-form $\omega \in \Omega_c^1(M)$ and 
a basepoint $b \in \cover{M}$.
Then there is a distribution $\Phi^{\omega,b}$ on $\bdy_\infty \cover{M}$ so that, 
for any material view $D$ and for any $\eta \in \Omega^2(D)$, we have
\[
\Phi^{\omega,b,D}(\eta) = 
\Phi^{\omega, b}((i_D^{-1})^*\eta)
\]

\end{restate}

The above discussion addresses smooth test functions. 
We can also prove convergence for a wider class of test functions; 
these include the indicator functions of regions with piecewise smooth boundary. 
However, we do not know whether or not the cohomology fractal converges to a measure.

We conclude with a few questions and directions for future work in \refsec{Questions}.

\subsection*{Acknowledgements}

This material is based in part upon work supported by the National Science Foundation under Grant No. DMS-1439786 and the Alfred P. Sloan Foundation award G-2019-11406 while the authors were in residence at the Institute for Computational and Experimental Research in Mathematics in Providence, RI, during the Illustrating Mathematics program.
The fourth author was supported in part by National Science Foundation grant DMS-1708239.

We thank Fran\c{c}ois Gu\'eritaud for suggesting we use ray-tracing to generate cohomology fractals.
We thank Curt McMullen for suggesting that \refthm{Pixel} should be true and also for giving us permission to reproduce \reffig{McMullen}.
We thank Mark Pollicott and Alex Kontorovich for guiding us through the literature on exponential mixing of the geodesic flow. 
We thank Ian Melbourne for enlightening conversations on central limit theorems.
We thank the anonymous referee for many helpful comments and corrections.

\section{Triangulations}
\label{Sec:Triangulations}

We briefly review the notions of material and ideal triangulations of three-manifolds.

\subsection{Combinatorics}

Suppose that $M$ is a compact, connected, oriented three-manifold. 
We will consider two cases.  Either 
\begin{itemize}
\item
the boundary $\bdy M$ is empty; here we call $M$ \emph{closed}, or  
\item
the boundary is non-empty, consisting entirely of tori; here we call $M$ \emph{cusped}. 
\end{itemize}

Suppose that $\calT$ is a \emph{triangulation}: that is, a collection of oriented model tetrahedra together with a collection of orientation-reversing face pairings.  We allow distinct faces of a tetrahedron to be glued, but we do not allow a face to be glued to itself.  The quotient space, denoted $|\calT|$, is thus a CW--complex which is an oriented three-manifold away from its zero-skeleton.  We say that $\calT$ is a \emph{material} triangulation of $M$ if there is an orientation-preserving homeomorphism from $|\calT|$ to $M$.  We say that $\calT$ is an \emph{ideal} triangulation of $M$ if there is an orientation-preserving homeomorphism from $|\calT|$, minus a small open neighbourhood of its vertices, to $M$. Equivalently, $|\calT|$ minus its vertices is homeomorphic to $M^\circ$, the interior of $M$.

\begin{example}
Suppose that $M$ is obtained from $S^3$ by removing a small open neighbourhood of the figure-eight knot.  See \reffig{FigureEightKnot}.  As discussed in~\cite[Chapter 1]{ThurstonNotes}, the knot exterior $M$ has an ideal triangulation with two tetrahedra.  See \reffig{FigureEightTriangulation}.  Here we have not truncated the model tetrahedra.  Instead we draw the vertex link; in this case it is a torus in $M$. 
\end{example}

\begin{figure}[htb]
\centering
\subfloat[{Tube around the figure-eight knot.}]{
\label{Fig:FigureEightKnot}
\includegraphics[width = 0.33\textwidth]{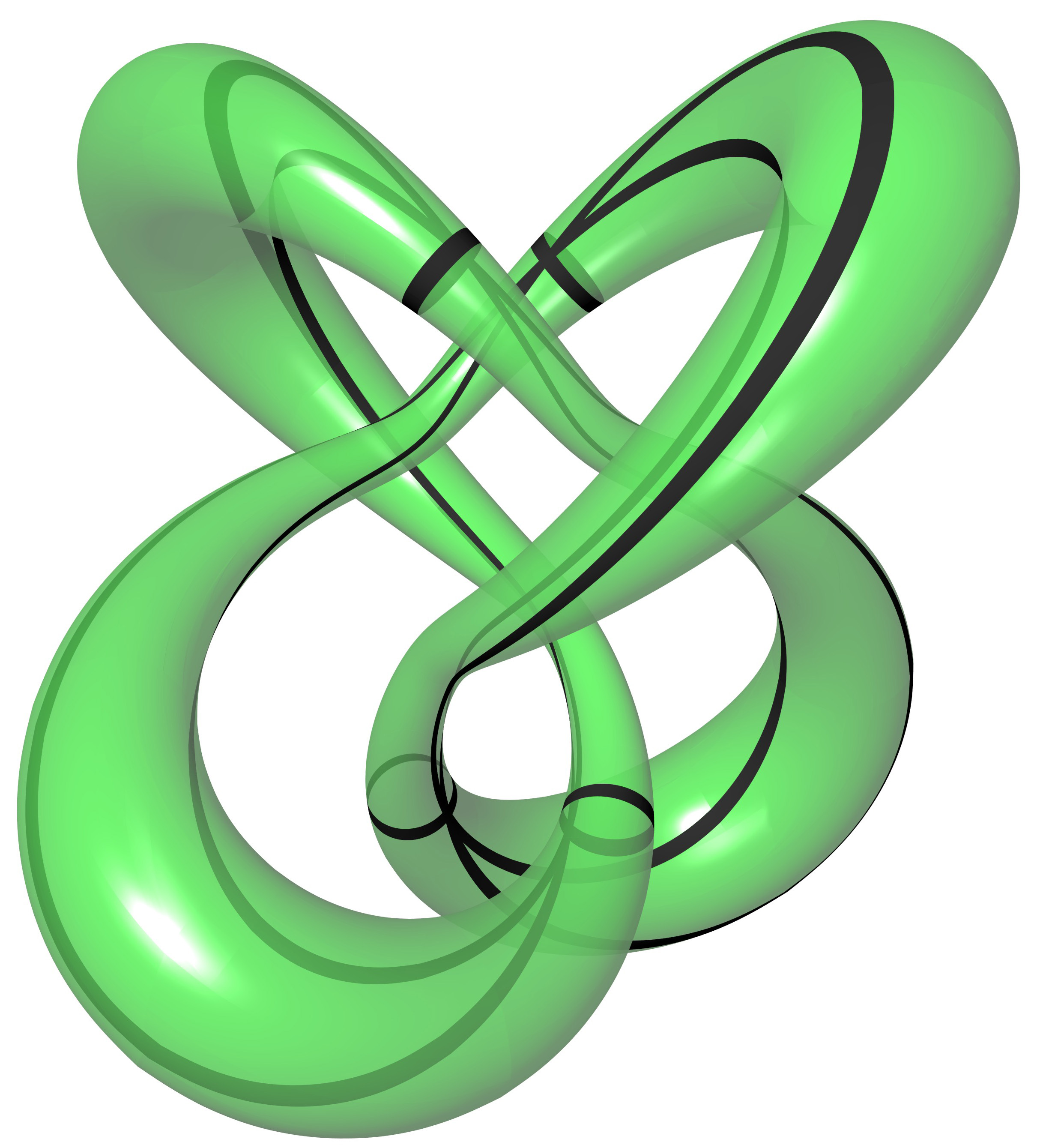}
}%
\hfill%
\subfloat[Triangulation of the figure-eight knot complement.  The colours and arrows indicate the gluings.]{
\label{Fig:FigureEightTriangulation}
\includegraphics[width = 0.61\textwidth]{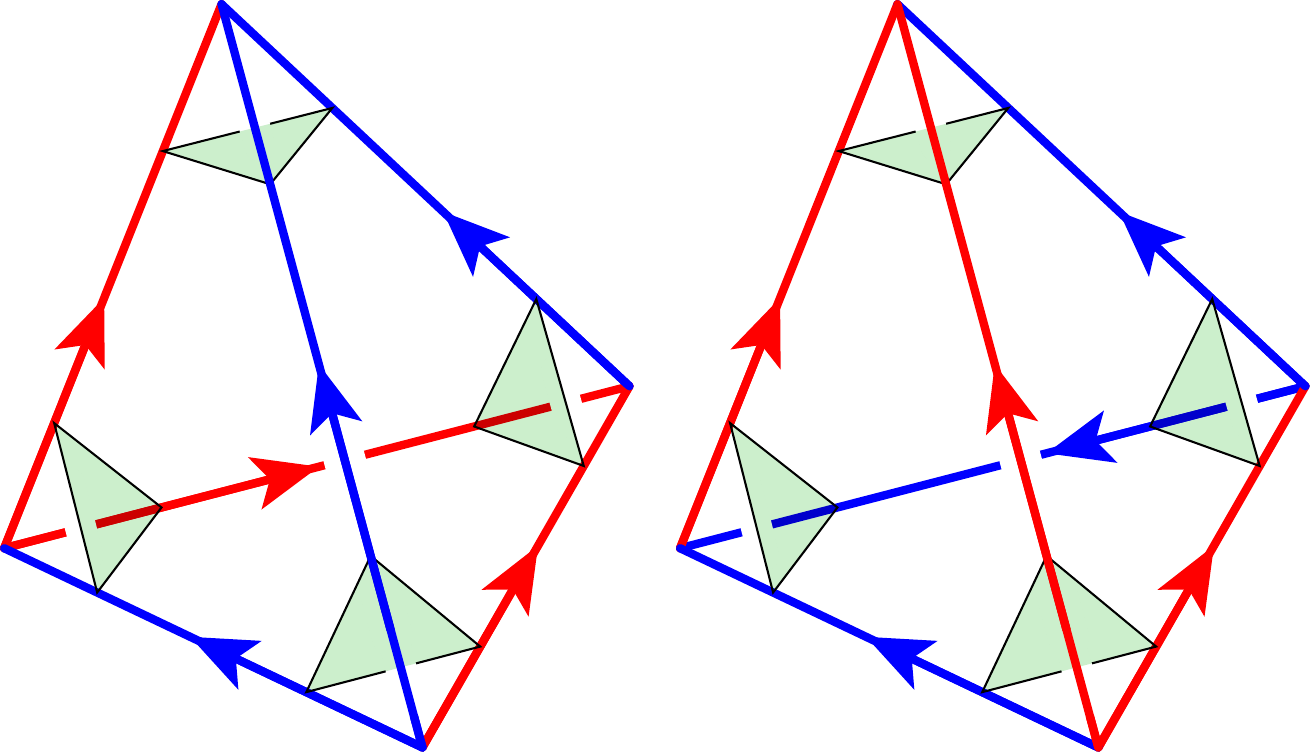}
}
\caption{The figure-eight knot complement. This manifold is known as \texttt{m004} in the SnapPy census. The black lines in \reffig{FigureEightKnot} cut the tube into triangles, corresponding to the eight green triangles in \reffig{FigureEightTriangulation}.}
\end{figure}

\subsection{Geometry}
\label{Sec:Geometry}

We deal with the geometry of the two types of triangulations separately.

\subsubsection{Ideal triangulations}

We give each model ideal tetrahedron $t$ a hyperbolic structure. 
That is, we realise $t$ as an ideal tetrahedron in $\HH^3$ with geodesic faces. 
This can be constructed as the convex hull of four points on $\bdy_\infty \HH^3$.
We require that the face pairings be orientation reversing isometries. 
For these hyperbolic tetrahedra to combine to give a complete hyperbolic structure on the manifold $M^\circ$ requires certain conditions to be satisfied. 
Very briefly: consider a loop in the dual one-skeleton of the triangulation.  
This visits the tetrahedra in some order.  The product of the corresponding sequence of isometries must give the identity if the loop is trivial in the fundamental group.  If the loop is peripheral then the product must be a parabolic element.

These conditions reduce to a finite set of algebraic constraints.  These are Thurston's \emph{gluing equations}, see~\cite[Section~4.2]{ThurstonNotes} and~\cite[Section~4.2]{PurcellKnotTheory}. Using the upper half-space model of $\HH^3$ we define the \emph{shape} of each ideal hyperbolic tetrahedron to be the cross-ratio of its four ideal points.
The gluing equations impose a finite number of polynomial conditions on these shapes. 

For our implementation, we also require that the shapes have positive imaginary part.  This ensures that the ideal hyperbolic tetrahedra glue together to give a complete, finite volume hyperbolic structure on $M^\circ$.  Furthermore, the model orientations of all of the tetrahedra agree with the orientation on $M$.  In particular, when a geodesic ray crosses a face, it has a sensible continuation.  

\subsubsection{Material triangulations}
\label{Sec:MaterialGeometry}
To find a hyperbolic structure for material triangulations, we replace Thurston's gluing equations with a construction due to Andrew Casson \cite{casson:geo} and Damian Heard \cite{Orb,heardThesis}.
To specify a hyperbolic structure on a material triangulation, it suffices to assign lengths to its edges. This is because the isometry class of an oriented material hyperbolic tetrahedron is determined by its six edge lengths.

There are two conditions that must be satisfied. First, for each model tetrahedron, there is a collection of inequalities that must be satisfied for its edges. Second, for each edge of the triangulation, the dihedral angles about it must sum to $2\pi$.

If these inequalities and equalities hold, then we obtain a hyperbolic structure on the three-manifold. See \cite[Section~2]{matthiasVerifyingFinite} for further details.

\section{Cannon--Thurston maps}
\label{Sec:CannonThurston}

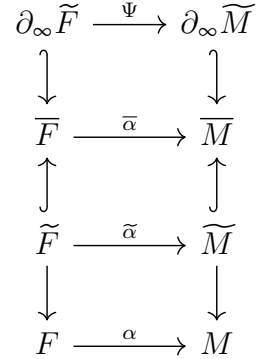
\begin{wrapfigure}[15]{r}{0.36\textwidth}
\vspace{-21pt}
\begin{minipage}{0.355\textwidth}
\[
\begin{tikzcd}
\bdy_\infty \cover{F} \arrow[r, "\Psi"] \arrow[d, hook'] & \bdy_\infty \cover{M}  \arrow[d, hook']  \\ 
\closure{F}  \arrow[r, "\closure{\alpha}"] & \closure{M} \\ 
\cover{F}  \arrow[r, "\cover{\alpha}"]  \arrow[u, hook] \arrow[d] & \cover{M} \arrow[u, hook] \arrow[d] \\ 
F \arrow[r, "\alpha"] & M
\end{tikzcd}
\]
\end{minipage}
\caption{The various spaces and maps involved in constructing the Cannon--Thurston map $\Psi$.}
\label{Fig:CannonThurston}
\end{wrapfigure}

Here we sketch Cannon and Thurston's construction; see \reffig{CannonThurston} for an overview. 
We refer to~\cite{CannonThurston07} for the details. See also~\cite{Mj18}.

Suppose that $F=F^2$ and $M=M^3$ are connected, compact, oriented two- and three-manifolds.  Suppose that $F^\circ$ and $M^\circ$ admit complete hyperbolic metrics of finite area and volume respectively. In an abuse of notation, we will conflate $F$ with $F^\circ$, and similarly $M$ with $M^\circ$.
We call a proper embedding $\alpha \from F \to M$ a \emph{fibre} if there is a map $\rho \from M \to S^1$ so that for all $t \in S^1$ the preimage $\rho^{-1}(t)$ is a surface properly isotopic to $\alpha(F)$. 

Let $\cover{F}$ and $\cover{M}$ be the universal covers of $F$ and $M$ respectively. Since $F$ and $M$ are hyperbolic, their covers are identified with hyperbolic two- and three-space respectively. Let $\bdy_\infty \cover{F} \homeo \bdy_\infty\HH^2 \homeo S^1$ and $\bdy_\infty \cover{M} \homeo \bdy_\infty\HH^3 \homeo S^2$ be their ideal boundaries. 
We set 
\[
\closure{F} = \cover{F} \cup \bdy_\infty \cover{F} \qquad \mbox{and} \qquad \closure{M} = \cover{M} \cup \bdy_\infty \cover{M}
\]
Each union is equipped with the unique topology that makes the group action continuous.  
Note that $\closure{F}$ and $\closure{M}$ are homeomorphic to a closed two- and three-ball, respectively. 

We compose the covering map from $\cover{F} \to F$ with the embedding $\alpha$ and then lift to obtain an equivariant map $\cover{\alpha} \from \cover{F} \to \cover{M}$.   
We call $\cover{\alpha}$ an \emph{elevation} of $F$.
\reffig{Elevation} shows an elevation of the fibre of the figure-eight knot complement.

Cannon and Thurston gave the first proof of the following theorem in the closed case~\cite[page 1319]{CannonThurston07}.  
The cusped case follows from work of Bowditch~\cite[Theorem 0.1]{Bowditch07}.

\begin{theorem}
\label{Thm:CannonThurston}
Suppose $M$ is a connected, oriented, finite volume hyperbolic three-manifold. 
Suppose that $\alpha \from F \to M$ is a fibre of a surface bundle structure on $M$. 
Then there is an extension of $\cover{\alpha}$ to a continuous and equivariant (with respect to the fundamental group of $M$) map $\closure{\alpha} \from \closure{F} \to \closure{M}$. 
The restriction of $\closure{\alpha}$ to $\bdy_\infty \cover{F}$ gives a sphere-filling curve. \qed
\end{theorem}

We will use the notation $\Psi \from \bdy_\infty \cover{F} \to \bdy_\infty \cover{M}$ for the restriction of $\closure{\alpha}$ to $S^1_\infty$. 
We call this a \emph{Cannon--Thurston map}.  
We now turn to the task of visualising $\Psi$.

\section{Illustrating Cannon--Thurston maps}
\label{Sec:Graphics}

The standard joke (see~\cite[page~373]{Thurston82} and~\cite[page~335]{IndrasPearls}) is that it is straightforward to draw an accurate picture of a Cannon--Thurston map; it is solid black.

\begin{figure}[htb]
\centering
\subfloat[An elevation of the fibre. The fibre is a pleated surface made from two ideal triangles. The three pleating angles are $\pi/3$, $\pi$, and $5\pi/3$.]{
\includegraphics[width = 0.47\textwidth]{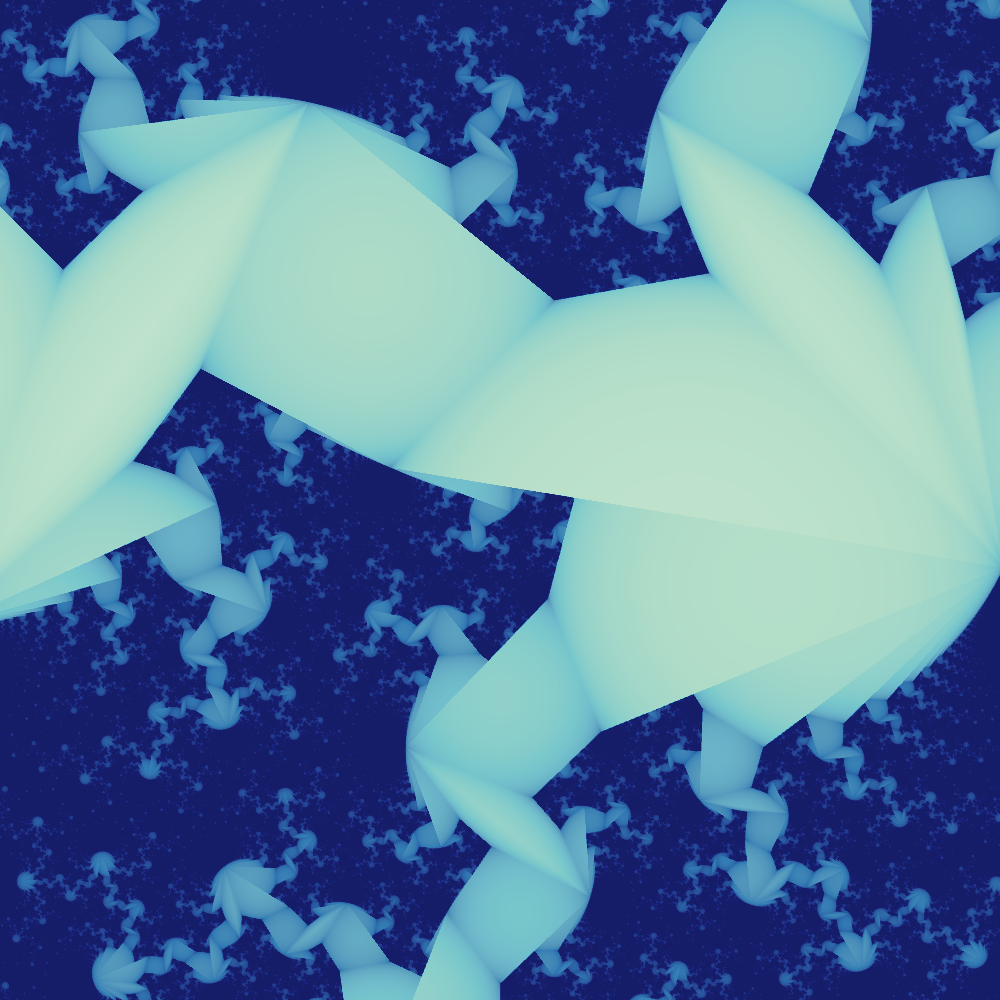}
\label{Fig:Elevation}
}
\subfloat[An approximation of the Cannon--Thurston map, reproduced from Figure~8 of~\cite{Thurston82}.]
{
\includegraphics[width = 0.47\textwidth]{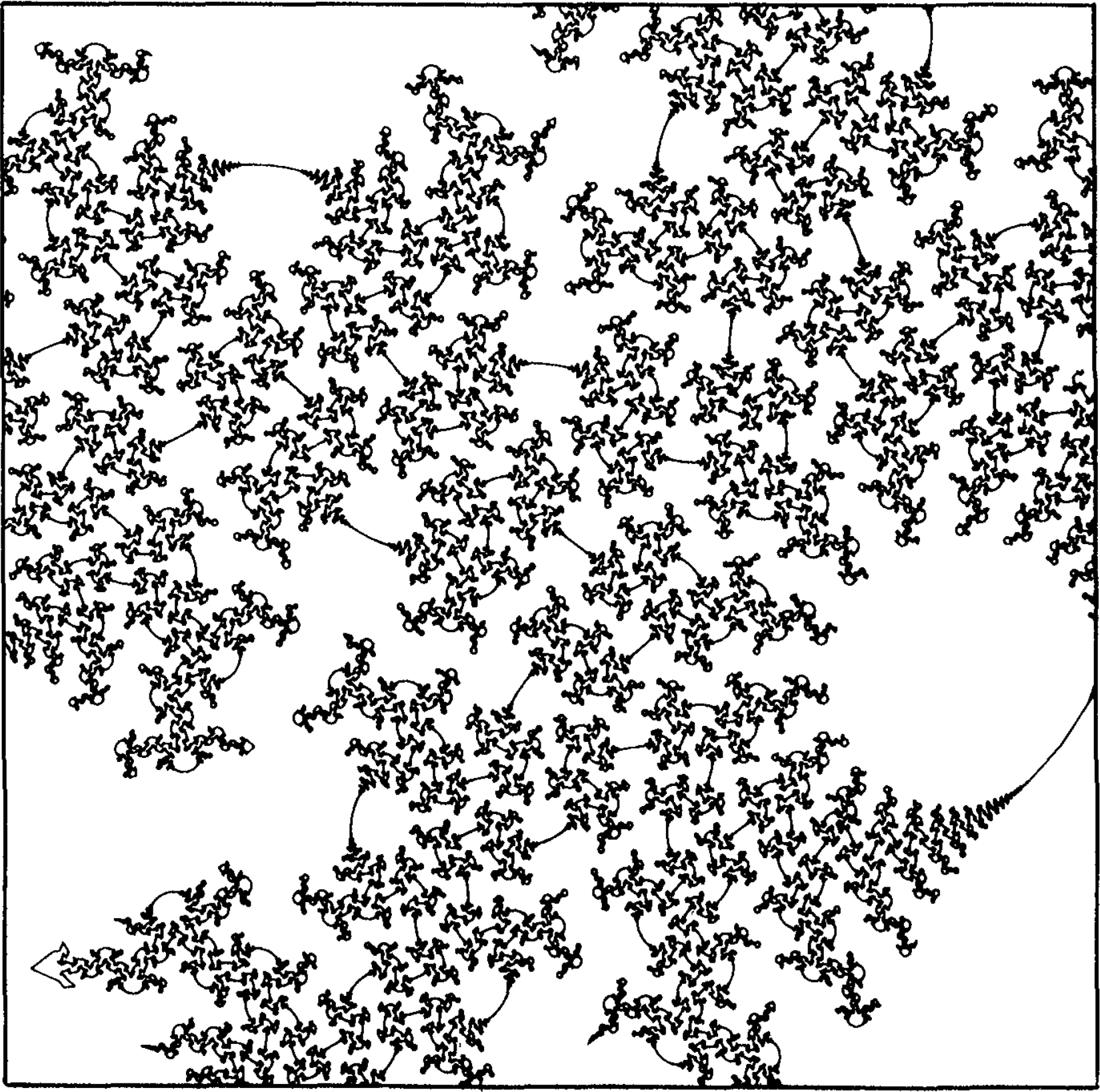}
\label{Fig:Thurston}
}
\caption{Views in the universal cover of the figure-eight knot complement.  }
\end{figure}

The first (more instructive) illustration of a Cannon--Thurston map is due to Thurston. He gives a sequence of approximations to the sphere-filling curve in~\cite[Figure~8]{Thurston82}.
We reproduce the last of these in \reffig{Thurston}. A striking version of this image by Wright also appears in \emph{Indra's Pearls}~\cite[Figure~10.11]{IndrasPearls}. 
In this example both $M$ and $F$ are non-compact; $M$ is the complement of the figure-eight knot and $F$ is a Seifert surface.  

\subsection{Vector and raster graphics}

In this section, we outline the technique used by Thurston and Wright to generate images of Cannon--Thurston maps,
in order to contrast it with our algorithm.

Our algorithm generates an image by producing a colour for each pixel on a screen. In other words, its output is a map from a grid of pixels in the \emph{image plane} into a space of possible colours. We call such a map a \emph{raster graphics image}.

In contrast, \refalg{CTApprox} (below) produces \emph{vector graphics} -- that is a description of an image as a collection of various primitive objects in the (euclidean) image plane. 
An example of a primitive is a line segment, specified by the coordinates of its end points. 
Other primitives include arcs, circles, and so on. 
Note that we generally need to convert vector graphics to raster graphics to make a physical representation of an image. 
To \emph{rasterise} a vector graphics image, we need to decide which pixels are coloured by which primitives. 
For example, a disk colours all of the pixels whose coordinates are close enough to the centre of the disk. 
Rasterisation is necessary for most output devices, such as screens or printers. 
The exceptions include plotters, laser cutters, and cathode-ray oscilloscopes. 
There are advantages to deferring rasterisation and saving the vector graphics to a file (in the PDF, PostScript, or SVG format, for example). For example, deferred rasterisation can take the resolution of the output device into account. Design programs such as Inkscape and Adobe Illustrator allow editing the geometric primitives in a vector graphics image. Rasterisation is usually carried out by a black-box general purpose algorithm, the details of which are hidden from the user.

\begin{algorithm}
\label{Alg:CTApprox} 
(Approximate a Cannon--Thurston map)
We are given a fibre $F$ of the three-manifold $M$. 
We choose an elevation $\cover{F} \subset \cover{M}$ of $F$. 
As described in \refsec{CannonThurston}, the map 
\[
\Psi \from S^1 \homeo \bdy_\infty \cover{F} \to \bdy_\infty \cover{M} \homeo S^2
\]
is sphere-filling. 

To approximate $\Psi$, we first choose a large disk $D \subset \cover{F}$.  
Typically, $M$ is described with an ideal triangulation $\calT$, with $F$ realised as a surface carried by the two-skeleton $\calT^{(2)}$.  
Therefore $\cover{F}$ is given as a surface carried by $\cover{\calT}^{(2)}$.  
The disk $D$ then consists of some finite collection of ideal hyperbolic triangles in $\cover{\calT}^{(2)}$.  
The boundary of $D$ consists of a loop of geodesics in $\HH^3 \cup \bdy_\infty \HH^3$.  
We now define $\Psi_D$ to be the loop in $\bdy_\infty \HH^3$ obtained by projecting each arc of $\bdy D$ to an arc in $\bdy_\infty \HH^3$. 
\end{algorithm} 

Note that the algorithm produces a circularly ordered collection of points in $\bdy_\infty \HH^3$ spanning geodesics in $\HH^3$. 
However, conventional vector graphics require primitives to be in the euclidean plane. 
Thus, we must make two choices of projections. 
The first projection from $\HH^3$ to $\bdy_\infty\HH^3$ takes the geodesics to arcs in $\bdy_\infty\HH^3$ and the second projection takes these arcs in $\bdy_\infty\HH^3$ to arcs in the euclidean image plane. 

We draw our pictures in the ``ideal view''.  
That is, we use the upper half space model of $\HH^3$ and project $\bdy D$ down to $\CC$ (viewed as the boundary of $\HH^3$). 
The arcs between vertices now simply become straight lines in $\CC$.  
This is also the choice made by Thurston in~\reffig{Thurston}, as well as Wada in his program \emph{OPTi}~\cite{opti}.   
Other depictions by McMullen~\cite{McMullenWeb} and Calegari~\cite[Figure~1.14]{Calegari07}
project outward from the origin to the boundary of the Poincar\'e ball model of $\HH^3$ (and then to the image plane using a perspective or orthogonal projection).  

\begin{remark}
The images in \emph{Indra's Pearls}~\cite{IndrasPearls} have been rasterised using a customised rasteriser that illustrates further features of the Cannon--Thurston map. For example, one side of the polygonal path of line segments is filled, or 
the line segments are coloured using some combinatorial condition. See Figures~10.11 and~10.13 of~\cite{IndrasPearls}. 
\end{remark}

\subsection{Motivating raster graphics}

Our work here began when we asked if we could avoid vector graphics when illustrating Cannon--Thurston maps.  
This is less natural, but would allow us to take advantage of extremely fast graphics processing unit (GPU) calculation.
We were inspired in part by work of Vladimir Bulatov~\cite{Bulatov18} and also of Roice Nelson and the fourth author~\cite{visualizing_hyperbolic_honeycombs}, using reflection orbihedra. (See also the work of Peter Stampfli~\cite{Stampfli19}.)
They all use raster graphics strategies to draw tilings of $\HH^2$ and $\HH^3$. 

A number of others have also used raster graphics to explore kleinian groups, 
outside of the setting of reflection orbihedra.  
They include Peter Liepa~\cite{Liepa11}, Jos Leys~\cite{Leys17}, and Abdelaziz Nait Merzouk~\cite{Knighty} (also see~\cite{Christensen}).

\section{Cohomology fractals}
\label{Sec:CohomologyFractals}

We give a sequence of more-or-less equivalent definitions of cohomology fractals, beginning with the conceptually simplest (for us), and moving towards versions that are most convenient for our implementation or our proofs.
Fixing notation, we take $M$ to be a riemannian manifold and $\cover{M}$ its universal cover.
We also take $X = X_M = \UT{}{M}$ and $\cover{X} = \UT{}{\cover{M}}$ to be their unit tangent bundles.
Take $\pi \from \cover{X} \to \cover{M}$ to be the bundle map.  

\begin{definition}
\label{Def:Basic}
Suppose that we are given the following data. 
\begin{itemize}
\item A connected, complete, oriented riemannian manifold $M^n$, 
\item a cocycle $\omega \in Z^1(M; \ZZ)$, 
\item a point $p \in M$, and 
\item a radius $R \in \RR_{>0}$.
\end{itemize}

From these, we define the \emph{cohomology fractal} on the unit sphere $\UT{p}{M} \homeo S^{n-1}$. 
This is a function $\Phi_R = \Phi^{\omega,p}_{R} \from \UT{p}{M} \to \ZZ$ defined as follows.
 
Suppose that $v \in \UT{p}{M}$ is a unit tangent vector.  
Let $\gamma$ be the unique geodesic segment starting at $p$ with initial direction $v$ and of length $R$.  
Let $q$ be the endpoint of $\gamma$. 
Choose any shortest path $\gamma'$ from $q$ to $p$ (on a set of full measure $\gamma'$ is unique). 
Thus $\gamma \cup \gamma'$ is a one-cycle. 
We define $\Phi_R(v) = \omega(\gamma \cup \gamma')$.
\end{definition}

Our next definition moves in the direction of concrete examples:

\begin{definition}
\label{Def:Dual}
Here we further assume that $M$ is a three-manifold.  
Let $F \subset M$ be a properly embedded, transversely oriented surface.
We choose $F$ so that $p \notin F$.
We define $\gamma$ and $q$ as above.
Now take $\gamma'$ to be the shortest path from $q$ to $p$ in the complement of $F$. 
We now define $\Phi_R(v) = \Phi^{F,p}_{R}(v)$ to be the algebraic intersection number between $F$ and $\gamma \cup \gamma'$.
\end{definition}

We modify once again to obtain a definition very close to our implementation. 

\begin{definition}
\label{Def:UsingTriangulation}
We equip $M$ with a material (or ideal) triangulation $\calT$. 
We properly homotope the surface $F$ to lie in the two-skeleton $\calT^{(2)}$. 
For each face $f$ this gives us a weight $\omega(f)$.
This is the signed number of sheets of $F$ running across $f$.
We dispense with $\gamma'$; 
we take $\Phi_R(v)$ to be the sum of the weighted intersections between $\gamma$ and the faces of the triangulation.
\end{definition}

To aid in comparing cohomology fractals to Cannon--Thurston maps (in \refsec{Matching}), we lift to the universal cover, $\cover{M}$.

\begin{definition}
\label{Def:UniversalCover}
Since cochains pull back, let $\cover{\omega}$ be the lift of $\omega$. 
Let $\cover{p}$ be a fixed lift of the point $p$.
Since $\cover{\omega}$ is a coboundary, it has a primitive, say $W$; we choose $W$ so that $W(\cover{p}) = 0$.
We form $\cover{\gamma}$ as before and let $\cover{q}$ be its endpoint. 
We define $\Phi_R(v) = W(\cover{q})$.
\end{definition}

To analyse the behaviour of the cohomology fractal as $R$ tends to infinity, we rephrase our definition in a dynamical setting. Here, the radius $R$ is replaced by a time $T$. 

\begin{definition}
\label{Def:OneForm}
Suppose that $\omega \in \Omega^1(M,\RR)$ is a closed one-form. 
Let $\varphi_t \from \UT{}{M} \to \UT{}{M}$ be the geodesic flow for time $t$. We define 
\[
\Phi_T(v) = \Phi_{T}^{\omega,p}(v) = \int_0^T \omega(\varphi_t(v))\, dt \qedhere
\]
\end{definition}

In a slight abuse of notation, we also use $\varphi_t$ to denote the geodesic flow on $\UT{}{\cover{M}}$. 

\begin{figure}[htb!]
\centering
\subfloat[$R=e^{0.5}$]{
\label{Fig:exp0.5}
\includegraphics[width=0.47\textwidth]{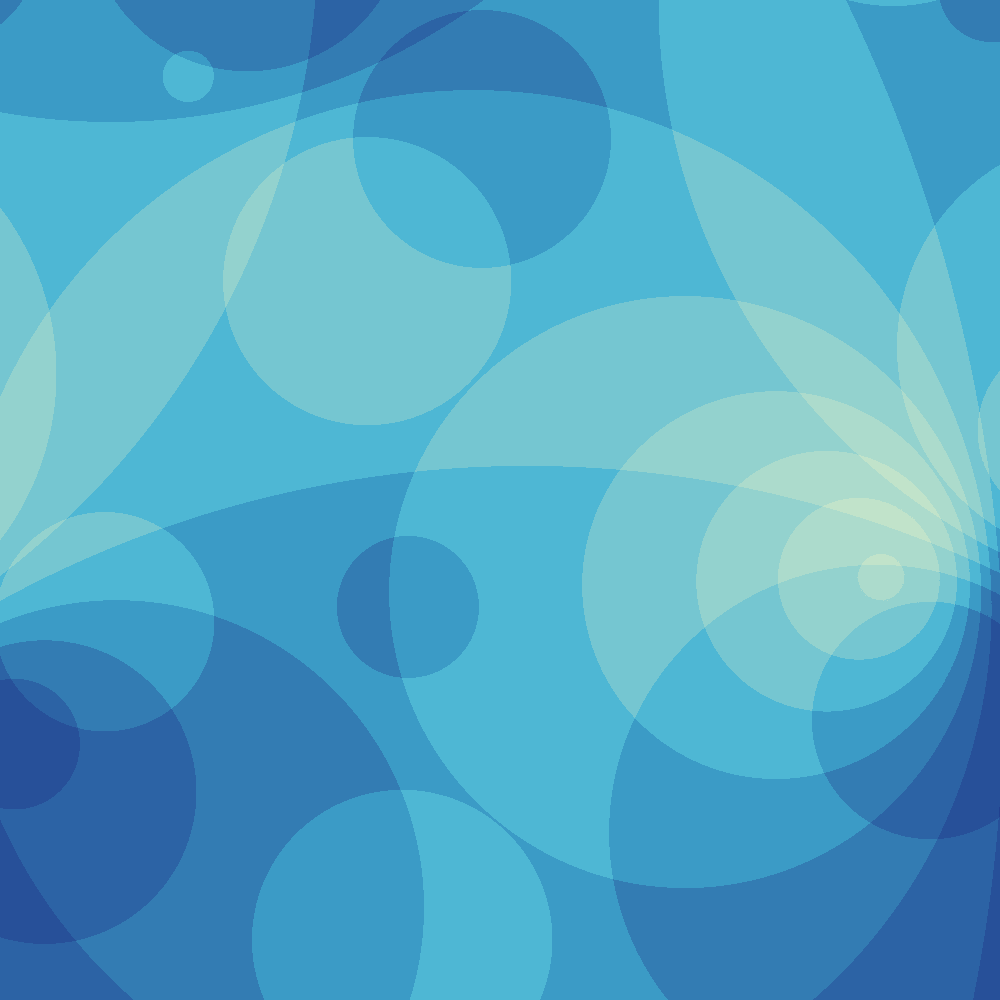}
}
\subfloat[$R=e^1$]{
\includegraphics[width=0.47\textwidth]{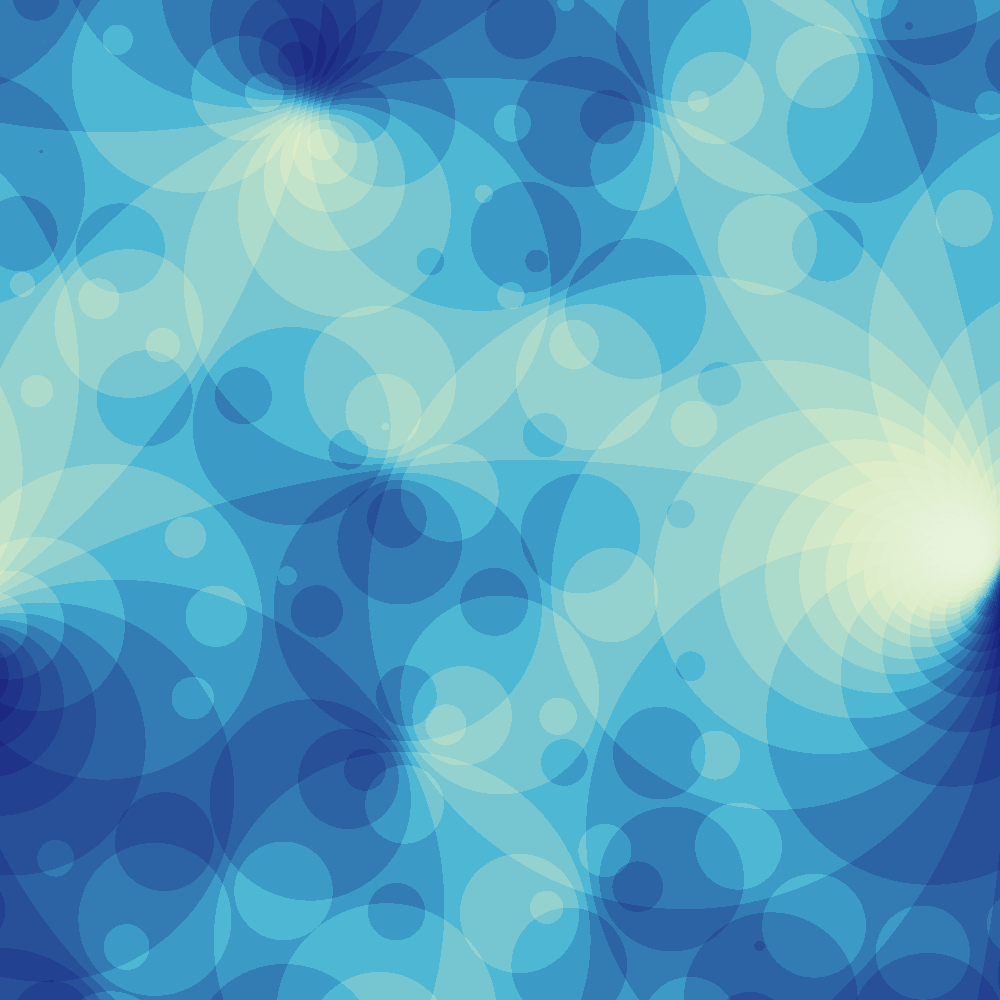}
}
\\
\subfloat[$R=e^{1.5}$]{
\includegraphics[width=0.47\textwidth]{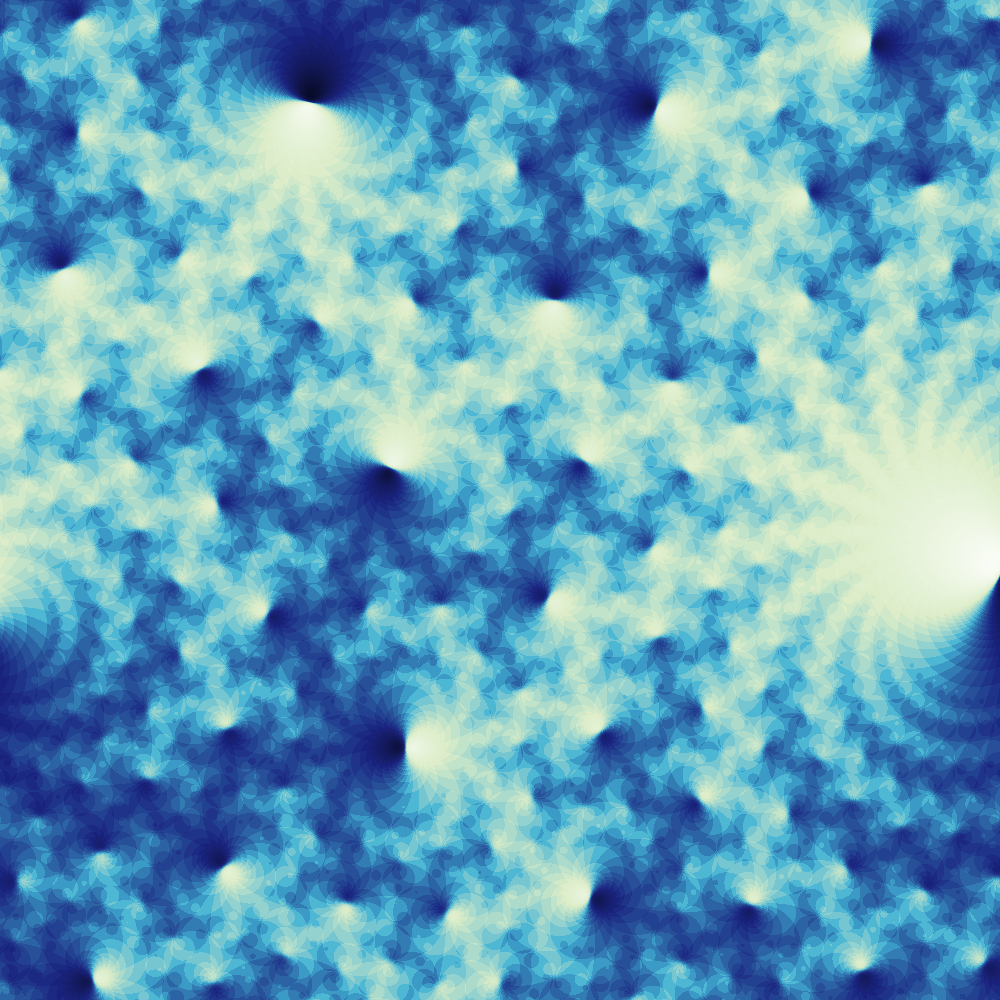}
}
\subfloat[$R=e^{2}$]{
\includegraphics[width=0.47\textwidth]{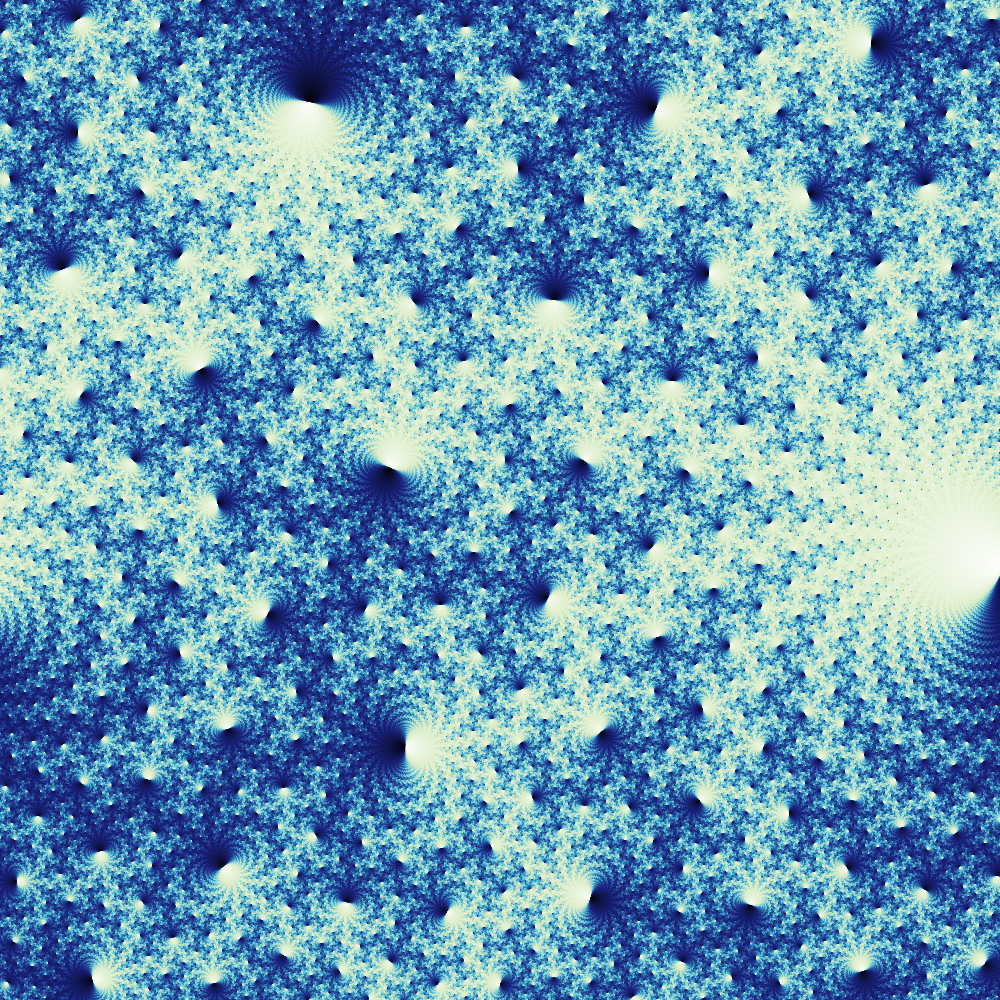}
}
\caption{Cohomology fractals for \texttt{m004}, with various values of $R$.}
\label{Fig:VisSphereRadii}
\end{figure}

In \refsec{Implement} we will discuss how to calculate $\Phi_R$ in practice. 
Before giving those details, we show the reader what $\Phi_R$ looks like for a few values of $R$.  See \reffig{VisSphereRadii}.  The map $\Phi_R$ maps into $\RR$; we indicate the value of $\Phi_R(v)$ by brightness.  For each value of $R$, we draw the value of $\Phi_R(v)$ for a small square subset of the unit tangent vectors, $\UT{p}{M}$.

Here we are using \refdef{UsingTriangulation}, our manifold $M$ is the complement of the figure-eight knot, and the surface $F$ is a fibre of $M$. 
Note that when $R$ is small, as in \reffig{exp0.5}, $\Phi_R$ is constant on large regions of the sphere. As $R$ increases, the value of $\Phi_R$ on nearby rays becomes less correlated, and we see a fractal structure come into focus.

\begin{remark}
This complicated behaviour is a consequence of the hyperbolic geometry of our manifold. 
Consider instead the example where $M$ is the three-torus $S^1 \times S^1 \times S^1$ and the surface $F$ is an essential torus embedded in $M$. 
Again, $\Phi_R$ counts the number of elevations of $F$ the ray $\gamma$ passes through. 
Since the elevations are parallel planes in $\cover{M} \homeo \RR^3$, the value of $\Phi_R$ is constant on circles in $\UT{p}{M}$ parallel to these planes. 
Here $\Phi_R$ is much simpler.
\end{remark}

\begin{remark}
Some of the geometry and topology of the manifold $M$ can be seen from the cohomology fractal $\Phi_R$. 
Our recent expository paper on cohomology fractals~\cite{BachmanSchleimerSegerman20} gives many such examples, including the appearances of cusps, totally geodesic subsurfaces, and loxodromic elements of the fundamental group.
\end{remark}

\section{Matching figures}
\label{Sec:Matching}

In this section we compare our cohomology fractals to Cannon--Thurston maps.  

\begin{example}
Suppose that $M$ is the figure-eight knot complement.  
Suppose that $F$ is a fibre of the fibration of $M$.  
Suppose that $\Psi$ is the associated Cannon--Thurston map.  
\reffig{CTVector} shows an approximation $\Psi_D \from S^1 \to S^2 \homeo \bdy_\infty \HH^3$ of $\Psi$; 
we produced this image using \refalg{CTApprox}. (Note that the vector graphics image in \reffig{CTVector} has been converted to a raster graphics image to save on file size and rendering time.)


\reffig{CTPixelColour} shows a cohomology fractal $\Phi_R$ corresponding to $F$ and looking towards the same part of $\bdy_\infty \HH^3$. 
\reffig{CTPixelBW} shows $\Phi_R$ again, but with the contrast increased and colour scheme simplified.  
Here the colour associated to a vector $v$ is either white or grey, according to whether $\Phi(v)$ is negative or not.  

\reffig{CTPixelBWandVector} shows \ref{Fig:CTVector} overlaid on \ref{Fig:CTPixelBW}. 
We see that the red curve of $\Psi_D$ is almost the common boundary of the white and grey regions of $\Phi_R$.  
There are several small areas where $\Psi_D$ does not track the boundary.  
These only appear close to fairly large cusps; 
they exist because implementations of \refalg{CTApprox} generally have trouble approaching cusps from the side.
In the cohomology fractal $\Phi_R$ we see that there are chains of ``octopus heads'' that reach almost all the way towards each cusp. 
\end{example}

This behaviour is generally true for fibrations, as follows. Suppose $\cover{p}\in\cover{M}$ and define $\rho \from \bdy_\infty \cover{M} \to \UT{\cover{p}}{\cover{M}}$ to be the inward central projection. 
Here we follow \refdef{UniversalCover}.

\begin{proposition}
\label{Prop:LightDark}
Suppose that $M$ is a connected, oriented, finite volume hyperbolic three-manifold. 
Suppose that $\alpha \from F \to M$ is a fibre of a surface bundle structure on $M$. 
Fix $p$ in $F$, and a lift $\cover{p} \in \cover{M}$. 
Fix any $R > 0$ and let $\Phi_R = \Phi^{F,p}_R$ be the resulting cohomology fractal for $F$.  
Let $Z = \Phi_R^{-1}(\RR_{\geq 0})$.

Then there is a disk $D \subset \cover{F}$, containing $\cover{p}$, so that (the image of) the Cannon--Thurston map approximation $\rho \circ \Psi_D$ is a component of $\bdy Z$ (with error at most $\exp(-R)$). 
\end{proposition}


\begin{proof}[Proof sketch]
We assume we are in the setting of \refdef{UsingTriangulation}.
Let $\omega$ be the one-cocycle dual to $F$.
Let $W$ be a primitive for $\cover{\omega}$.  
Consider the two regions of $\HH^3 \isom \cover{M}$ where $W$ is negative or, respectively, non-negative.  
The common boundary of these is exactly an elevation $\cover{F}$ of the fibre $F$. 

Let $B^3_R$ be the ball in $\HH^3$ of radius $R$ with centre $\cover{p}$. 
Let $S^2_R$ be the boundary of $B^3_R$. 
Thus the intersection $\cover{F} \cap S^2_R$ is a collection of curves; these separate the points $\cover{q} \in S^2_R$ where $W$ is negative from those where it is non-negative.  
Finally, note that $\pi \circ \varphi_R \from \UT{\cover{p}}{\HH^3} \to S^2_R$ is the exponential map.  
Thus $\Phi_R = W \circ \pi \circ \varphi_R$.

Recall that $\cover{F}$ is a union of triangles. 
Assume that one of these contains $\cover{p}$.  
Let $E$ be the collection of triangles of $\cover{F}$ having at least one edge meeting the ball $B^3_R$.  
Let $D$ be the connected component of (the union over) $E$ that contains $\cover{p}$.  

In a slight abuse of notation, let $\rho \from \HH^3 - \{\cover{p}\} \to \UT{\cover{p}}{\HH^3}$ be the inward central projection.  
For any set $K \subset \HH^3 - \{\cover{p}\}$ we call the diameter of $\rho(K)$ the \emph{visual diameter} of $K$.  
This is measured with respect to the fixed metric on the unit two-sphere $\UT{\cover{p}}{\HH^3}$. 

Suppose that $e$ is a bi-infinite geodesic in $\HH^3$. 
If $e$ lies outside of $B^3_R$ then the visual diameter of $e$ is small; 
in fact, for large $R$ the visual diameter of $e$ is less than $2\exp(-R)$.  
Likewise, if $e$ meets $S^2_R$, then for either component $e'$ of $e - B^3_R$ the visual diameter of $e'$ is less than $\exp(-R)$.

Now suppose that $T$ is a triangle of $E$.  
Using the above, we deduce that the visual diameter of each component of $T - B^3_R$ is small. 
Thus the inward central projections of $\bdy D$ and $D \cap S^2_R$ have Hausdorff distance bounded by a small multiple of $\exp(-R)$. 

So let $\Psi_D$ be the curve in $\bdy_\infty \HH^3$ obtained by projecting $\bdy D$ outwards.   
By the above, each arc of the resulting polygonal curve has small visual diameter.  
Also by the above, the Hausdorff distance between the curves $\rho \circ \Psi_D$ and $\rho \circ (\cover{F} \cap S^2_R)$ is small.  
\end{proof}

\begin{remark}
Suppose that $F$ is totally geodesic or, more generally, quasi-fuchsian. 
In this case, the Cannon--Thurston map $\Psi$ is a circle or quasi-circle, respectively.
Note however that different elevations now give distinct Cannon--Thurston maps.
It is natural to take their union and obtain a circle (or quasi-circle) packing.  

Now, if $F$ is also Thurston-norm minimising then we still obtain matches. 
For example in~\cite[Figure~6]{BachmanSchleimerSegerman20}, we see how, for a totally geodesic surface in the Whitehead link complement, the cohomology fractal matches the associated circle packing.
On the other hand, if $[F]$ is trivial in $H_2(M, \bdy M)$ then the cohomology fractal $\Phi_R$ is bounded and oscillates as $R$ tends to infinity.  
\end{remark}

\subsection{Lightning curves}
\label{Sec:Lightning}

\begin{figure}[htb]
\centering
\subfloat[The lightning curve superimposed on the cohomology fractal.]{
\includegraphics[width = 0.47\textwidth]{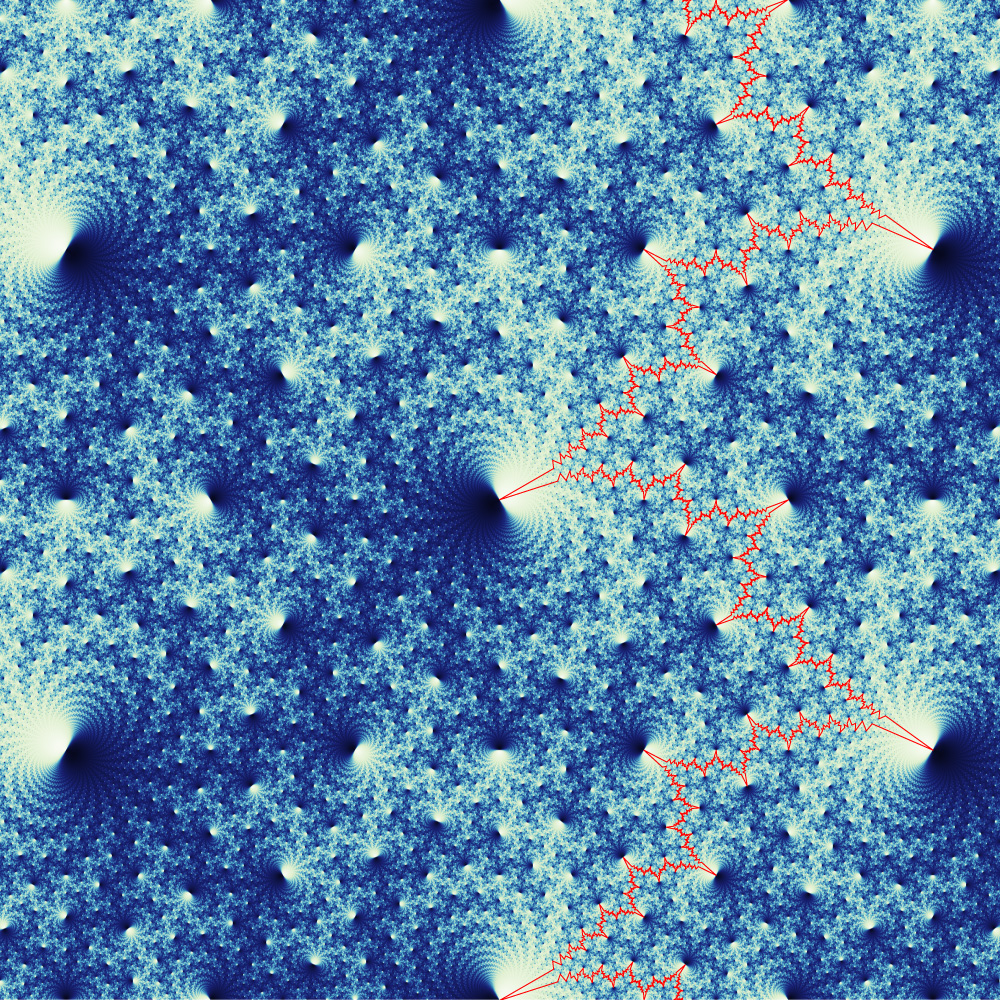}
\label{Fig:LightningMatchUp}
}
\subfloat[A black and white version of \reffig{CTPixelColour} highlighting only the brightest pixels.]{
\includegraphics[width = 0.47\textwidth]{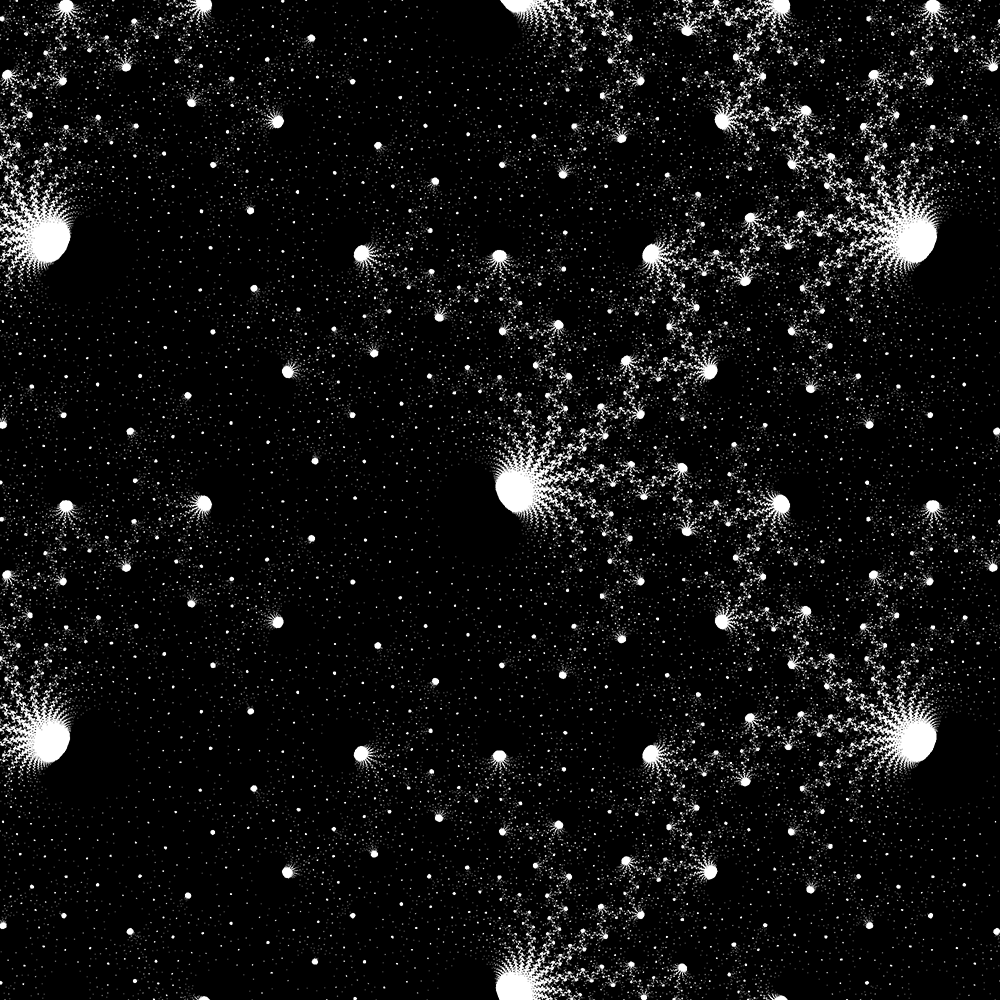}
\label{Fig:LightningThreshhold}
}
\caption{Matching up the lightning curve, reproduced from \cite[Figure~7]{DicksEtAl06}, with the cohomology fractal for \texttt{m004}.  }
\end{figure}

Suppose that $M$ is a cusped, fibred three-manifold, with fibre $F$.  
Dicks with various co-authors defines and studies the \emph{lightning curves}~\cite{AlperinEtAl99, DicksPorti02, DicksEtAl02, DicksEtAl06, DicksEtAl10, DicksEtAl12}; these are certain fractal arcs in the plane.  
In more detail; suppose that $c$ and $d$ are distinct cusps of an elevation $\cover{F}$ of $F$.  Let $[c, d]^\acw$ be the arc of $\bdy_\infty \cover{F}$ that is between $c$ and $d$ and anti-clockwise of $c$.  The Cannon--Thurston map $\Psi$ sends the arc $[c, d]^\acw$ to a union of disks in $\bdy_\infty \cover{M}$ meeting only along points.  The boundary of any one of these disks is a \emph{lightning curve}.

Since the lightning curve is defined in terms of the Cannon--Thurston map, it is not too surprising that we can also see something of the lightning curve in the cohomology fractal. 
In \reffig{LightningMatchUp} we show a segment of the lightning curve for the figure-eight knot complement generated by Cannon and Dicks~\cite[Figure~7]{DicksEtAl06} overlaid on \reffig{CTPixelColour}. 
The lightning curve seems to follow some of the brightest pixels in the cohomology fractal.  
\reffig{LightningThreshhold} is a black and white version of \reffig{CTPixelColour}, with a relatively high threshold set for a pixel to be white -- the lightning curve seems to be there, but this is nowhere near as clear as it was for the approximations to the Cannon--Thurston map. 
We do not fully understand the correspondence here.

We note that for clarity, \reffig{LightningMatchUp} shows only one segment of the lightning curve. 
There is another segment, symmetrical with the shown segment under a 180 degree rotation about the centre of the image. 
This second segment seems to follow the darkest pixels of the cohomology fractal.

\section{Implementation}
\label{Sec:Implement}

In this section we give an overview of an implementation of cohomology fractals.  
Our implementation is written in Javascript and GLSL; the code is available at \cite{github_cohomology_fractals}.
We have also made cohomology fractals available in SnapPy~\cite{snappy}. Note that SnapPy cannot find hyperbolic structures on finite triangulations.

We now follow \refdef{UsingTriangulation}. 
Suppose that $M$ is a connected, oriented, finite volume hyperbolic three-manifold. 
Let $\calT$ be a material or ideal triangulation of $M$.  
We are given a weighting $\omega \from \calT^{(2)} \to \RR$ for the faces of the two-skeleton. 

We represent the triangulation $\calT$ as a collection $\{t_i\}$ of model hyperbolic tetrahedra.  
Each tetrahedron $t_i$ has four faces $f^i_m$ lying in four geodesic planes $P^i_m$ in the hyperboloid model of $\HH^3$.  
Suppose that $t_j \in \calT$ is another model tetrahedron, with faces $f^j_n$.  
If the face $f^i_m$ is glued to $f^j_n$, then we have isometries $g^i_m$ and $g^j_n$ realising the gluings. 
Note that $g^i_m$ and $g^j_n$ are inverses. 

We are given a camera location $p$ in $M$; this is realised as a point (again called $p$) in some tetrahedron $t_i$. 
\begin{remark}
The reader familiar with computer graphics will note that we also require a frame at the camera location. 
To simplify the exposition, we will mostly suppress this detail.
\end{remark}
We are also given a radius $R$ as well as a maximum allowed step count $S$.  

\subsection{Ray-tracing}

For each pixel of the screen, we generate a corresponding unit tangent vector $u$ 
in the tangent space to the current tetrahedron $t_i$.
We then \emph{ray-trace} through $\calT$.  
That is, we travel along the geodesic starting at $p$, in the direction $u$, for distance $R$, taking at most $S$ steps.  \reffig{DesiredRayPath} shows a toy example, where we replace the three-dimensional hyperbolic triangulation $\calT$ of $M$ with a two-dimensional euclidean triangulation of the two-torus. 

\begin{figure}[htb]
\centering
\subfloat[The desired ray path, starting from the pair $(p, u)$ of length $R$.]{
\label{Fig:DesiredRayPath}
\labellist
\small\hair 2pt
\pinlabel $p$ [r] at 25 14
\pinlabel $u$ [b] at 32.5 18
\endlabellist
\includegraphics[height=1.67in]{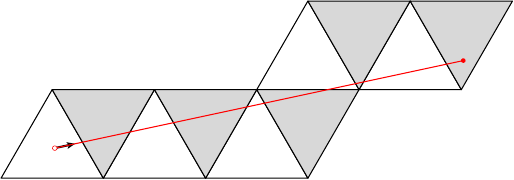}
}

\subfloat[The implementation of the ray path. The iterations of the loop are labelled with integers.]{
\label{Fig:ImplementationRayPath}
\labellist
\small\hair 2pt
\pinlabel 1 at 32 20
\pinlabel 3 at 20 27.5
\pinlabel 5 at 23.5 34
\pinlabel 9 at 13.5 10.5
\pinlabel 7 at 43 5.5

\pinlabel \rotatebox{300}{6} at 90 30
\pinlabel \rotatebox{300}{4} at 94 20
\pinlabel \rotatebox{300}{2} at 84 14
\pinlabel \rotatebox{300}{8} at 72 5
\pinlabel \rotatebox{300}{10} at 82 5
\endlabellist
\includegraphics[height=1.67in]{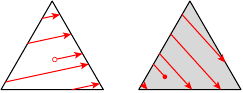}
}
\caption{A toy example of developing a ray through a tiling of a euclidean torus. 
Note that the geodesic segments passing through a tile are parallel; 
this is only because the geometry is euclidean.
In a hyperbolic tiling the segments are much less ordered. }
\end{figure}
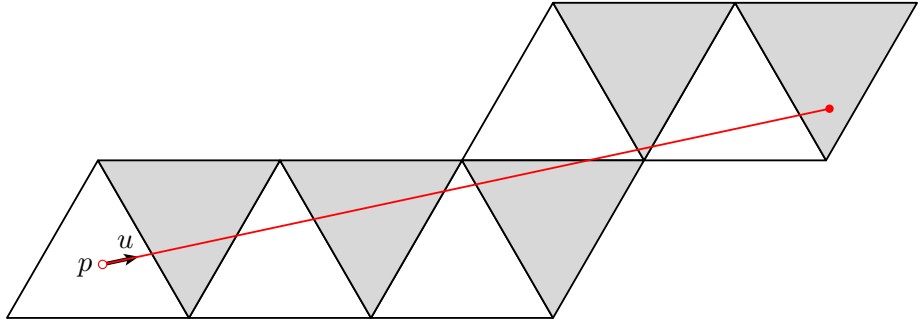
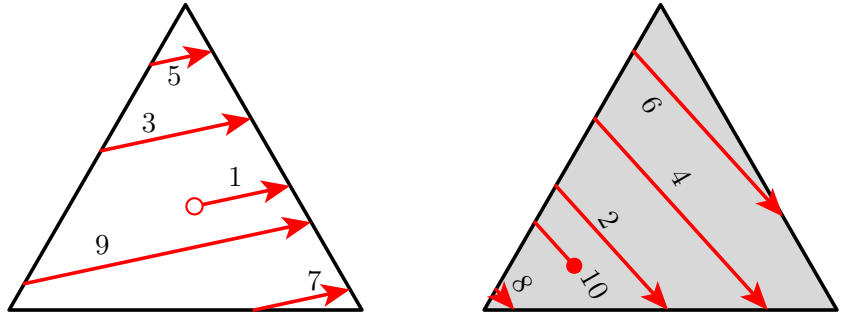

It is perhaps most natural to think of ray-tracing as occurring in $\cover{M}$, the universal cover of the manifold, as shown in \reffig{DesiredRayPath}.  
However, the na\"ive floating-point implementation in the hyperboloid model quickly loses precision.
We instead ray-trace in the manifold, as illustrated in \reffig{ImplementationRayPath}.  
Thus, all points we calculate lie within our fixed collection of ideal hyperbolic tetrahedra $\{t_i\}$.  

For each pixel, we do the following. 

\begin{enumerate}
\item 
The following initial data are given: an index $i$ of a tetrahedron, a point $p$ in $t_i$, and a tangent vector $u$ at $p$.  
Initialise the following variables. 
\begin{itemize}
\item
The total distance travelled: $r \gets 0$.  
\item
The number of steps taken: $s \gets 0$.  
\item
The current tetrahedron index: $j \gets i$. 
\item
The current position: $q \gets p$.
\item
The current tangent vector: $v \gets u$. 
\end{itemize}
\item 
\label{Itm:LoopStart}
Let $\gamma$ be the geodesic ray starting at $p$ in the direction of $u$. 
Find the index $n$ so that $\gamma$ exits $t_j$ through the face $f^j_n$. 
Let $t_k$ be the other tetrahedron glued to face $f^j_n$.
\item 
\label{Itm:LoopHitFace}
Calculate the position $q'$ and tangent vector $v'$ where $\gamma$ intersects $f^j_n$.  
Let $r'$ be the distance from $q$ to $q'$.
Set $r \gets r + r'$ and set $s \gets s + 1$.
\item 
If $r > R$ or $s > S$ then stop.  
\item 
Set $j \gets k$, set $q \gets g^j_n(q')$, and set $v \gets  Dg^j_n(v')$.
\item 
Go to step \refitm{LoopStart}. 
\end{enumerate}

This implements the ray-tracing part of the algorithm. 
In our toy example, this is shown in \reffig{ImplementationRayPath}. 

\subsection{Integrating}
To determine the colour of the pixel, we also track the total signed weight we accumulate along the ray. 
For this, we add the following steps to the loop above.

\begin{itemize}
\item[(1b)] 
An initial weight $w_0$ is given.  Initialise the following.
\begin{itemize}
\item[$\bullet$]
The current weight: $w_c \gets w_0$.
\end{itemize}
\item[(5b)] 
Let $f$ be the face between $t$ and $t'$, co-oriented towards $t'$. 
Set $w_c \gets w_c + \omega(f)$.
\end{itemize}

At the end of the loop, the value of $w_c$ gives the brightness of the current pixel. 
(In fact, we apply a function very similar to the arctangent function to remap the possible values of $w_c$ to a bounded interval. 
We then apply a gradient that passes through a number of different colours. 
This helps the eye see finer differences between values than a direct map to brightness.)

\subsection{Moving the camera}

In our applications, we enable the user to fly through the manifold $M$.  
Depending on the keys pressed by the user at each time step, we apply an isometry $g$ to $p$. 
We also track an orthonormal frame for the user; this determines how tangent vectors correspond to pixels of the screen.  
We also apply the isometry $g$ to this frame.  
When the user flies out of a face $f$ of the tetrahedron they are in, we apply the corresponding isometry $g^i_k$ to the position $p$ and the user's frame.
We also add $w(f)$ to the initial weight $w_0$.
Without this last step, the overall brightness of the image would change abruptly as the user flies through a face with non-zero weight.

\begin{remark}
\label{Rem:Basepoint}
With this last modification, the cohomology fractal depends on a choice of basepoint $b \in \cover{M}$. The point $p \in M$ must now also be replaced by $p \in \cover{M}$ (abusing notation, we use the same symbol for both points).  We add $b$ to the notation, and now write the cohomology fractal as 
\[
\Phi_R^{\omega, b, p} \from \UT{p}{\cover{M}} \to \RR  \qedhere
\]
\end{remark}

\begin{remark}
\label{Rem:DependenceOnb}
The dependence of the cohomology fractal on $b$ is minor: If we change $b$ to $b'$, then the value of $\Phi_R^{\omega, b, p}(v)$ changes by the weight we pick up along any path from $b'$ to $b$.
\end{remark}

\subsection{Material, ideal, and hyperideal views}
\label{Sec:Views}

The above discussion describes the \emph{material view}; the geodesic rays emanate radially from $p$. 
To render an image, we place a rectangle in the tangent space at $p$. 
For each pixel of the screen, we take the tangent vector $u$ to be the corresponding point of the rectangle.  
See \reffig{ViewsMaterial}.

\begin{definition}
The \emph{field of view} of a material image is the angle between the tangent vectors pointing at the midpoints of the left and right sides of the image. 
\end{definition}

\begin{remark}
The material view suffers from perspective distortion. This is most noticeable towards the edges of the image, and is worse when the field of view is large.
\end{remark}

To generalise the material view to the ideal and hyperideal, we introduce the following terminology.  
We say that a subset $D \subset \UT{}{\cover{M}}$ is a \emph{view} if it is one of the following.

\begin{enumerate}
\item In the \emph{material view}, $D$ is a fibre $\UT{p}{\cover{M}}\homeo S^2$. 

\item In the \emph{ideal view}, we take $D$ to be the collection of outward normals to a horosphere $H$. 
That is, the vectors point away from $\bdy_\infty H$. 
To render an image we place a rectangle in $D$. 
For each pixel of the screen we set the initial vector $v$ to be the corresponding point of the rectangle. 
The starting point is then $\pi(v) \in \cover{M}$, the basepoint of $v$.  
Finally, we set $w_0$ to be the total weight accumulated, along the arc from $b$ to $\pi(v)$, as we pass through faces of the triangulation.
See \reffig{ViewsIdeal}.

\item In the \emph{hyperideal view}, we take $D$ to be the collection of normals to a transversely oriented geodesic plane $P$. 
We draw $P$ on the euclidean rectangle of the screen using the Klein model. 
The algorithm is otherwise identical to the ideal view case. 
See \reffig{ViewsHyperideal}.
\end{enumerate}

\begin{remark}
The ideal view in hyperbolic geometry is the analogue of an orthogonal view in euclidean geometry.  
In both cases this is the limit of backing the camera away from the subject while simultaneously zooming in.  
\end{remark}

\begin{remark}
The hyperideal view suffers from an ``inverse'' form of perspective distortion. Towards the edges of the image, round circles look like ellipses, with the minor axis along the radial direction.
\end{remark}

\begin{definition}
\label{Def:View}
Let $D \subset \UT{}{\cover{M}}$ be a view, as discussed above.
In the notation for the cohomology fractal, we replace $p$ by $D$:
\[
\Phi_R^{\omega, b, D} \from D \to \RR \qedhere
\]
\end{definition}

\begin{figure}[htb]
\centering
\subfloat[Material.]{%
\includegraphics[width = 0.32\textwidth]{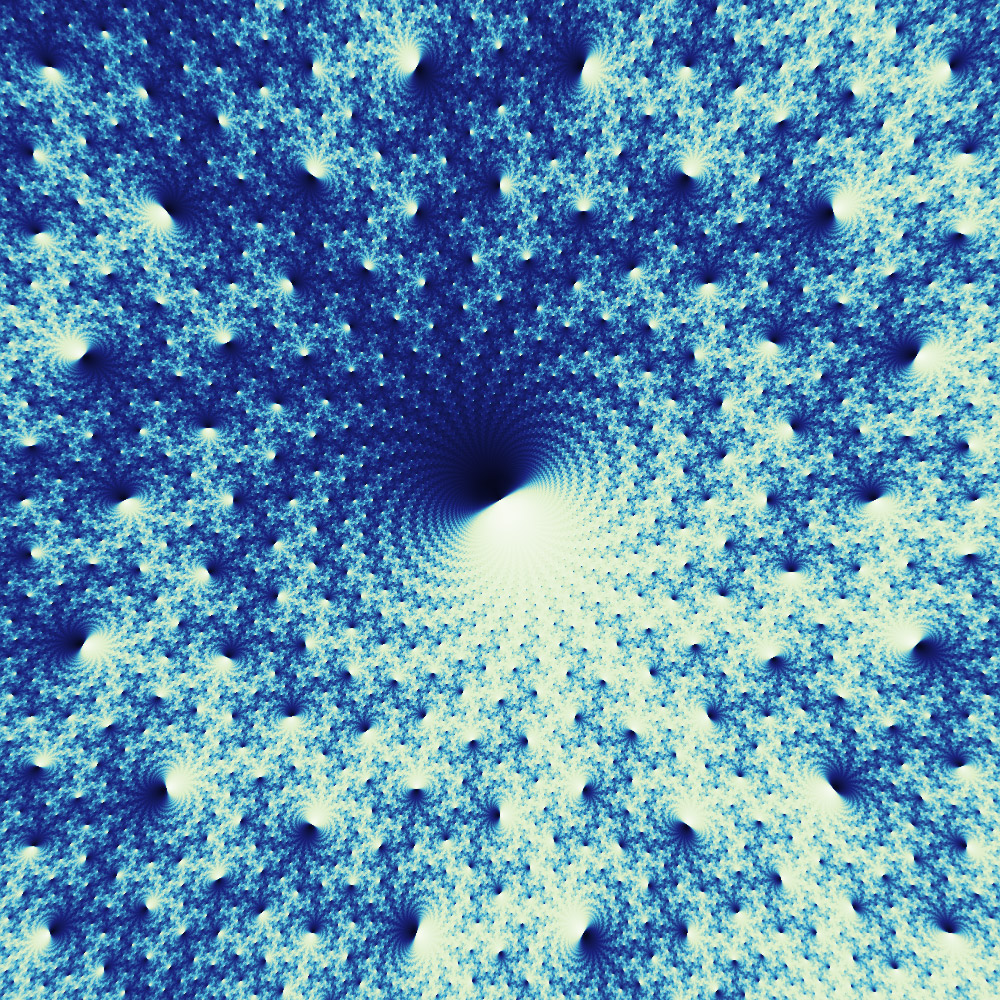}%
\label{Fig:ViewsMaterial}
}
\subfloat[Ideal.]{%
\includegraphics[width = 0.32\textwidth]{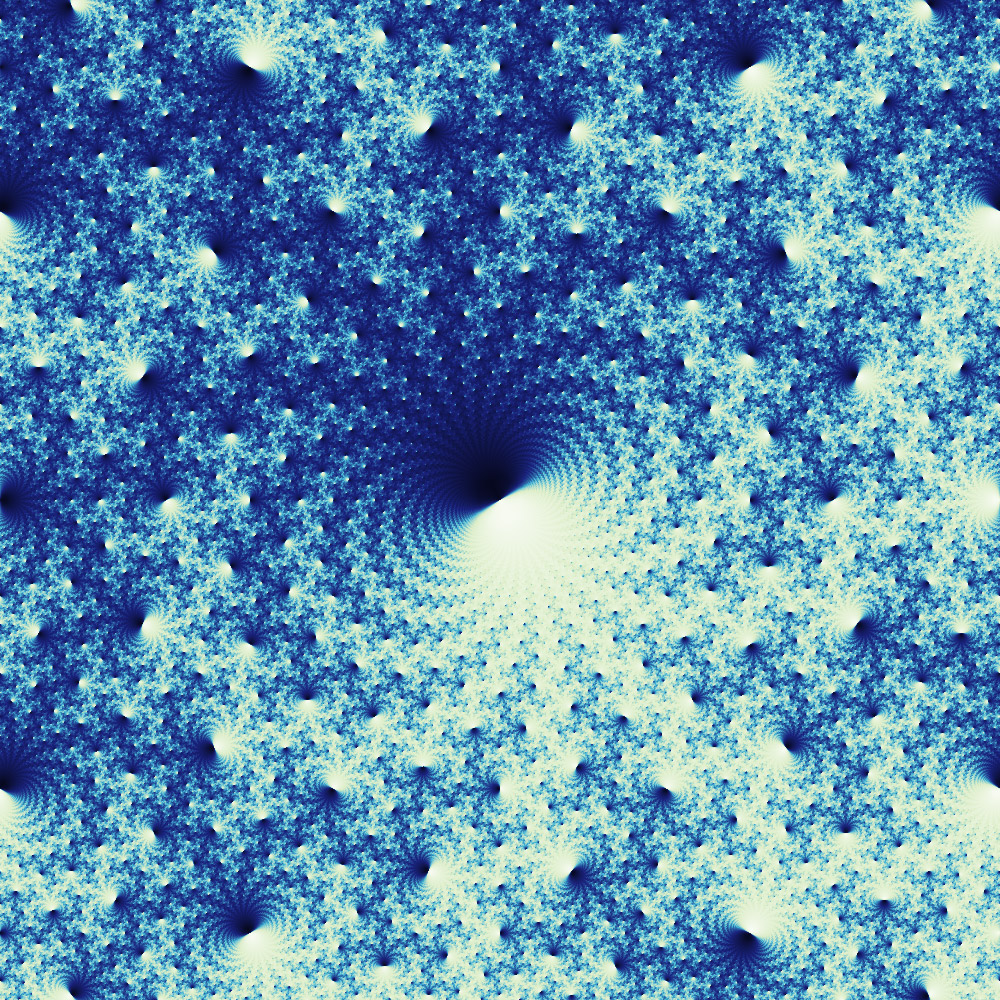}%
\label{Fig:ViewsIdeal}
}
\subfloat[Hyperideal.]{%
\includegraphics[width = 0.32\textwidth]{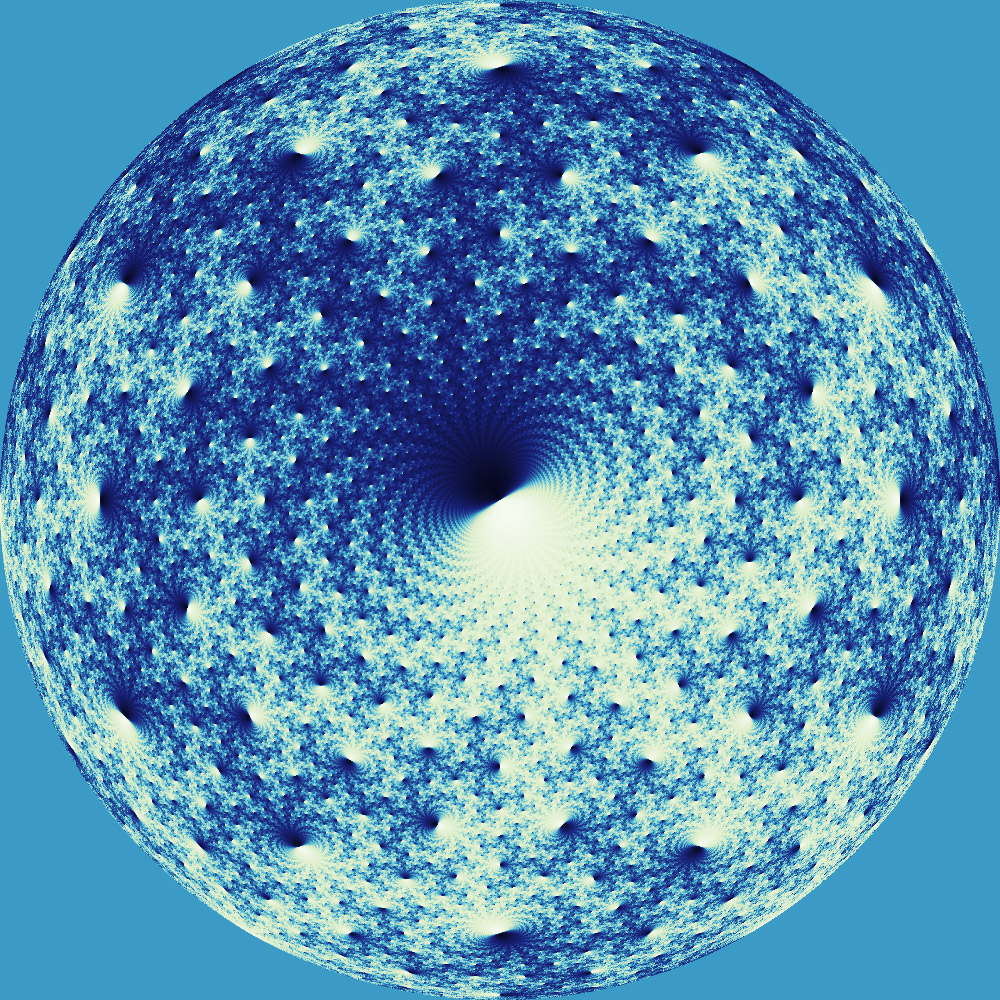}%
\label{Fig:ViewsHyperideal}%
}
\caption{Comparison between different views of the cohomology fractal for \texttt{m004}. }
\end{figure}

\subsection{Edges} 

We give the user the option to see the edges of the triangulation. The user selects an edge thickness $\varepsilon > 0$. The web application implements this in a lightweight fashion: 
In step \refitm{LoopHitFace}, if the distance from the point $q'$ to one of the three edges of the face we have intersected is less than $\varepsilon$, then we exit the loop early.
Depending on user choice, the pixel is either coloured by the weight $w_c$ or by the distance $d$. 
See \reffig{Edges}. 
In SnapPy, we compute the intersection of the ray with a cylinder about the edge in addition to the intersection with the faces.

\begin{figure}[htb]  
\centering
\subfloat[Coloured by cohomology fractal.]{
\includegraphics[width = 0.47\textwidth]{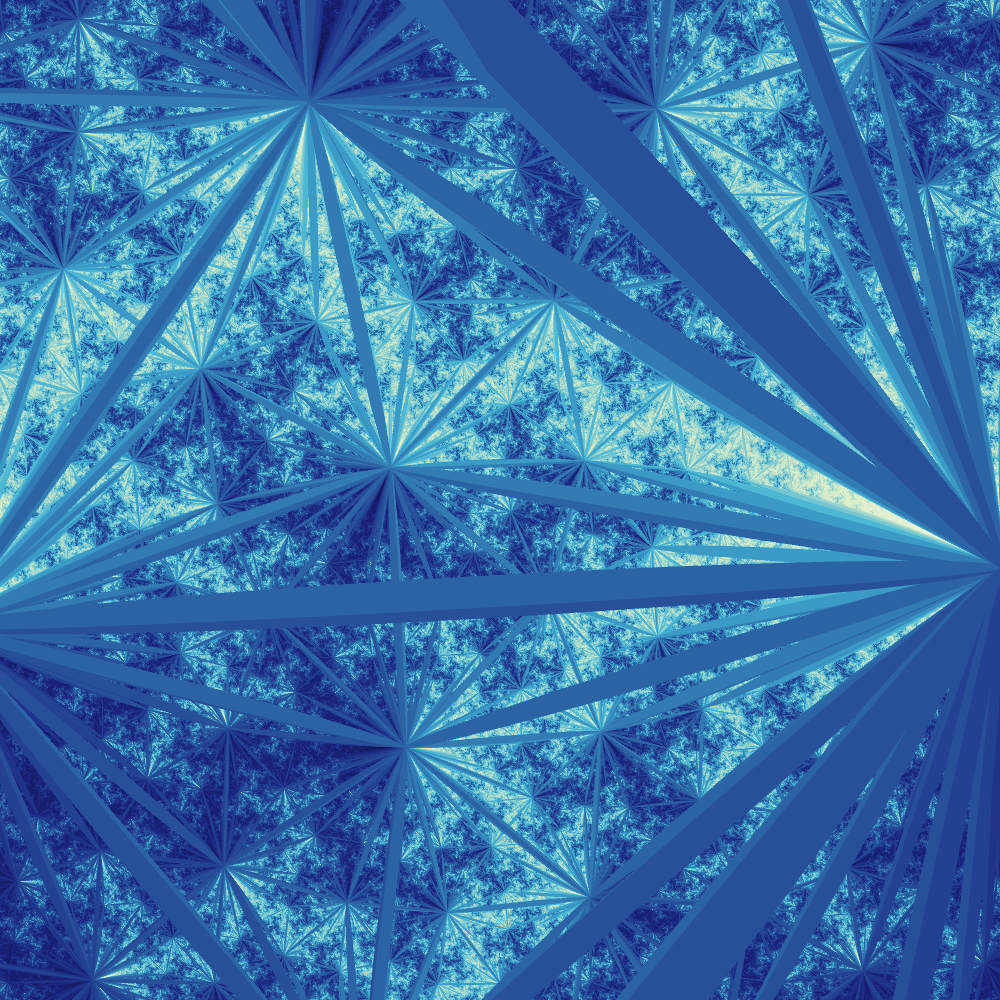}
}
\subfloat[Coloured by distance.]{
\includegraphics[width = 0.47\textwidth]{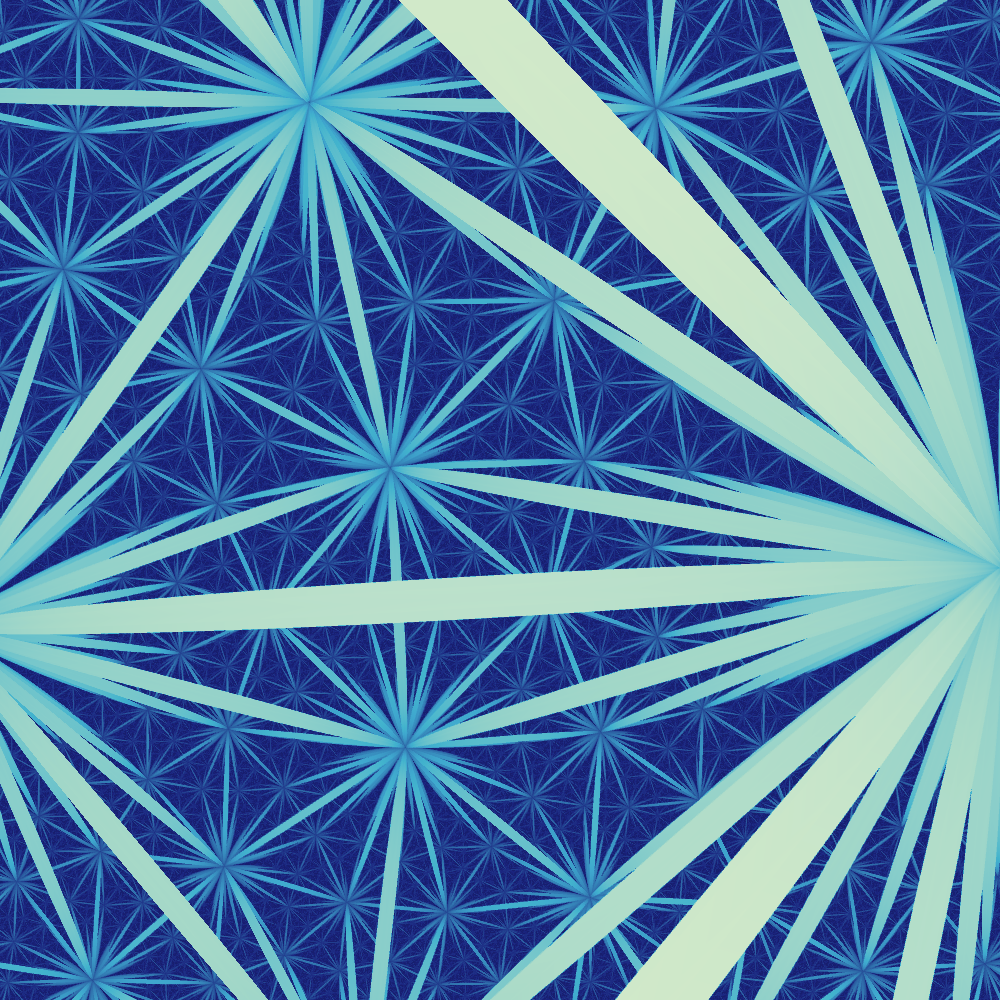}
}
\caption{Edges of the ideal triangulation of \texttt{m004}, as seen in the material view.}
\label{Fig:Edges}
\end{figure}

\subsection{Elevations}

We also give the user the option to see several elevations of the surface $F$. The user selects a weight $w_{\max} > 0$. 
In step (5b), if $w_0 < 0$ but $w_c>0$, then we have crossed the elevation at weight zero. In this case we exit the loop, and colour the pixel by the distance $d$. Similarly, if $w_0 > w_{\max}$ but $w_c < w_{\max}$, then we have crossed the elevation at weight $w_{\max}$, and again we stop and colour by distance. Finally, if $0<w_0<w_{\max}$, then we stop if $w_c$ has changed from $w_0$. \reffig{Elevation} shows a single elevation.


\subsection{Triangulations, geometry, and cocycles}

We obtain our triangulations and their hyperbolic shapes from the SnapPy census. We put some effort into choosing good representative cocycles; the choice here makes very little difference to the appearance of the cohomology fractal, but it makes a large difference to the appearance of the elevations. That is, a poor choice of cocycle gives a ``noisy'' elevation. For example, adding the boundary of a tetrahedron to the Poincar\'e dual surface may perform a one-three move to its triangulation. This adds unnecessary ``spikes'' to the elevations. 

When our manifold has Betti number one, there is only one cohomology class of interest. Here we searched for taut ideal structures dual to this class~\cite{Lackenby00}. When the SnapPy triangulation did not admit such a taut structure, we randomly searched for one that did. A taut structure gives a Poincar\'e dual surface with the minimum possible Euler characteristic. 

When the Betti number is larger than one, we used tnorm~\cite{tnorm20} to find initial simplicial representatives of vertices of the Thurston norm ball~\cite{Thurston86} in $H_2(M,\bdy M)$. We then greedily performed Pachner moves to reduce the complexity of the cocycles. We often, but not always, realised the minimum possible Euler characteristic.


\subsection{Discussion}

Any visualisation of a hyperbolic tiling suffers from the mismatch between the hyperbolic metric of the tiling and the euclidean metric of the image. 
The tools for generating more of the tiling involve applying hyperbolic isometries. 
The tiles thus shrink exponentially in size while growing exponentially in number.  
This makes it difficult for the tiles to cleanly approach $\bdy_\infty \HH^2$ or $\bdy_\infty \HH^3$.  
Approaching a ``parabolic'' point at infinity is even more difficult.

In the vector graphics approach, one must be careful to avoid wasting time generating huge numbers of invisible objects: tiles may be too small or their aspect ratios too large.

The ray-tracing approach (and any similar raster graphics approach) deals with this mismatch directly.  
Here we start with the pixel that is to be coloured and then generate only the hyperbolic geometry needed to determine its colour. 

A disadvantage of the ray-tracing approach is that we generate the hyperbolic geometry necessary for each pixel independently, meaning that much work is duplicated. However, the massive parallelism in modern graphics processing units mitigates, and is in fact designed to deal with, this kind of issue. It often turns out to be faster to duplicate work in many parallel processes rather than compute once then transmit the result to all processes requiring it.

\section{Incomplete structures and closed manifolds}
\label{Sec:Cone}

Suppose that $M$ is a cusped hyperbolic manifold.  Recall that we generate cohomology fractals for $M$ by using an ideal triangulation $\calT$.  Associated to $\calT$ there is the \emph{shape variety}; that is we impose the gluing equations outlined in \refsec{Geometry}, omitting the peripheral ones.  This gives us a space of deformations of the complete hyperbolic structure to incomplete hyperbolic structures; see~\cite[Section~4.4]{ThurstonNotes} and~\cite[Section~6.2]{PurcellKnotTheory}.   
If we deform correctly, we reach an incomplete structure whose completion has the structure of a hyperbolic manifold. The result is a \emph{hyperbolic Dehn filling} of the original cusped manifold. 

\subsection{Incomplete structures}

Suppose that $(M, \calT)$ is an ideally triangulated manifold. Let $Z^s$ be a path in the shape variety, where $Z^\infty$ is the complete structure and the completion of $Z^1$ is a closed hyperbolic three-manifold obtained by Dehn filling $M$. 
Between the two endpoints, we have \emph{incomplete structures} $M_s$ on the manifold $M$. 

In an incomplete geometry, there are geodesic segments that cannot be extended indefinitely. 
Suppose that, as in our algorithm, we only consider geodesic segments emanating from $p$ of length at most $R$. 
The endpoints of the rays that do not extend to distance $R$ form the \emph{incompleteness locus} $\Sigma_s$ in the ball $B^3_R \subset \HH^3$. 
It follows from work of Thurston that $\Sigma_s$ is a discrete collection of geodesic segments, for generic values of $s$~\cite{Thurston82}.

Suppose that $\omega \from \calT^{(2)} \to \RR$ is the given weight function dual to a properly embedded surface $F$ in $M$. 
We assume that the boundary of $F$ (if any) gives loops in the filled manifold that, there, bound disks. 
Thus $F$ also gives a cohomology fractal in the filled manifold. 

\begin{remark}
Note that there is no canonical way of transferring a base point $b$ or view $D$ between two different geometric structures $M_s$ and $M_{s'}$. However, we can choose $b$ and $D$ for each $M_s$ in a way that gives us continuously varying pictures.
We do not dwell on the details here.
\end{remark}

\begin{figure}[htb!]
\centering
\subfloat[$s=10$]{
\includegraphics[width=0.31\textwidth]{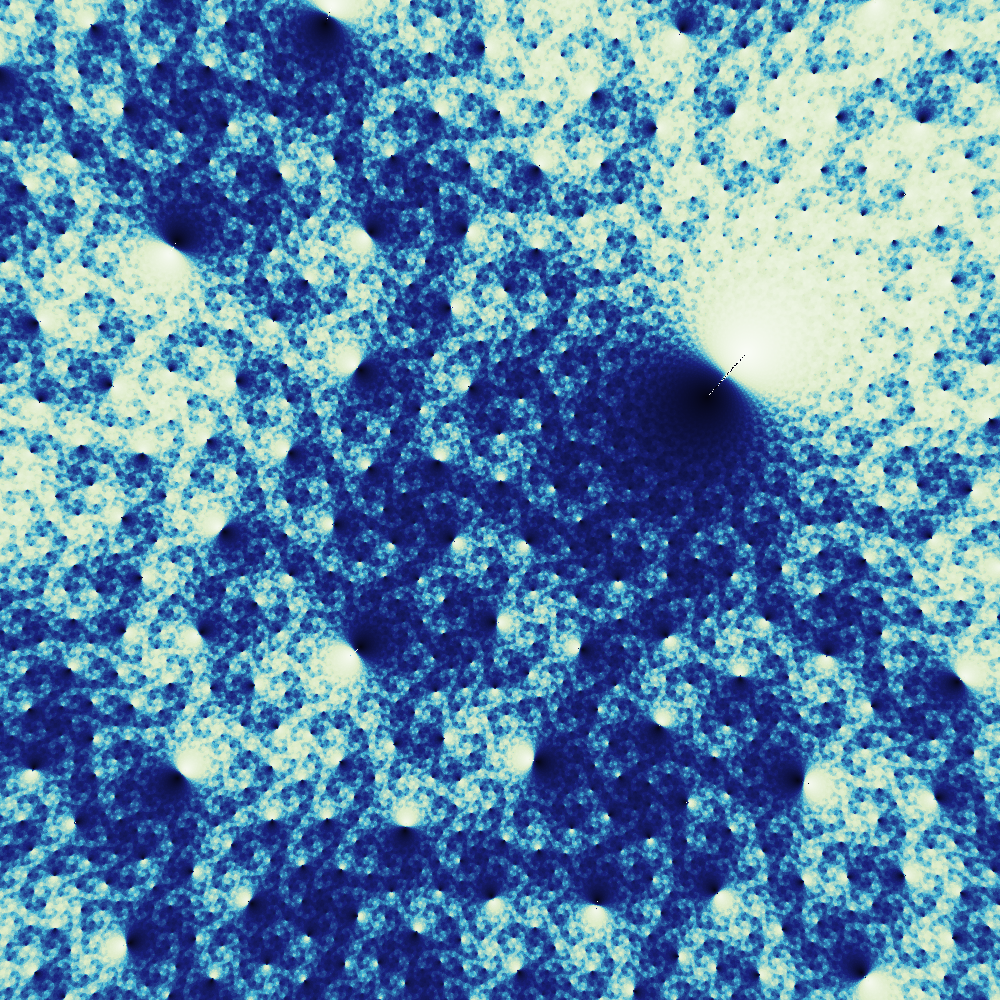}
}
\subfloat[$s=5$]{
\includegraphics[width=0.31\textwidth]{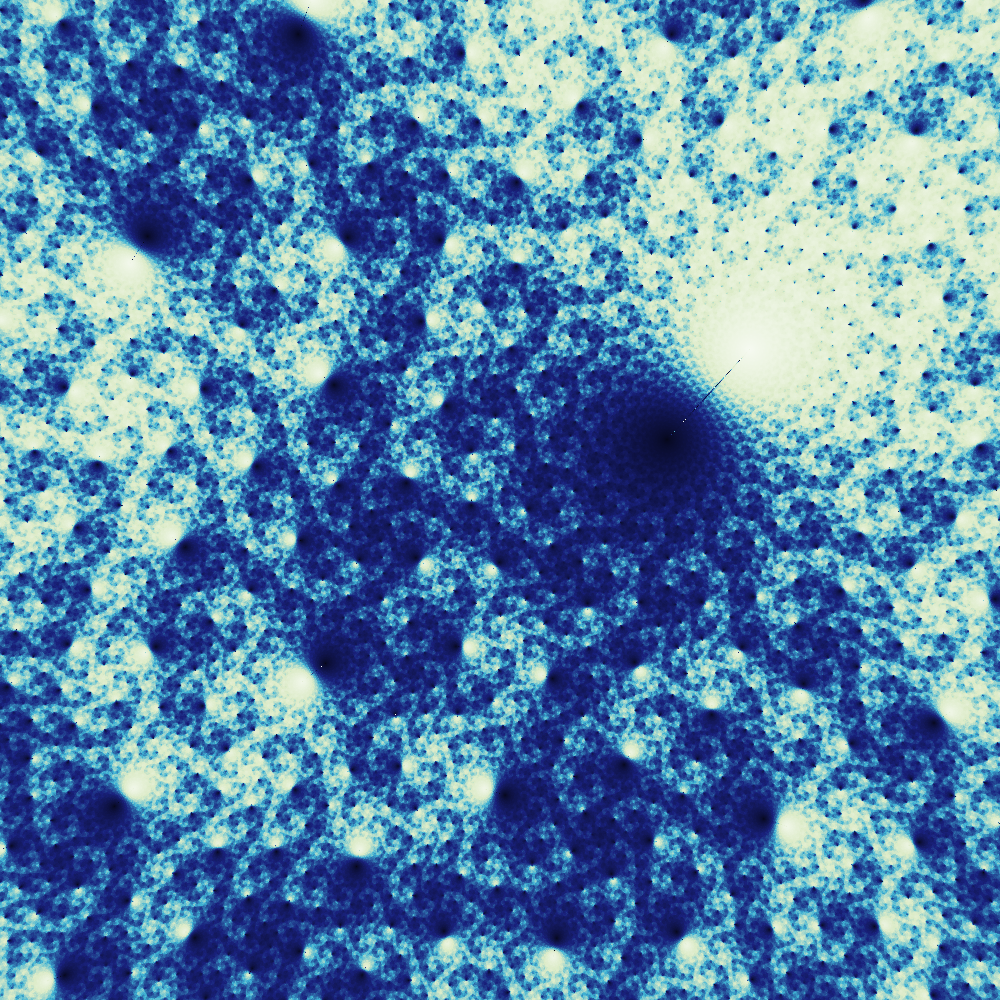}
}
\subfloat[$s=4$]{
\includegraphics[width=0.31\textwidth]{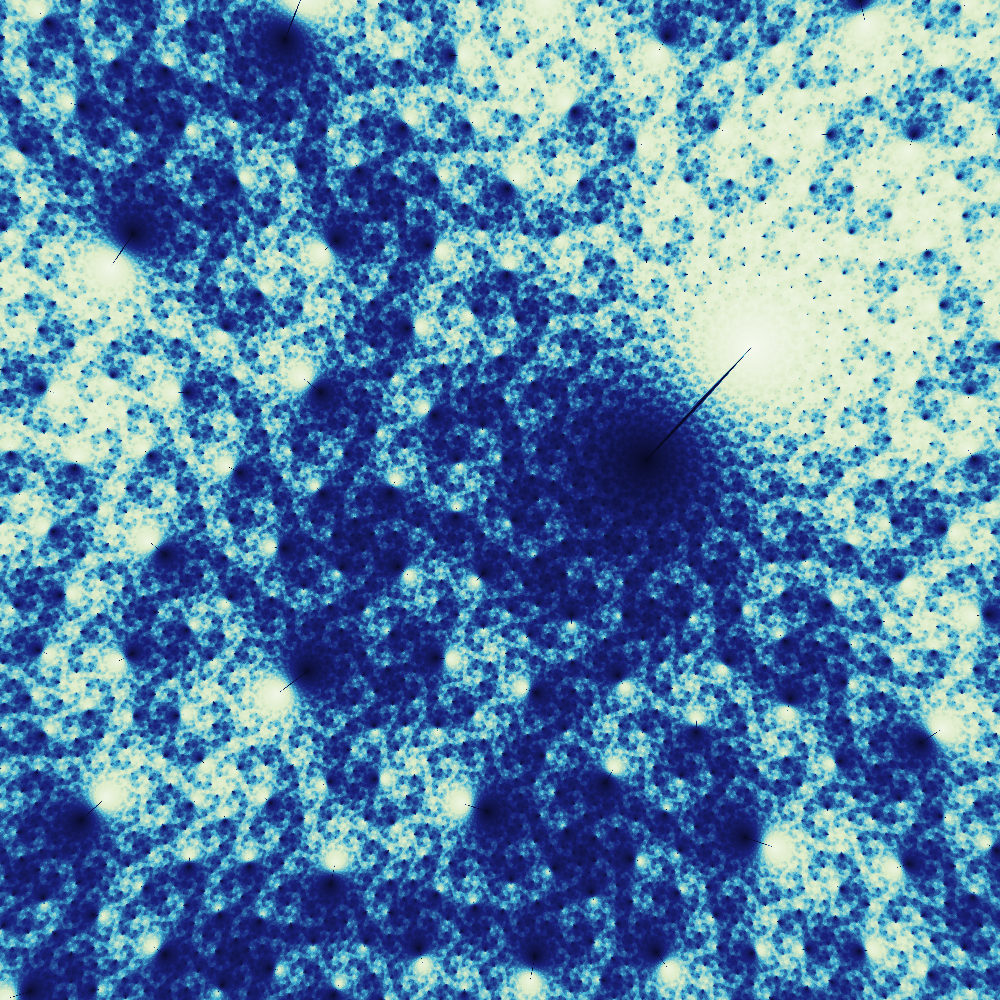}
}
\\
\subfloat[$s=3$]{
\includegraphics[width=0.31\textwidth]{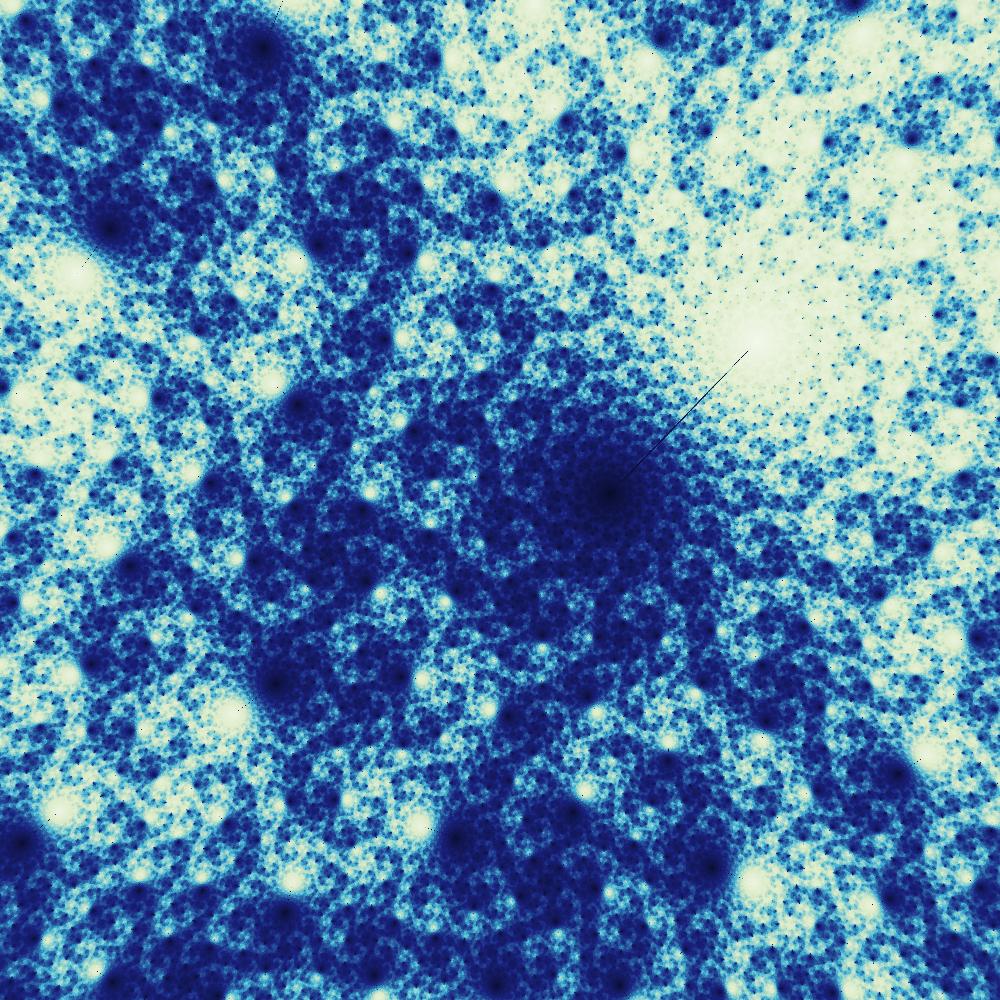}
}
\subfloat[$s=2$]{
\includegraphics[width=0.31\textwidth]{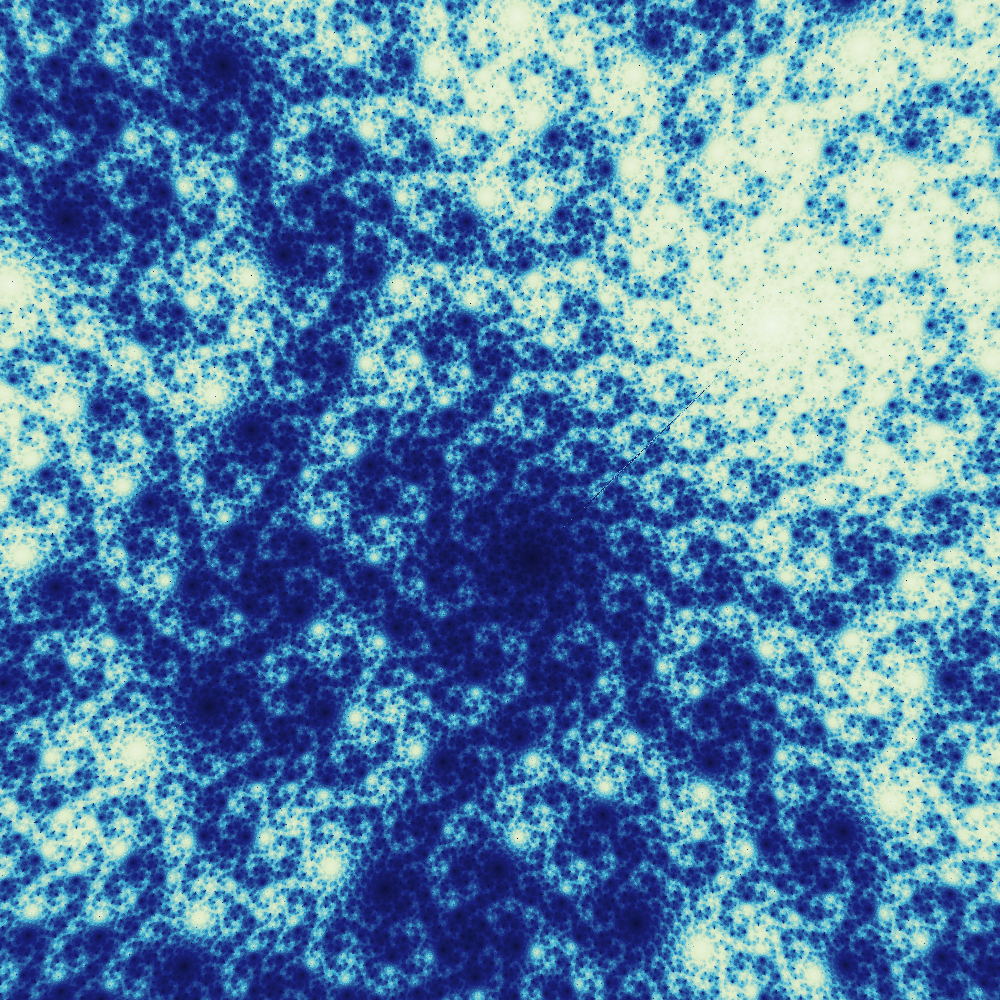}
}
\subfloat[$s=1.8$]{
\includegraphics[width=0.31\textwidth]{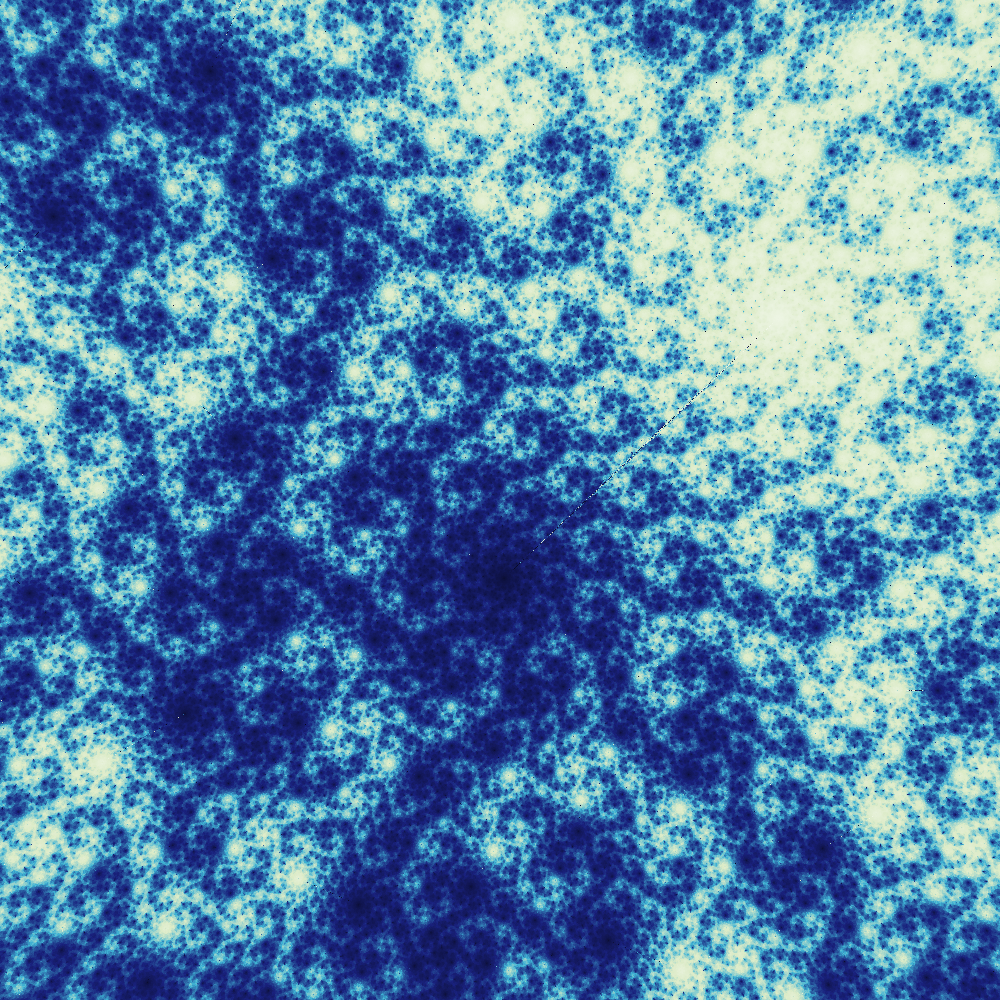}
}
\\
\subfloat[$s=1.6$]{
\includegraphics[width=0.31\textwidth]{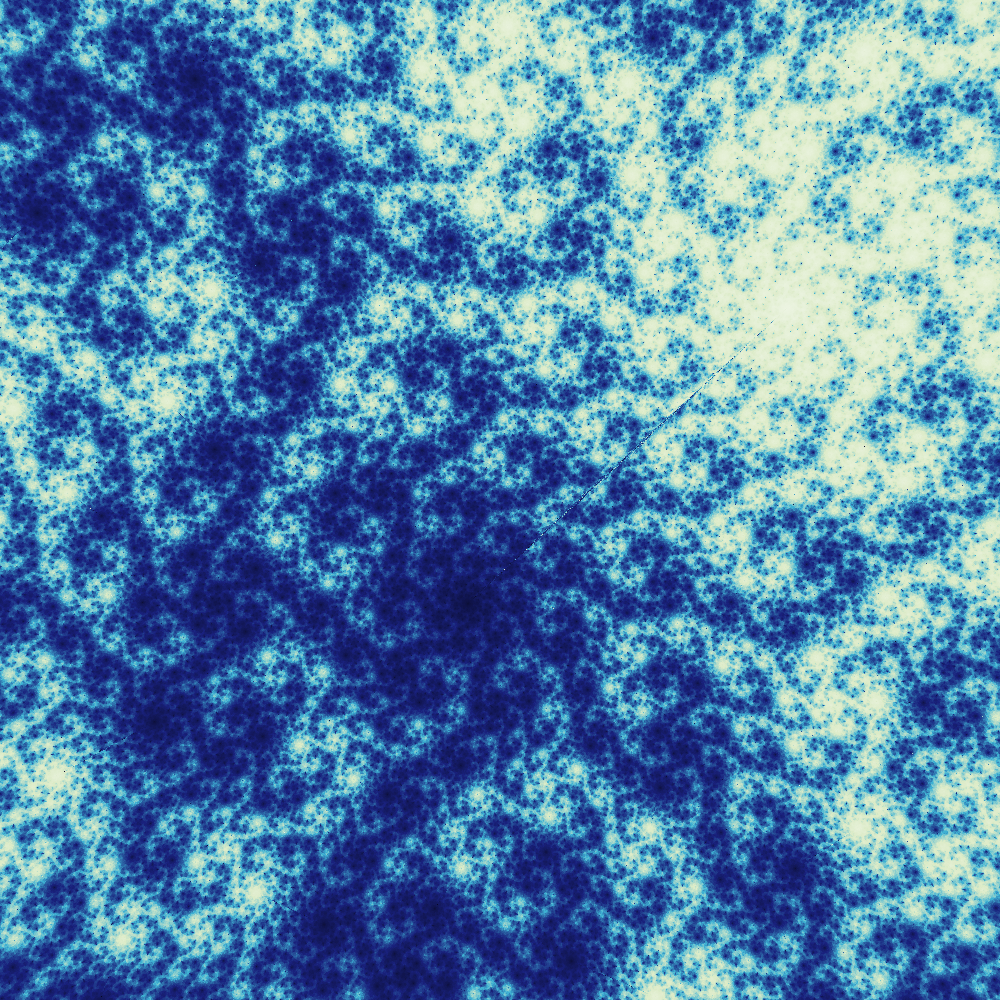}
}
\subfloat[$s=1.4$]{
\includegraphics[width=0.31\textwidth]{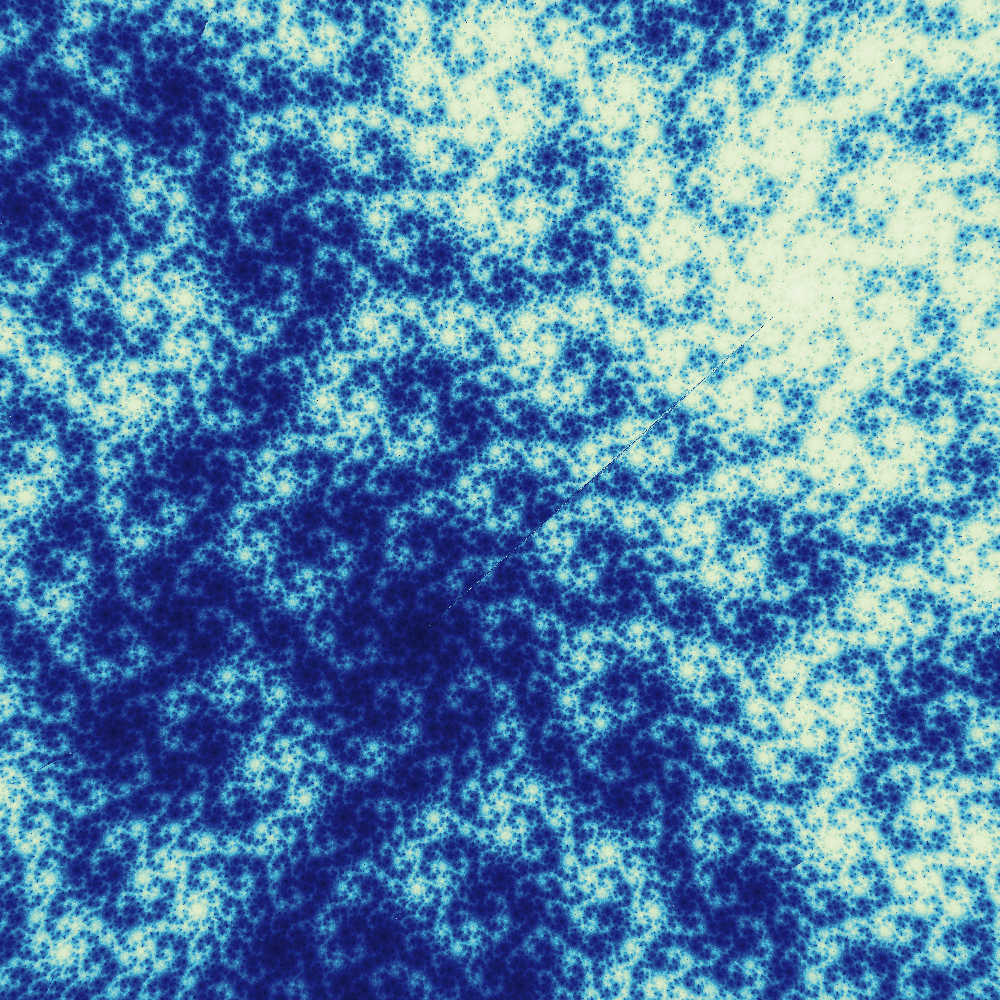}
}
\subfloat[$s=1.2$]{
\includegraphics[width=0.31\textwidth]{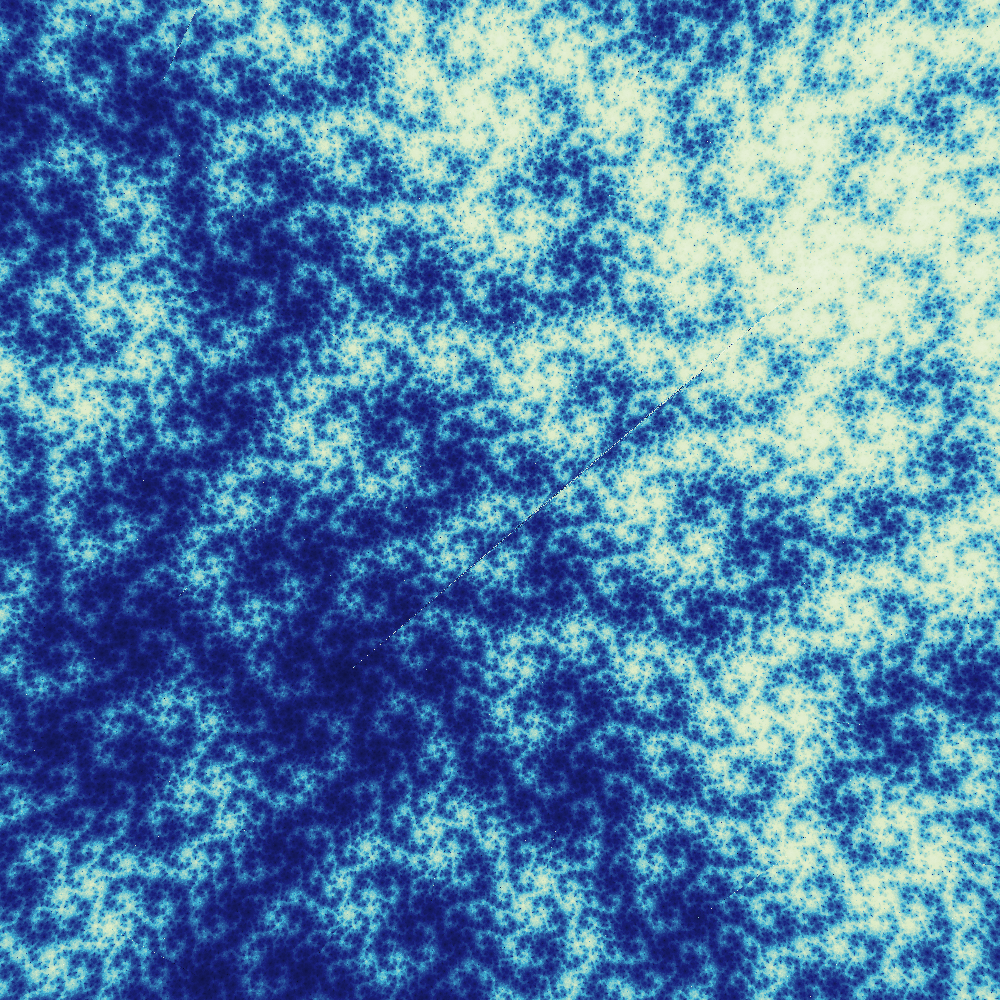}
}
\caption{Cohomology fractals for \texttt{m122($4s, -s$)} as $s$ varies.}
\label{Fig:Bending}
\end{figure}

\reffig{Bending} shows cohomology fractals for various $M_s$.
We see a kind of branch cut in the background to either side of the incompleteness locus $\Sigma_s$. 
As we vary $s$, the background appears to bend along the geodesic. Other paths in the shape variety will give shearing as well as (or instead of) bending.

When we reach a Dehn filling, the two sides again match, and we see the structure of the closed filled manifold. 
See \reffig{m122_4_-1}. 
(The two sides can also match before we reach the Dehn filling due to symmetries of the cusped manifold lining up with the cone structure.)

\begin{figure}[htb]
\centering
\includegraphics[width = 0.65\textwidth]{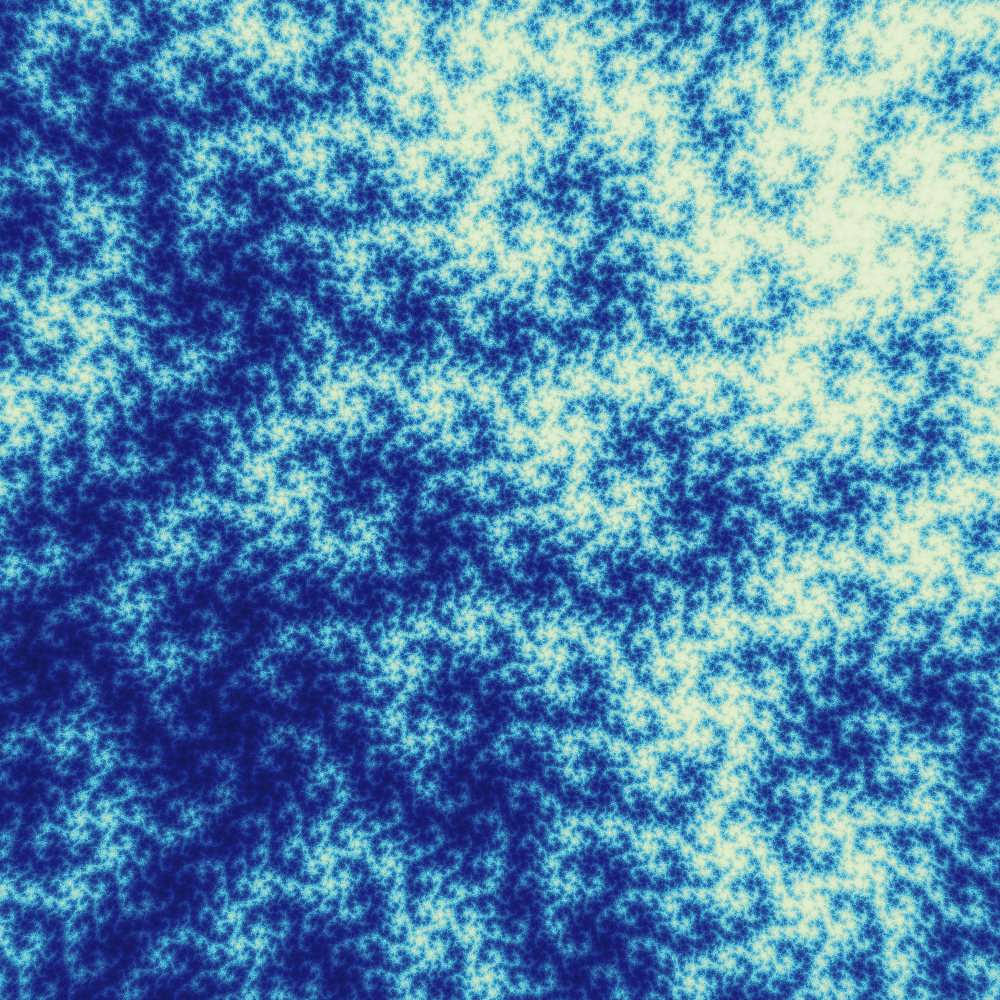}
\caption{Cohomology fractal for the Dehn filling \texttt{m122(4,-1)}. 
This gives a final image for \reffig{Bending}, with $s = 1$.}
\label{Fig:m122_4_-1}
\end{figure} 

\subsection{Numerical instability near the incompleteness locus}
\label{Sec:NumericalInstability}

Our algorithm, given in \refsec{Implement}, does not require completeness.  However, a ray from $p$ to $\Sigma_s$ necessarily meets infinitely many tetrahedra. 
This is because near $\Sigma_s$ we are far from the thick part of any tetrahedron, and the thin parts of the tetrahedra are almost ``parallel'' to $\Sigma_s$.  
Thus the innermost loop of the algorithm will always halt by reaching the maximum step count; it follows that we cannot ``see through'' a neighbourhood of $\Sigma_s$.  \reffig{m122_4_-1_with_evil} shows the cohomology fractal drawn with a small maximum step count, making such a neighbourhood visible.

\begin{figure}[htb]
\centering
\includegraphics[width = 0.65\textwidth]{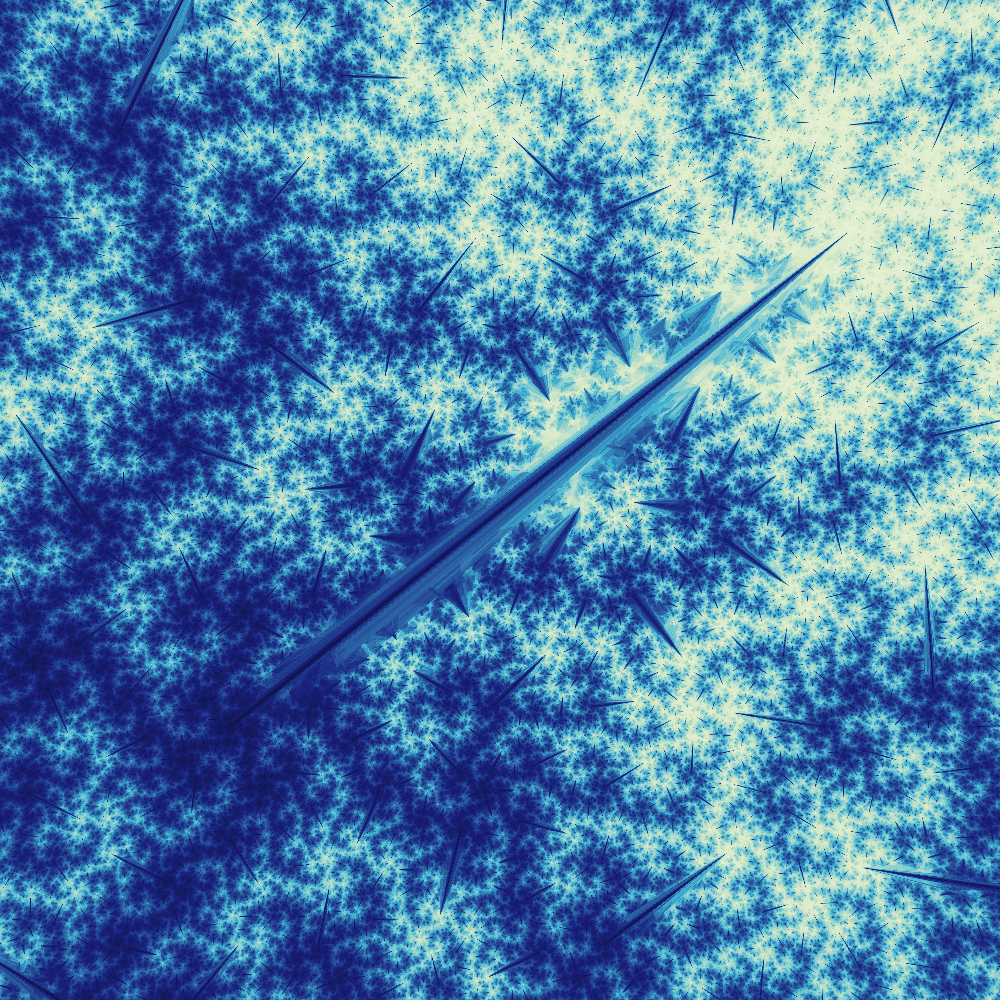}
\caption{Cohomology fractal for the Dehn filling \texttt{m122(4,-1)} drawn with an incomplete structure on an ideal triangulation. 
Here the maximum number of steps $S$ is 55. Compare with \reffig{m122_4_-1}.}
\label{Fig:m122_4_-1_with_evil}
\end{figure} 

Increasing the maximum step count shrinks the opaque neighbourhood of $\Sigma_s$. However, as a ray approaches $\Sigma_s$, its segments within the model tetrahedra tend to their ideal vertices. Thus the coordinates blow up; this appears to lead to numerically unstable behaviour. See \reffig{EvilCloseUp}. In the next section we describe a method to eliminate these numerical defects; we use this to produce \reffig{WithoutEvilCloseUp}.

\begin{figure}[htb]
\centering
\subfloat[Numerical instability near the incompleteness locus.]{
\includegraphics[width = 0.47\textwidth]{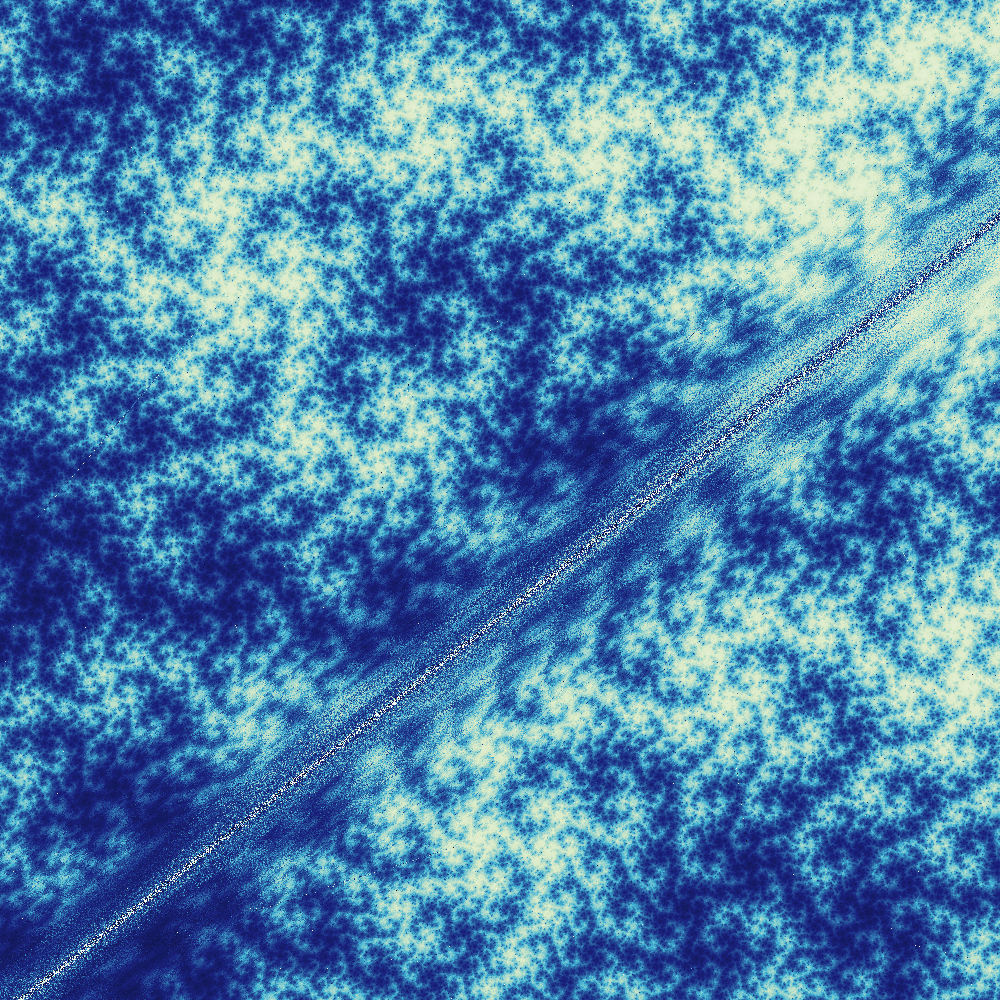}
\label{Fig:EvilCloseUp}
}
\subfloat[The same view with a material triangulation implementation.]{
\includegraphics[width = 0.47\textwidth]{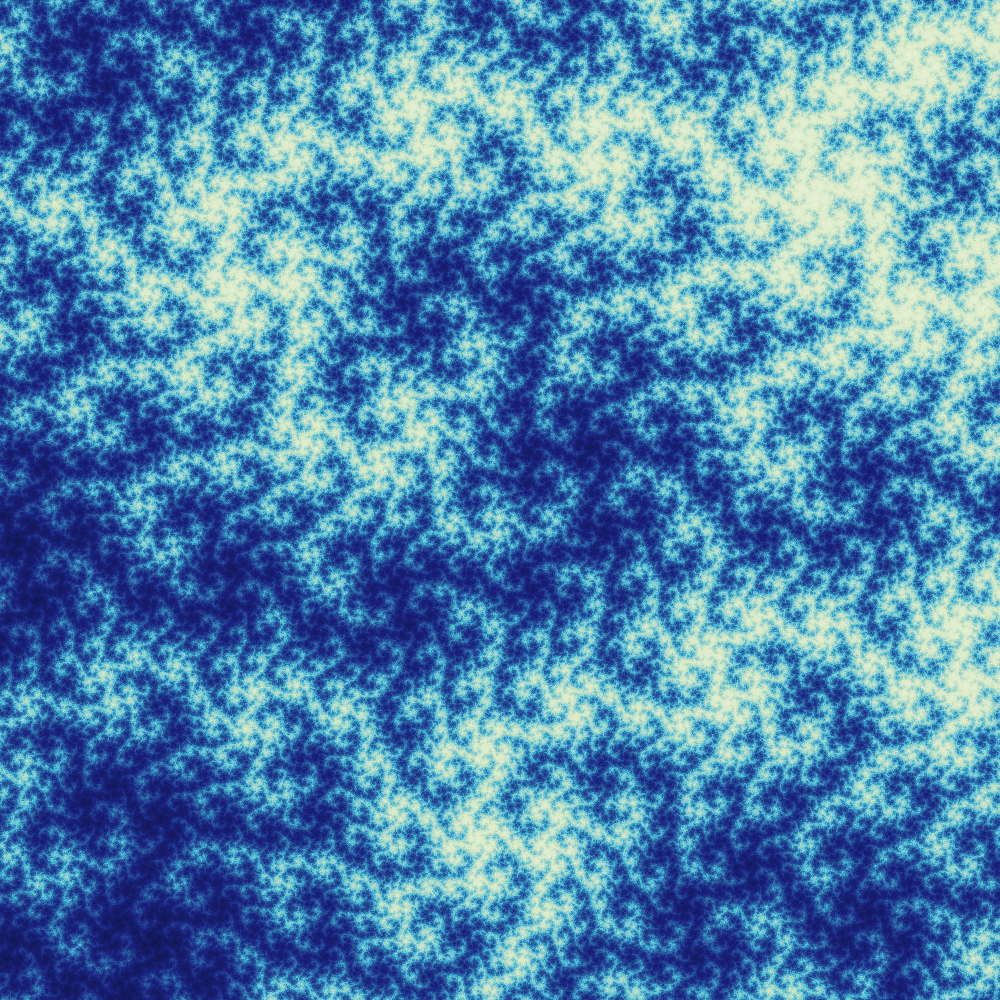}
\label{Fig:WithoutEvilCloseUp}
}
\caption{A view of the cohomology fractal for the manifold \texttt{m122(4,-1)} near the incompleteness locus.
On the left we have taken the maximum number of steps $S$ sufficiently large to ensure that all rays reach distance $R$.}
\end{figure} 

Note that numerical instability caused by rays approaching the ideal vertices also occurs for the complete structure on a cusped manifold. It is less noticeable in this case however, because these errors occur in a small part of the visual sphere for typical positions. 

\subsection{Material triangulations}
\label{Sec:MaterialTriangulations}

In order to remove instability around the incompleteness locus, we remove it. That is, we abandon (spun) ideal triangulations in favour of material triangulations. There is no change to the algorithm in \refsec{Implement}; we only alter the input data (the planes $P^i_k$ and face-pairing matrices $g^i_k$):

Given the edge lengths (see \refsec{MaterialGeometry}) for a material triangulation, \cite[Lemma~3.4]{matthiasVerifyingFinite} assigns hyperbolic isometries to the edges of a doubly truncated simplex (also known as permutahedron). 
These can be used to switch a tetrahedron between different standard positions (as defined in \cite[Definition~3.2]{matthiasVerifyingFinite}) where one of its faces is in the $\HH^2\subset\HH^3$ plane. 
We assume that every tetrahedron is in $(0,1,2,3)$--standard position. 
Given a face-pairing, we apply the respective isometries to each of the two tetrahedra such that the faces in question line up in the $\HH^2\subset\HH^3$ plane. 
The face-pairing matrix $g^i_k$ is now given by composing the inverse of the first isometry with the second isometry. 
For example, let face 3 of one tetrahedron be paired with face 2 of another tetrahedron via the permutation $(0,1,2,3)\mapsto(0,1,3,2)$. 
To line up the faces, we need to bring the second tetrahedron from the default $(0,1,2,3)$--standard position into $(0,1,3,2)$--standard position by applying $\gamma_{012}$ from \cite[Lemma~3.4]{matthiasVerifyingFinite} which will thus be the face-pairing matrix, see \cite[Figure~4]{matthiasVerifyingFinite}. 
It is left to compute the planes $P^i_k$. 
Note that $P^i_3$ (for each $i$) is the canonical copy of $\HH^2 \subset \HH^3$. 
All other $P^i_k$ can be obtained by applying the isometries from \cite[Lemma~3.4]{matthiasVerifyingFinite} again.

\subsection{Cannon--Thurston maps in the closed case}

\begin{figure}[htb]
\centering
\subfloat[McMullen's illustration~\cite{McMullen19}. See also~\cite{McMullenWeb}.]{
\includegraphics[width = 0.47\textwidth]{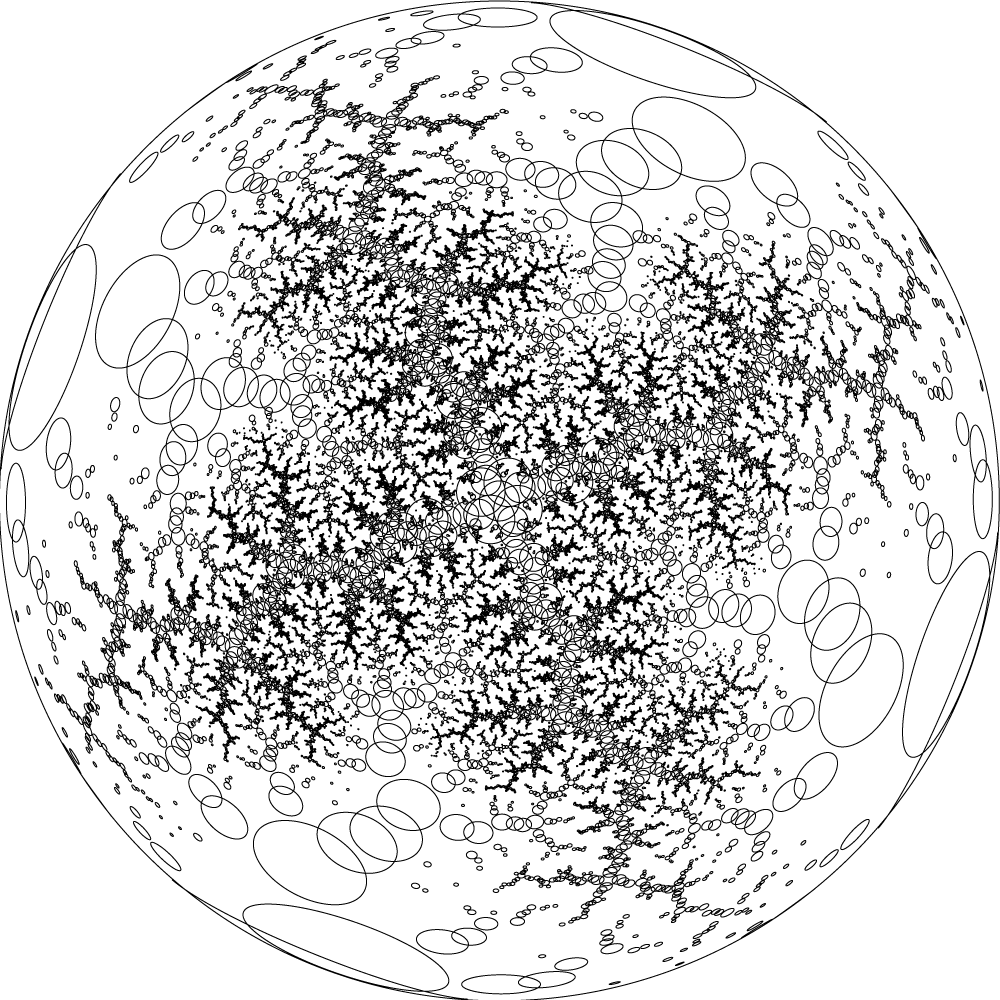}
\label{Fig:McMullen}
}
\subfloat[Cohomology fractal - hyperideal view.]{
\includegraphics[width = 0.47\textwidth]{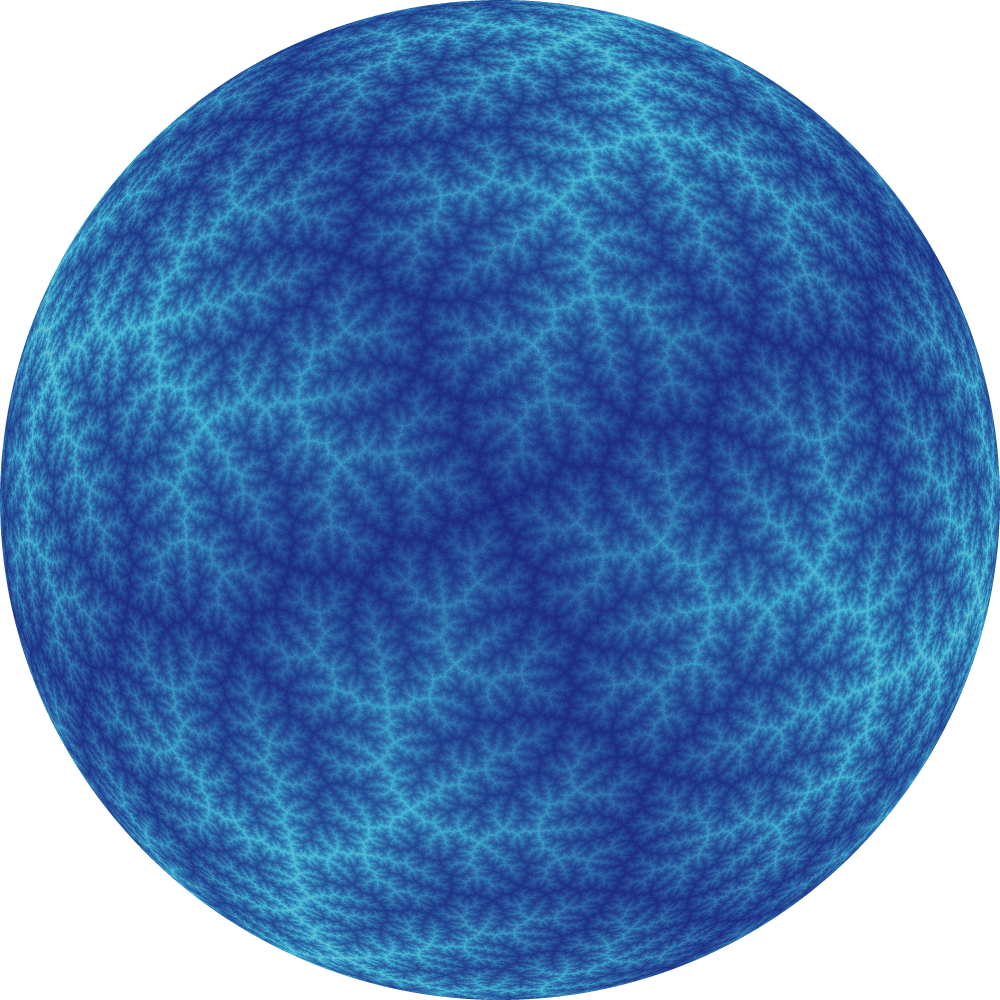}
\label{Fig:OrbifoldCohomFractal}
}
\caption{Views of \texttt{m004(0,2)}.}
\label{Fig:Orbifold}
\end{figure}

Cannon and Thurston's original proof was in the closed case.  Thurston's original images and all subsequent renderings, with one notable exception, are in the cusped case. 
With some minor modifications, \refprop{LightDark} applies in the closed case; thus the cohomology fractals again approximate Cannon--Thurston maps. 

We are aware of only one previous example in the closed case, due to McMullen~\cite{McMullenWeb}. 
In \reffig{Orbifold} we give a rasterisation of his original vector graphics image~\cite{McMullen19}, and our version of the same view.
The filling \texttt{m004(0,2)} of the figure-eight knot complement has an incomplete hyperbolic metric.
The completion is a hyperbolic orbifold $\calO$ with angle $\pi$ about the orbifold locus; the universal cover is $\HH^3$. 

Since the filling is a multiple of the longitude, the orbifold $\calO$ is again fibred. An elevation of this fibre to $\HH^3$ gives a Cannon--Thurston map. Our image, \reffig{OrbifoldCohomFractal} is the cohomology fractal for the fibre in $\calO$, in the hyperideal view. This is implemented using a material triangulation of an eight-fold cover $M$. Since $M$ with its fibre, is commensurable with $\calO$ with its fibre, we obtain the same image.

McMullen's image, reproduced in \reffig{McMullen} was generated using his program \texttt{lim}~\cite{lim}. Briefly, let $\calO^\infty$ be the infinite cyclic cover of $\calO$. McMullen produces a sequence $\calO^n$ of quasi-fuchsian orbifolds that converge in the geometric topology to $\calO^\infty$. In each of these the convex core boundary is a pleated surface. The supporting planes of this pleated surfaces give round circles in $\bdy \HH^3$. His image then is obtained by taking $n$ fairly large, passing to the universal cover of $\calO^n$, and drawing the boundaries of many supporting planes~\cite{McMullen19}.

\subsection{Accumulation of floating point errors}
\label{Sec:Accumulate}

Our implementation uses single-precision floating point numbers. As we saw in \refsec{NumericalInstability}, this can cause problems when rays approach the vertices of ideal tetrahedra. However, floating point errors can accumulate for large values of $R$ whether or not rays approach the vertices. This can therefore also affect material triangulations.

With these problems in mind, we cannot claim that our images are rigorously correct.
However, for small values of $R$ we can be confident that our images are accurate. For very small values the endpoints of our rays all sit within the same tetrahedron, and so all pixels are the same colour. As we increase $R$ (as in \reffig{VisSphereRadii}), we see regions of constant colour, separated by arcs of circles. This is provably correct: (horo-)spheres meet the totally geodesic faces of tetrahedra in circles. 

If we zoom in whilst increasing $R$, eventually floating point errors become visible.
\reffig{NumNoise} shows the results of an experiment to determine when this happens, for a material triangulation. At around $R = 11$, the circular arcs separating regions of the same colour become stippled. At around $R = 13$, the regions are no longer distinct. 

\begin{remark}
\label{Rem:Quasigeodesic}
Perhaps surprisingly, this accumulation of error does not mean that our pictures are inaccurate. Suppose that the side lengths of our pixels are on a somewhat larger scale than the precision of our floating point numbers. For each pixel, our implementation produces a piecewise geodesic, starting in the direction through the centre of the pixel, but with small angle defect at each vertex. 
Due to the nature of hyperbolic geometry, this piecewise geodesic cannot curve away from the true geodesic fast enough to leave the visual cone on the desired pixel. Thus, as long as the pixel size is not too small, each pixel is coloured according to some sample within that pixel.
\end{remark}

\begin{figure}[htb]
\centering
\subfloat[$R=10.5,$ field of view $\sim 0.015^\circ$]{
\includegraphics[width=0.45\textwidth]{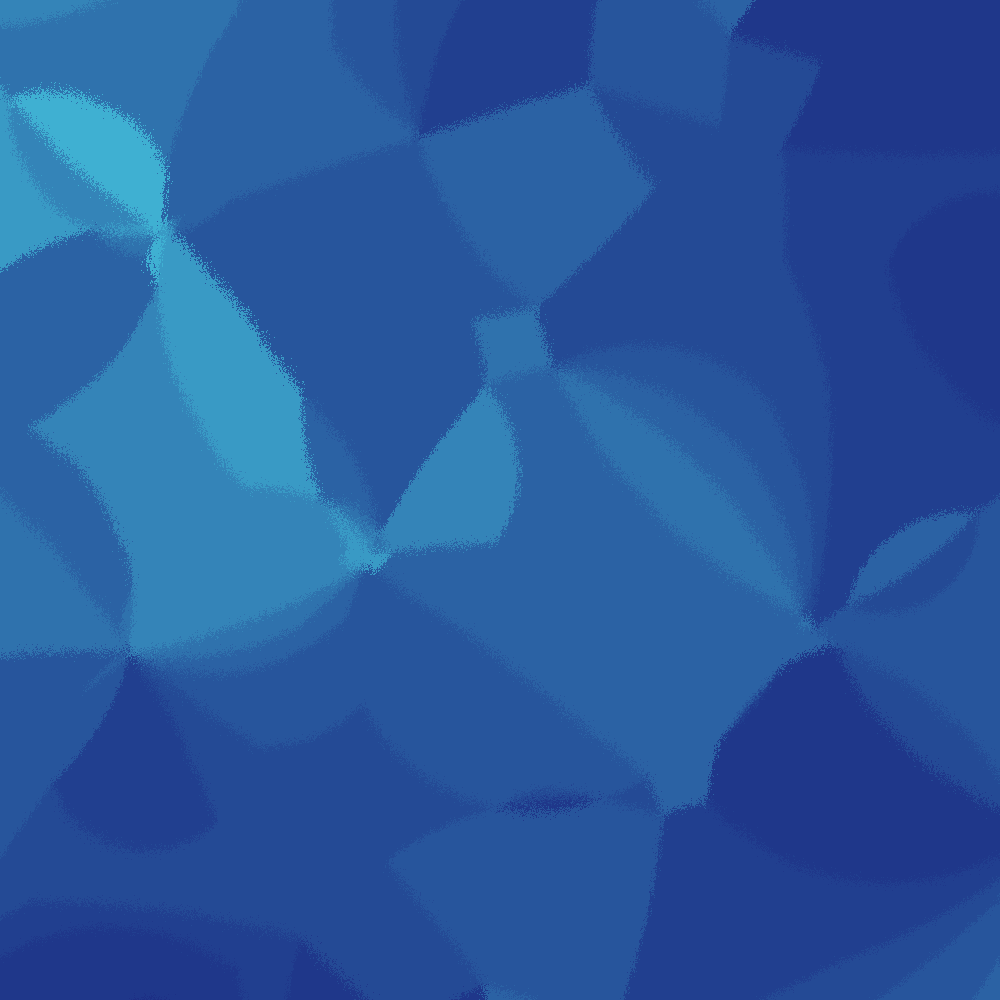}
}
\thinspace
\subfloat[$R=11.5,$ field of view $\sim 0.005^\circ$]{
\includegraphics[width=0.45\textwidth]{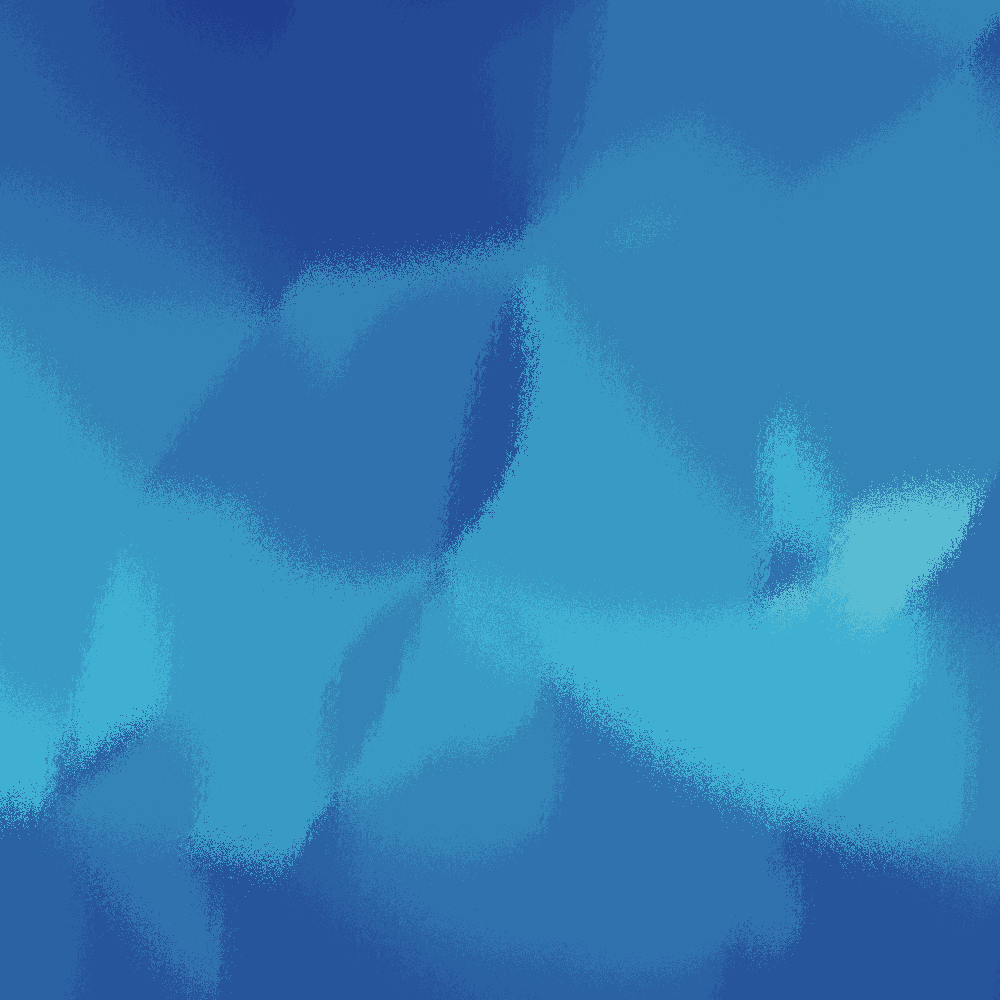}
}

\subfloat[$R=12.5,$ field of view $\sim 0.002^\circ$]{
\includegraphics[width=0.45\textwidth]{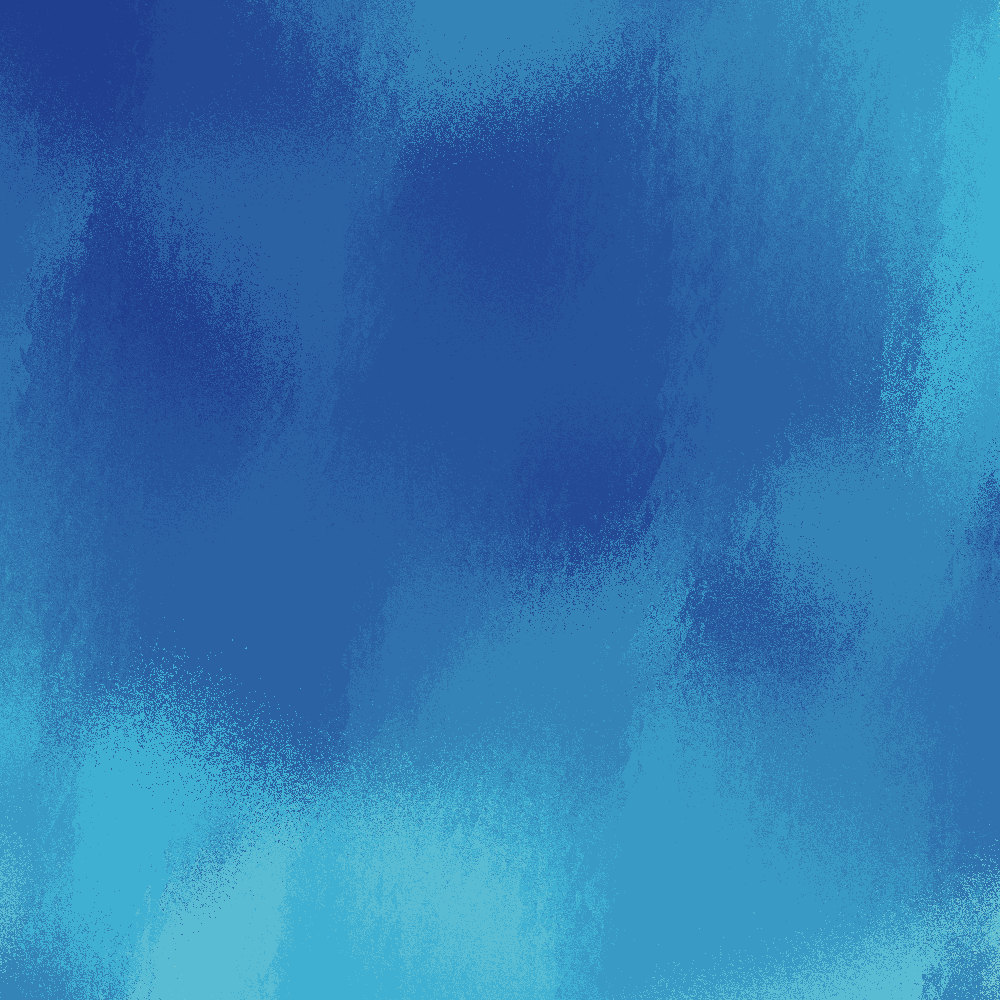}
}
\thinspace
\subfloat[$R=13.5,$ field of view $\sim 0.0007^\circ$]{
\includegraphics[width=0.45\textwidth]{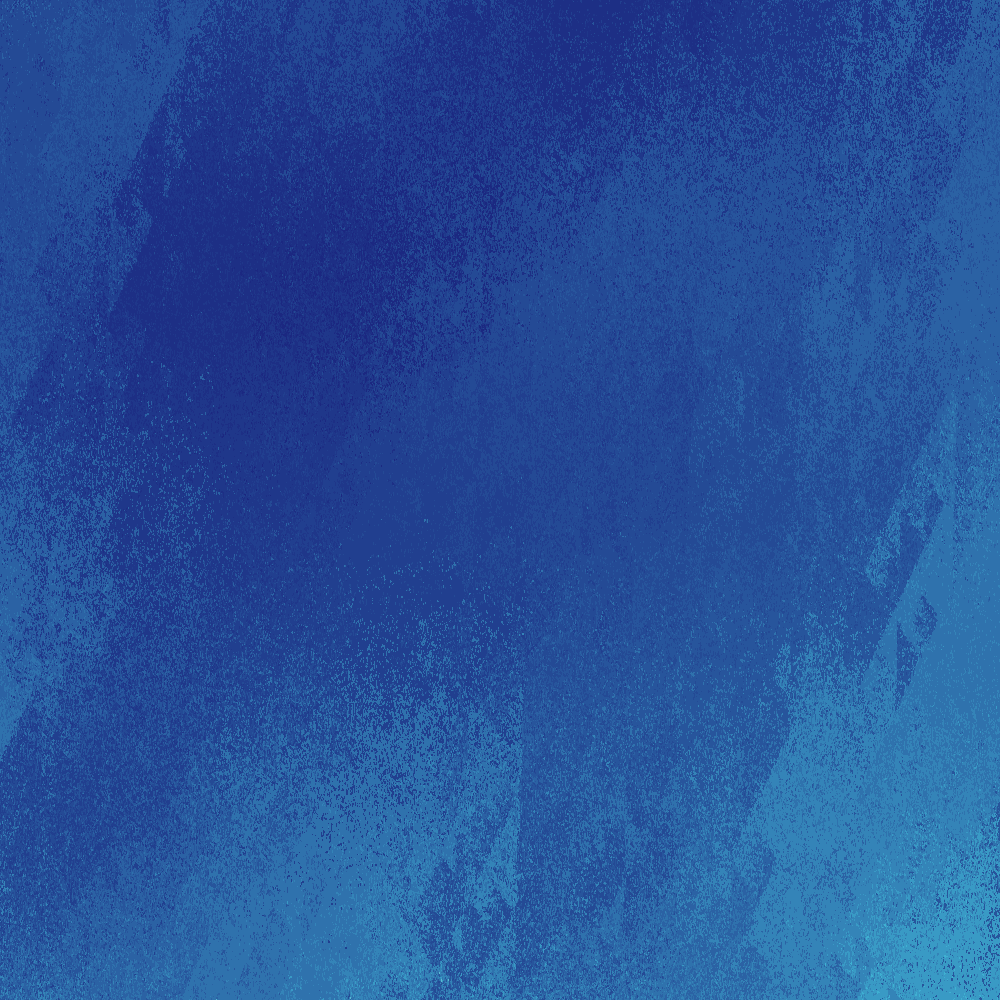}
}
\caption{We zoom into the cohomology fractal for \texttt{m122(4,-1)} while increasing $R$. 
The field of view  of the image is proportional to $e^{-R}$.
Noise due to rounding errors becomes visible at $R\sim 10.5$ and completely dominates the picture when $R\sim13.5$. }
\label{Fig:NumNoise}
\end{figure}

\section{Experiments}
\label{Sec:Experiments}

The sequence of images in \reffig{VisSphereRadii} suggests that some form of fractal object is coming into focus.  
When $R$ is small, the function $\Phi_R = \Phi^{F,b,D}_{R}$ is constant on large regions of $D$. 
As $R$ increases, these regions subdivide, producing intricate structures. 

As we have defined it so far, the cohomology fractal depends on $R$. 
A natural question is whether or not there is a limiting object that does not depend on $R$.
In this section we describe a sequence of experiments we undertook to explore this question. 
Inspired by these, in Sections~\ref{Sec:CLT} and~\ref{Sec:Pixel} we provide mathematical explanations of our observations.

\subsection{These pictures do not exist}

A na\"ive guess might be that the cohomology fractal converges to a function as $R$ tends to infinity. 
However, consider a ray following a closed geodesic $\gamma$ in $M$ that has positive algebraic intersection with the surface $F$. 
Choosing $D$ so that it contains a tangent direction $v$ along $\gamma$, we see that $\Phi^{F}_{R}(v)$ diverges to infinity as $R$ tends to infinity. 
The issue is not restricted to the measure zero set of rays along closed geodesics. 
Suppose that $v$ is a generic vector in a material view $D$. 
Recall that the geodesic flow is ergodic~\cite[Hauptsatz~7.1]{Hopf39}.  
Thus the ray starting from $v$ hits $F$ infinitely many times. 
So $\Phi^{F}_{R}(v)$ again diverges. Thus we have the following theorem.

\begin{theorem}
\label{Thm:NoPicture}
Suppose that $M$ is a finite volume, oriented hyperbolic three-manifold.  
Suppose that $p$ is any point of $M$.  
Suppose that $F$ is a compact, transversely oriented surface.
Then the limit 
\[
\lim_{R \to \infty} \Phi_R(v) 
\]
does not exist for almost all $v \in \UT{p}{M}$. \qed
\end{theorem}

\begin{remark}
To generalise \refthm{NoPicture} from finite volume to infinite volume manifolds, 
we must replace Hopf's ergodicity theorem by some other dynamical property. For example, Rees~\cite[Theorem~4.7]{Rees81} proves the ergodicity of the geodesic flow on the infinite cyclic cover of a hyperbolic surface bundle. This is generalised to the bounded geometry case by Bishop and Jones~\cite[Corollary~1.4]{BishopJones97}. Both of these works rely in a crucial fashion on Sullivan's equivalent criteria for ergodicity~\cite[page~172]{Sullivan79}. 
\end{remark}

One might hope that as $R$ tends to infinity, nearby points diverge in similar ways. If so, we might be able to rescale and have, say, $\Phi_{R}/R$ or $\Phi_{R}/\sqrt{R}$ converge. 
However, increasing $R$ in our implementation produces the sequence of images shown in \reffig{VisSphereRadii2}.  
We see that, as we increase $R$, the images become noisy as neighbouring pixels appear to decorrelate.
Eventually the fractal structure is washed away. Dividing the cohomology fractal by, say, some power of $R$ only changes the contrast. Depending on this power, the limit is either almost always zero or does not exist.

\begin{figure}[htb!]
\centering
\subfloat[$R=e^{2.5}$]{
\includegraphics[width=0.225\textwidth]{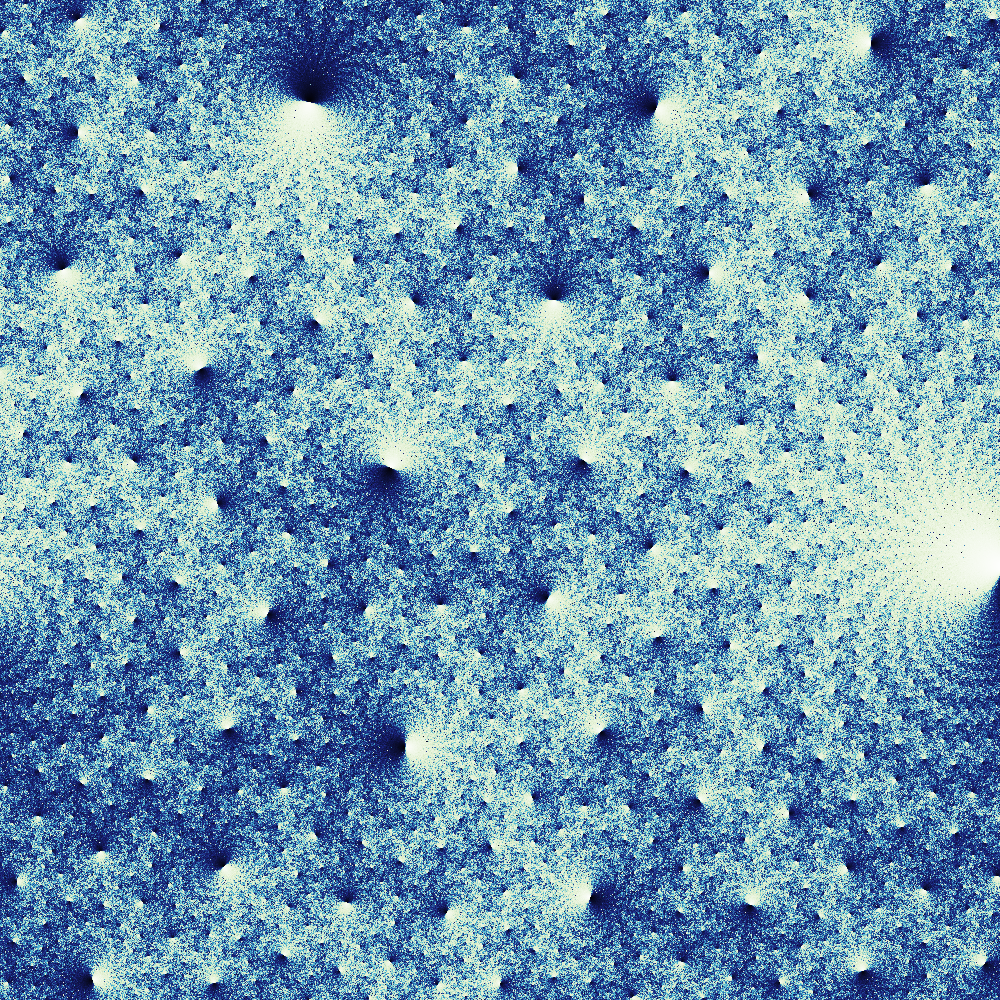}
}
\subfloat[$R=e^{3}$]{
\includegraphics[width=0.225\textwidth]{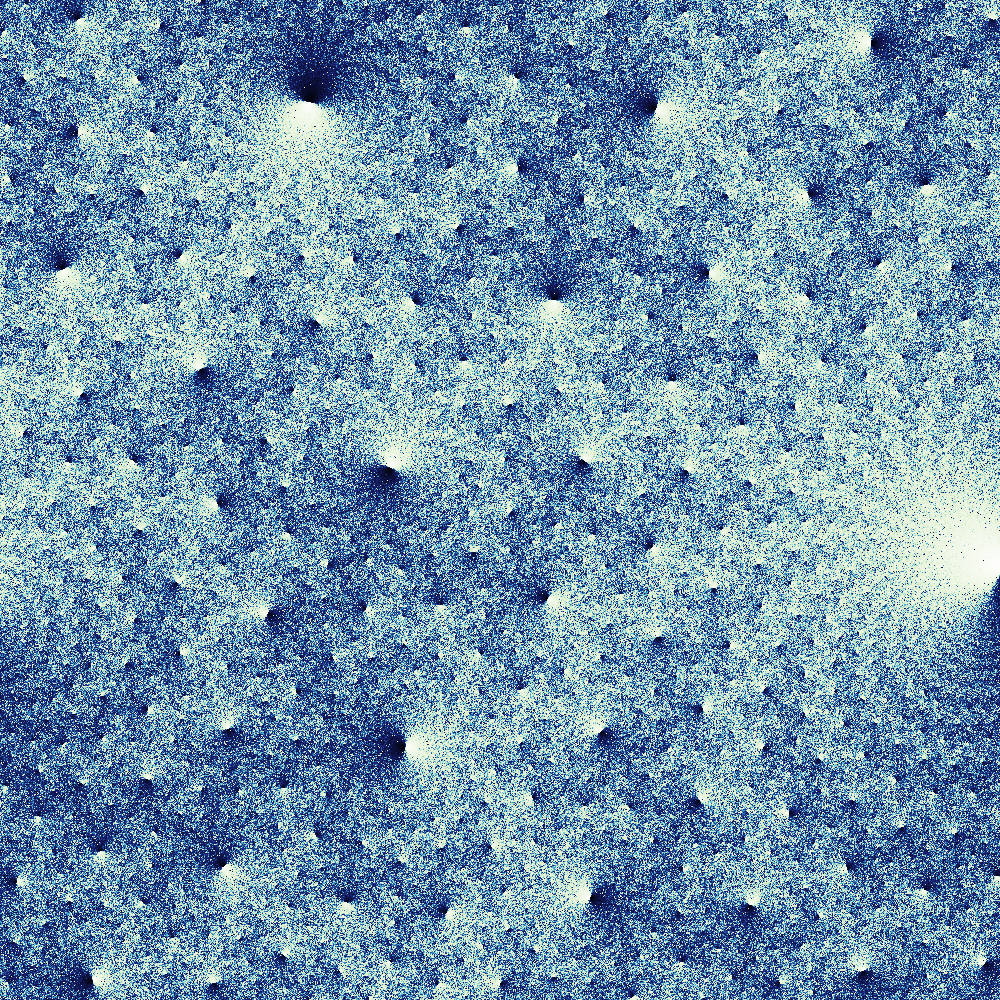}
}
\subfloat[$R=e^{4}$]{
\includegraphics[width=0.225\textwidth]{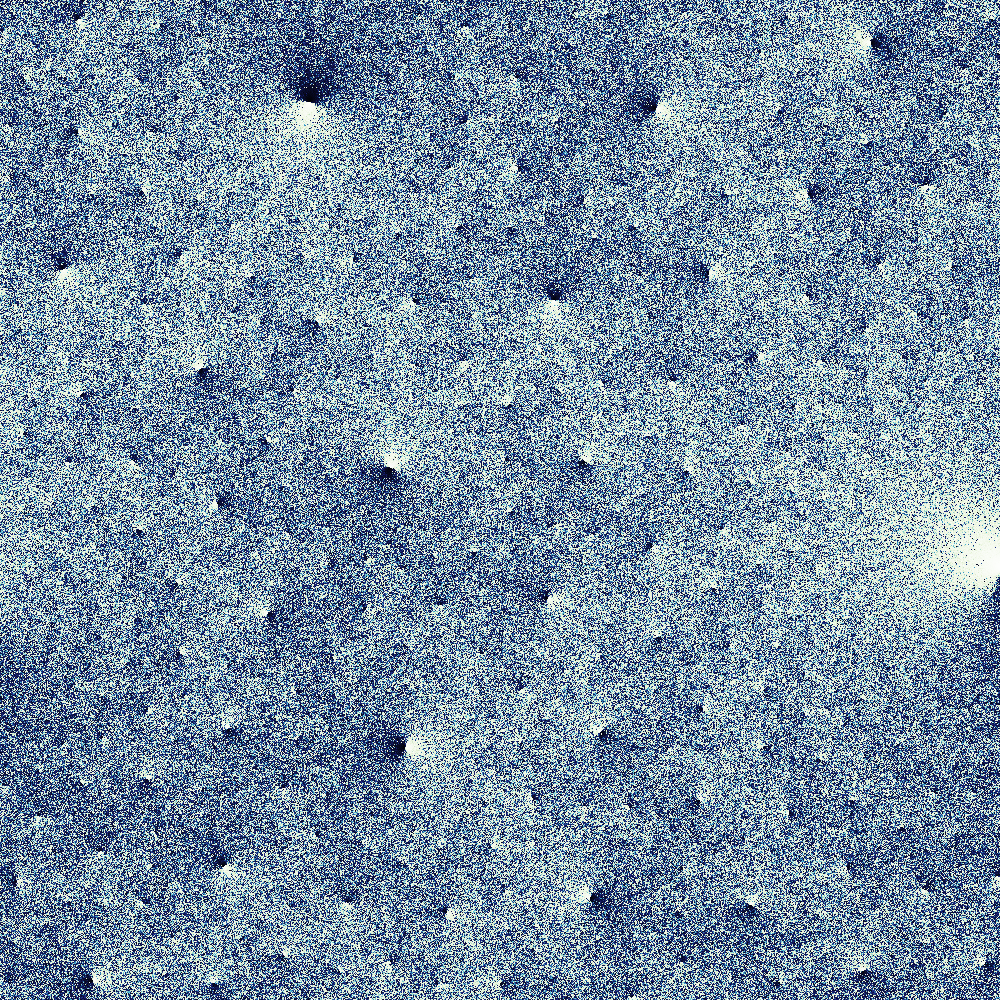}
}
\subfloat[$R=e^{5}$]{
\includegraphics[width=0.225\textwidth]{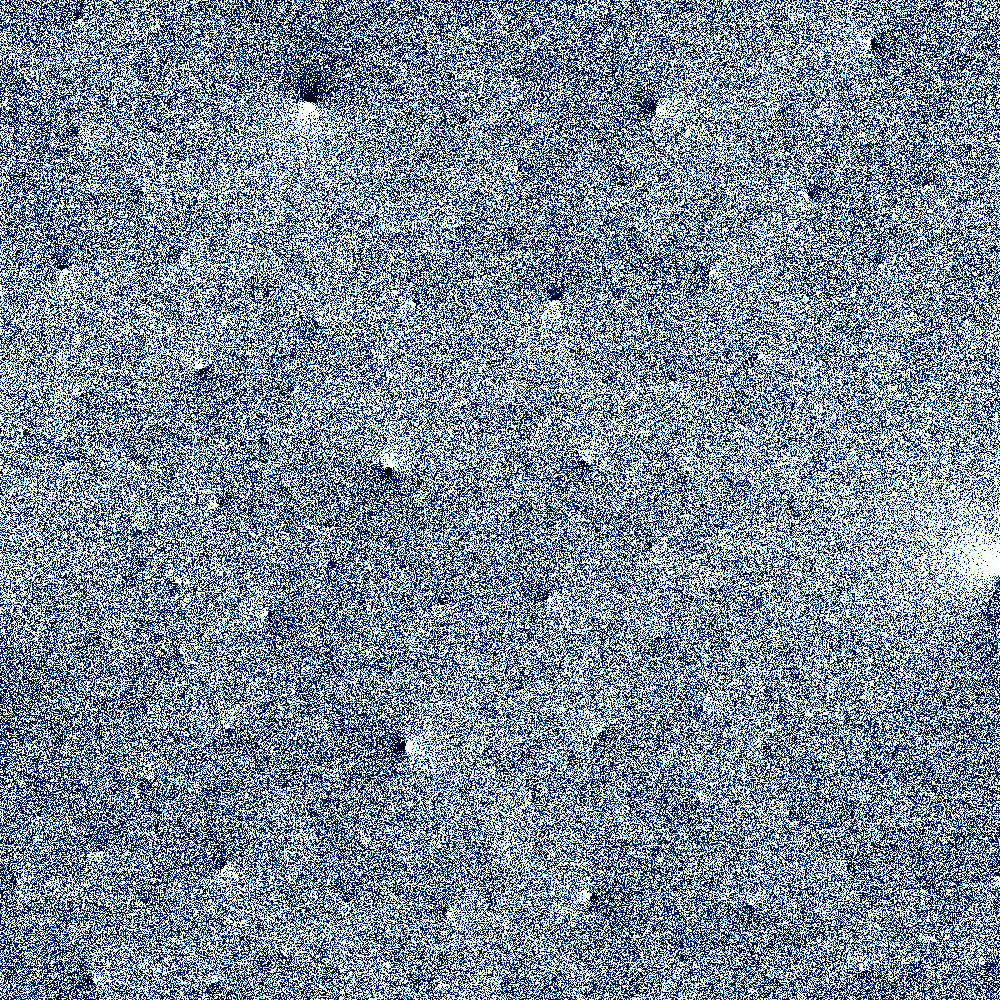}
}
\caption{Cohomology fractals for \texttt{m004}, with larger values of $R$. Each image here and in \reffig{VisSphereRadii} has $1000\times1000$ pixels. }
\label{Fig:VisSphereRadii2}
\end{figure}

\reffig{VisSphereRadii2} also demonstrates that \refrem{Quasigeodesic}, while valid, is misleading; it is true that for large $R$, every ray ends up somewhere within its pixel, but the colour one obtains is random noise. 
This noise is due to undersampling.  
In our images each pixel $U$ is coloured using a single ray passing (almost, as we saw in \refsec{Accumulate}) through its centre. 
When $R$ is small relative to the side length of $U$ the function $\Phi_{R}|U$ is generally constant; thus any sample is a good representative.  
As $R$ becomes larger 
the function $\Phi_{R}|U$ varies more and more wildly; thus a single sample does not suffice. 

\subsection{Take a step back and look from afar}

Let $D$ be an ideal view in the sense of \refsec{Views}. We identify $\pi(D)$, isometrically, with the euclidean plane $\CC$. Using this identification, we may refer to the vectors of $D$ as $z_D$ for $z \in \CC$.
Let $E$ be the ideal view obtained from $D$ by flowing outwards by a distance $d$. Thus, $\varphi_d(D) = E$. We similarly identify $\pi(E)$ with the euclidean plane, in such a way that for each $z \in \CC$, we have $\varphi_d(z_D) = (e^d z)_E$. We may now state the following.

\begin{figure}[htbp]
\labellist
\small\hair 2pt
\pinlabel {$\bdy \HH^3$} [r] at 0 0
\pinlabel {$D$} [r] at 34 218
\pinlabel {$E$} [r] at 173 74
\pinlabel {$R$} [l] at 340 50
\pinlabel {$R+d$} [l] at 493 122.5
\endlabellist
\includegraphics[width = 0.6\textwidth]{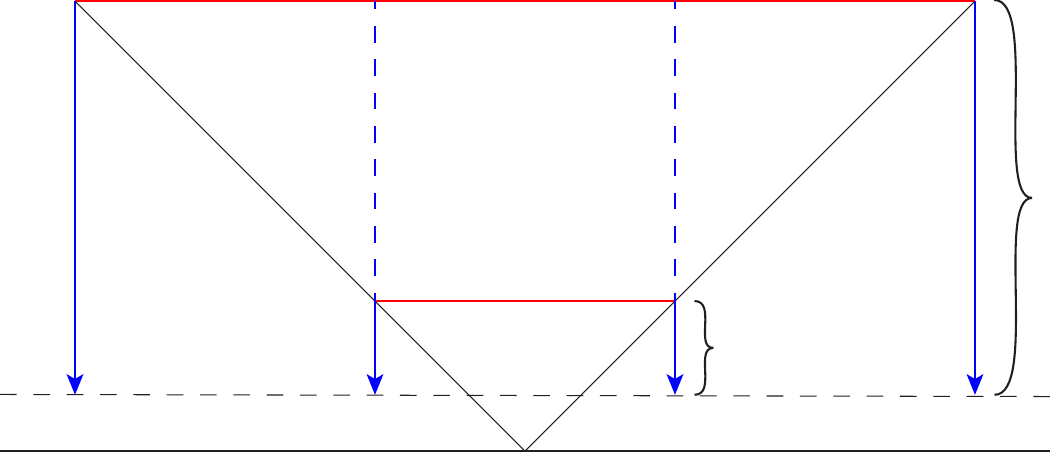}
\caption{Side view of ``screens'' (in red) for two ideal views, drawn in the upper half space model. The outward pointing normals to each horosphere point down in the figure. }
\label{Fig:TwoIdealViews}
\end{figure}
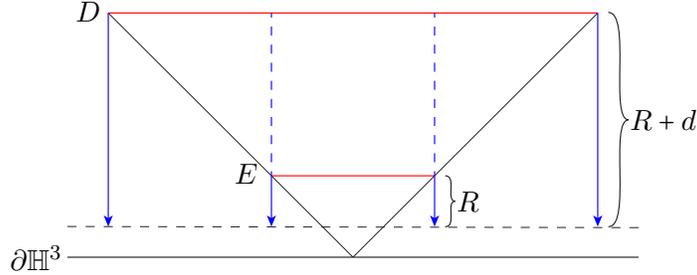

\begin{lemma}
\label{Lem:MovingIsScaling}
Suppose that $D$ is an ideal view and $E = \varphi_d(D)$.
Then the cohomology fractal based at $b$ satisfies
\[
\Phi_{R+d}^{\omega, b,D}\left(z_D\right) = \Phi_{R}^{\omega, b,E}\left((e^d z)_E\right)
\] 
\end{lemma}

\begin{proof}
Consider \reffig{TwoIdealViews}.
\end{proof}

Said another way, if we fly backwards a distance $d$ and replace $R$ with $R+d$, we see the exact same image, scaled down by a factor of $e^d$. 
As a consequence, in the ideal view we have the following.

\begin{remark}
Each small part of a cohomology fractal with large $R$ is the same as the cohomology fractal for a smaller $R$ with a different view.
\end{remark}

\begin{remark}
\label{Rem:Embiggen}
Since we know that we can make non-noisy images for small enough values of $R$, 
we can therefore make a non-noisy image of a cohomology fractal for any value of $R$, as long as we are willing to use a screen with high enough resolution. 
\end{remark}

The natural question then is how the \emph{perceived} image changes as we simultaneously increase the resolution and increase $R$. 
This convergence question is different from the convergence of the cohomology fractal to a function as in \refthm{NoPicture}: when we look at a very large screen from far away, our eyes average the colours of nearby pixels. Thus, we move away from thinking of the limit as a function evaluated at points, towards thinking of it as a measure evaluated by integrating over a region. As we will see later, in fact the correct limiting object is a distribution.

\subsection{Supersampling}

\begin{figure}[htbp]
\begin{tabular}{c m{0.275\textwidth} m{0.275\textwidth} m{0.275\textwidth}}
 & \multicolumn{1}{c}{$1\times 1$} & \multicolumn{1}{c}{$2\times2$} & \multicolumn{1}{c}{$128\times 128$} \\
 4 &
\includegraphics[width=0.29\textwidth]{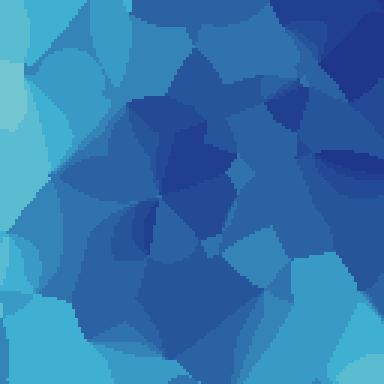} &
\includegraphics[width=0.29\textwidth]{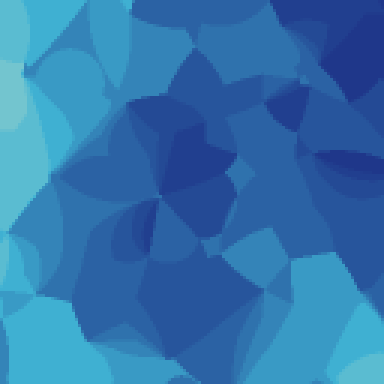} &
\includegraphics[width=0.29\textwidth]{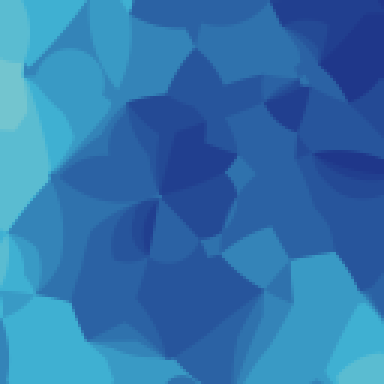} \\
6 &
\includegraphics[width=0.29\textwidth]{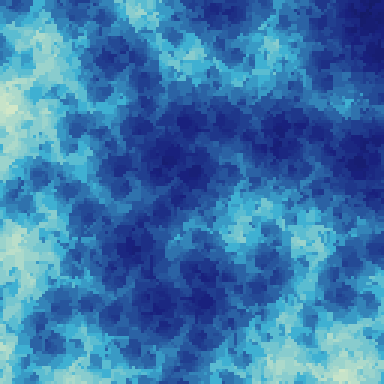} &
\includegraphics[width=0.29\textwidth]{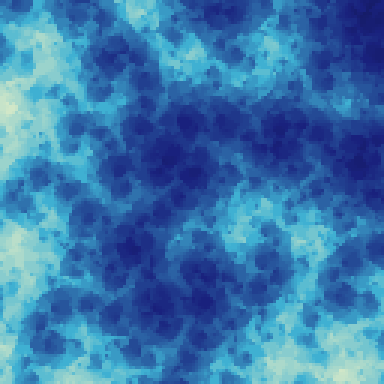} &
\includegraphics[width=0.29\textwidth]{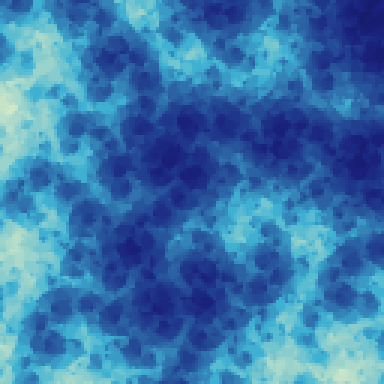} \\
 8 &
\includegraphics[width=0.29\textwidth]{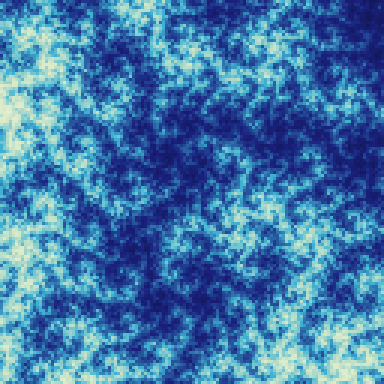} &
\includegraphics[width=0.29\textwidth]{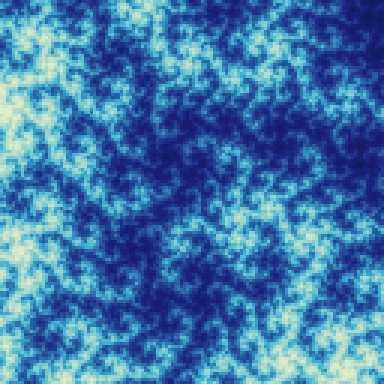} &
\includegraphics[width=0.29\textwidth]{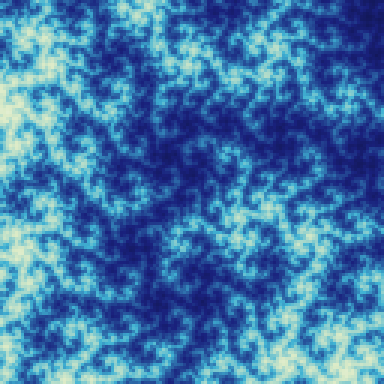} \\
10 &
\includegraphics[width=0.29\textwidth]{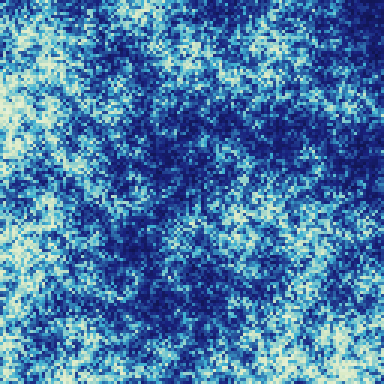} &
\includegraphics[width=0.29\textwidth]{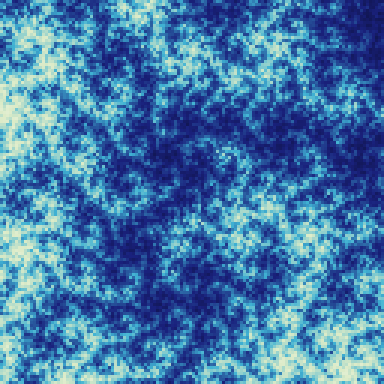} &
\includegraphics[width=0.29\textwidth]{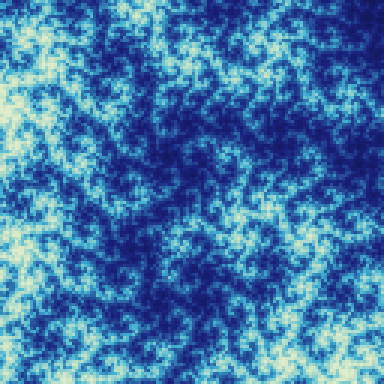} \\
12 &
\includegraphics[width=0.29\textwidth]{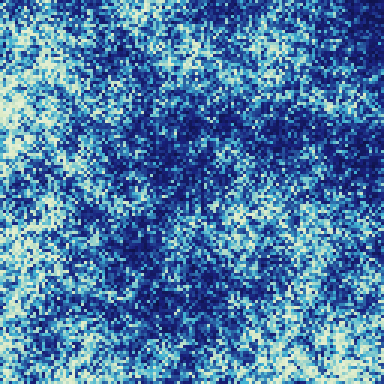} &
\includegraphics[width=0.29\textwidth]{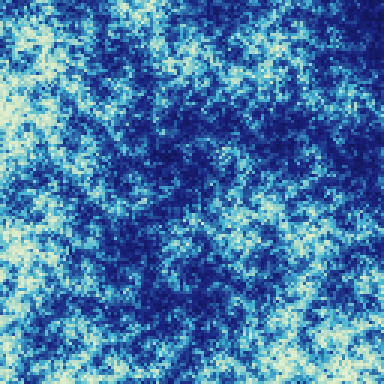} &
\includegraphics[width=0.29\textwidth]{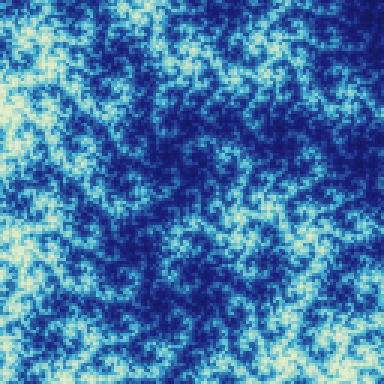} \\
\end{tabular}
\caption{\texttt{m122(4,-1)}. Field of view: $12.8^\circ$. $128\times 128$ pixels. For each image, the visual radius $R$ is given at the start of its row, while the number of samples per pixel is given at the top of its column.}
\label{Fig:RagainstNumSamples}
\end{figure}

To investigate this without requiring ever larger screens to view the results, we sample the cohomology fractal at many vectors in a grid within each pixel and average the results to give a colour for the pixel. That is, we employ supersampling. See \reffig{RagainstNumSamples}.
Here we draw cohomology fractals with $R$ ranging from $4$ to $12$, and with either $1$, $2^2$, or $128^2$ subsamples for each pixel. Each image has resolution $128 \times 128$.


\begin{remark*}
Note that some pdf readers do not show individual pixels with sharp boundaries: they automatically blur the image when zooming in. To combat this blurring and see the pixels clearly, we have scaled each image by a factor of three, so each pixel of our result is represented by nine pixels in these images.
\end{remark*}

With one sample per pixel, as we increase $R$ the fractal structure comes into focus but then is lost to noise. This matches our observations in Figures~\ref{Fig:VisSphereRadii} and~\ref{Fig:VisSphereRadii2}.
Taking subsamples and averaging makes little difference for small $R$: the only advantage is an anti-aliasing effect on the boundaries between regions of constant value. However, subsamples help greatly with reducing noise for larger $R$. With $2\times 2$ subsamples, we see much less noise at $R=10$, becoming more noticeable at $R=12$. Taking $128\times128$ samples seems to be very stable: there is almost no difference between the images with $R=10$ and $R=12$. This suggests that the perceived images converge. 

\subsection{Mean and variance within a pixel}

To better understand how subsampling interacts with increasing $R$, in \reffig{SubPixEvolution} we graph the average value within a selection of pixel-sized regions as $R$ increases.

\begin{figure}[htb]
\includegraphics[width=0.3\textwidth]{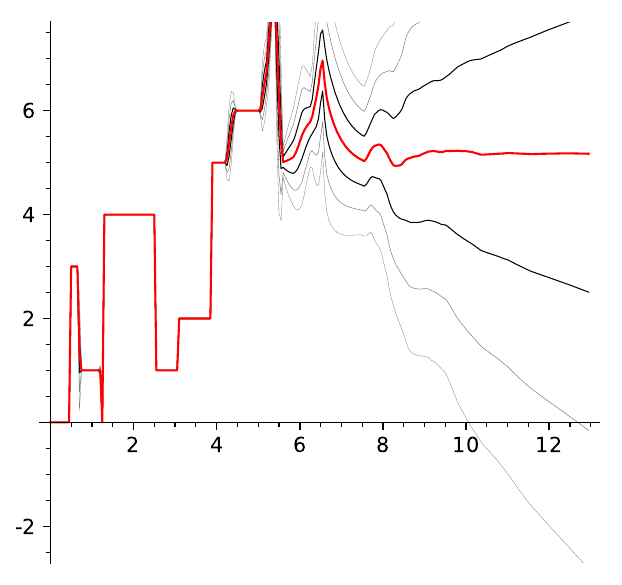}
\includegraphics[width=0.3\textwidth]{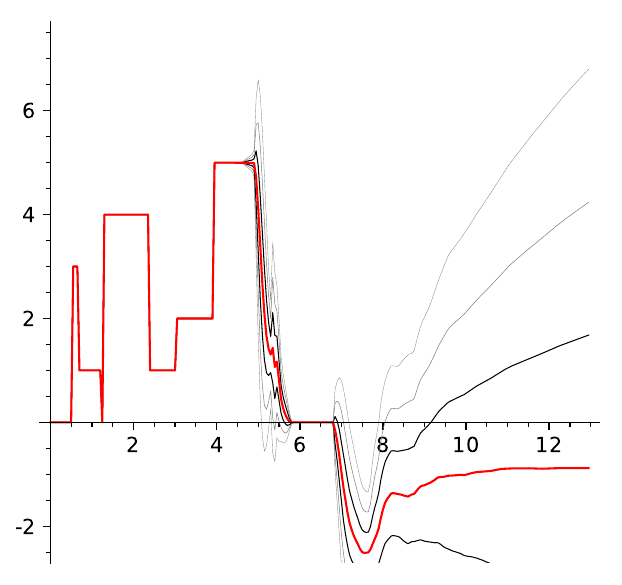}
\includegraphics[width=0.3\textwidth]{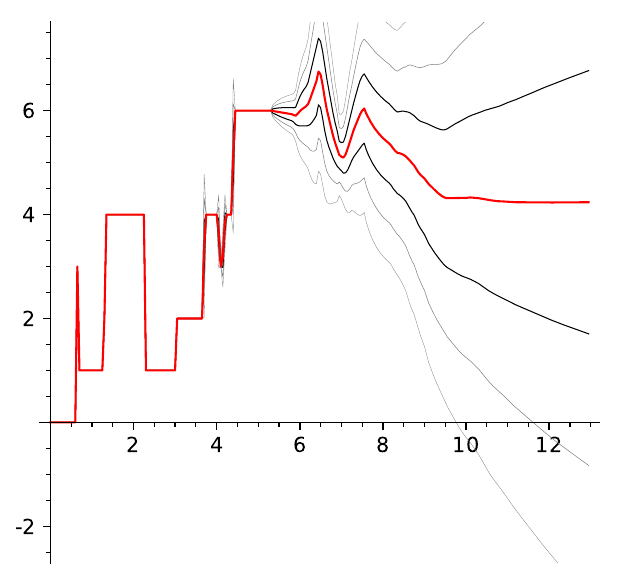}\\
\includegraphics[width=0.3\textwidth]{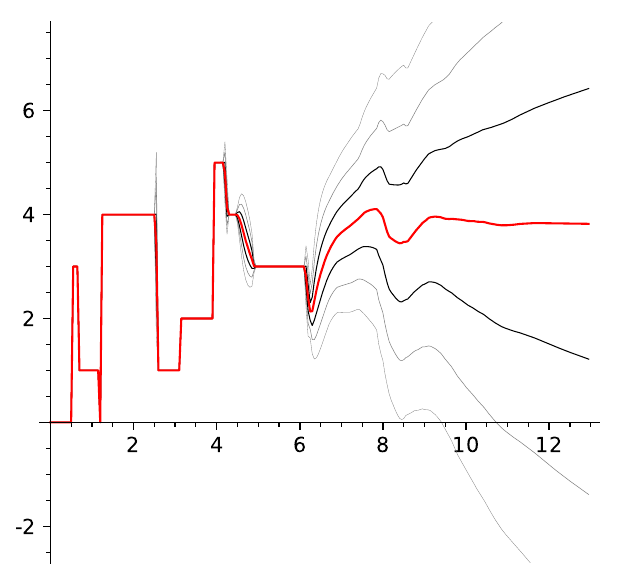}
\includegraphics[width=0.3\textwidth]{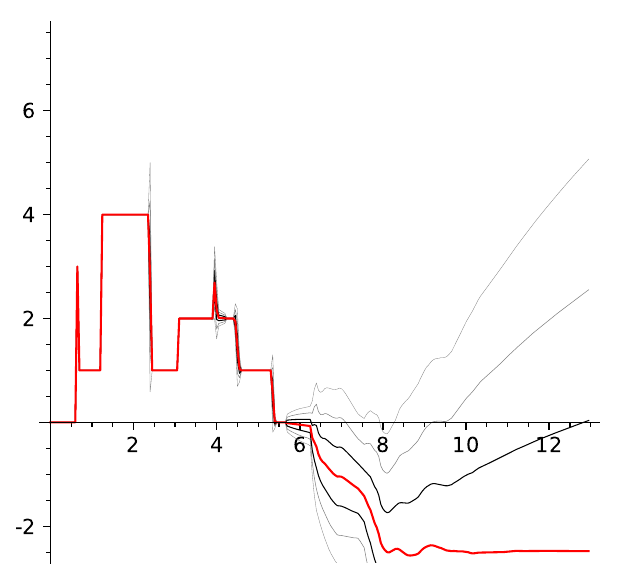}
\includegraphics[width=0.3\textwidth]{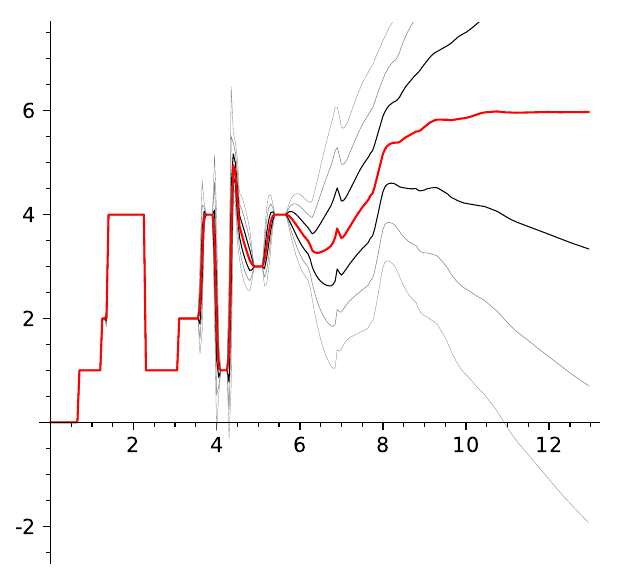}\\
\includegraphics[width=0.3\textwidth]{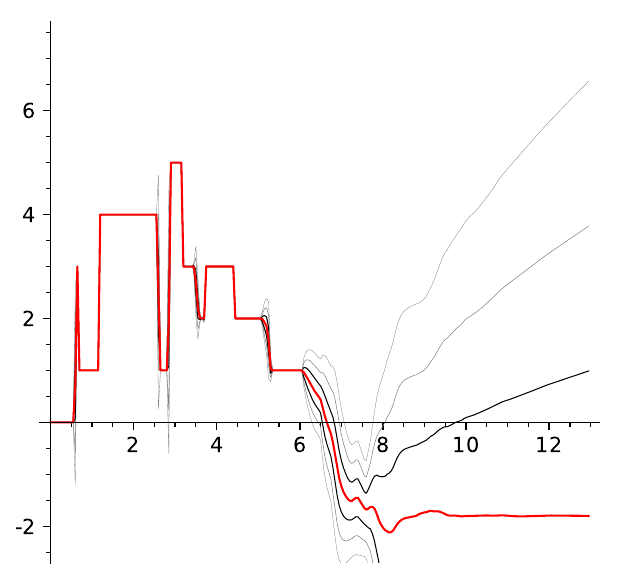}
\includegraphics[width=0.3\textwidth]{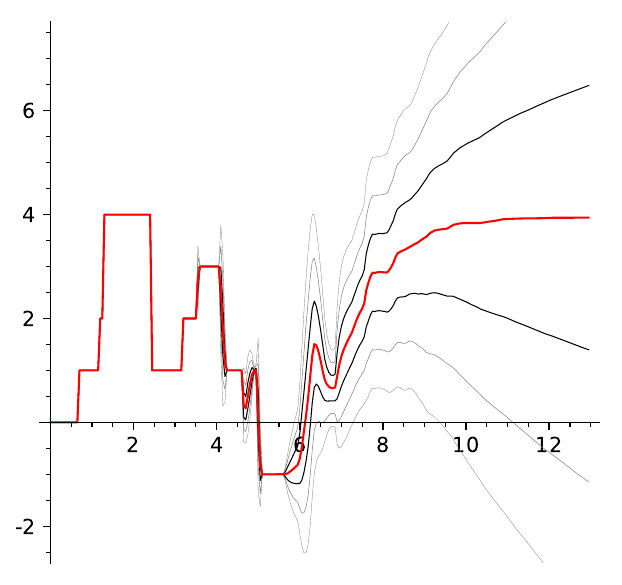}
\includegraphics[width=0.3\textwidth]{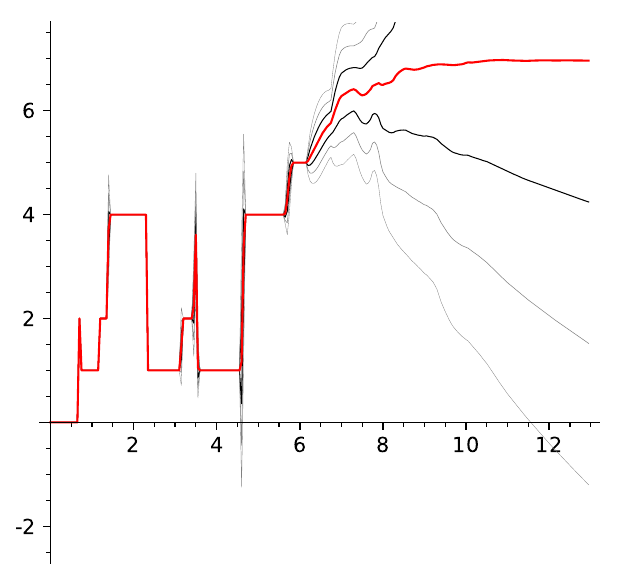}
\caption{The graph of the average value of the cohomology fractal  for \texttt{m122(4,-1)} for various square regions $U$ with field of view $0.1^\circ$. Thus, these are the same size as the pixels of \reffig{RagainstNumSamples}. These are each computed by taking $1000\times 1000$ samples. We also show the envelopes of 0.5, 1.0 and 1.5 standard deviations.
}
\label{Fig:SubPixEvolution}
\end{figure}

When $R$ is small, the graphs are more-or-less step functions, as much of the time the pixel $U$ is inside of a constant value region of the cohomology fractal. The graphs are also very similar for small $R$. This is because the pixels are close to each other, so all of their rays initially cross the same sequence of faces of the triangulation.
Around $R=6$, we reach the ``last step'' of the step function, then the regions of constant value become smaller than $U$.
For $R \geq 10$, the mean seems to settle down, while the standard deviation appears to grow like $\sqrt{R}$. 

Again this suggests that the perceived images converge. However, if the standard deviation continues to increase with $R$, then eventually any number of subsamples within each pixel will succumb to noise.

\subsection{Histograms}

We have looked at the standard deviation of a sample of values within a pixel. Next, we analyse the distribution of these values in more detail. See \reffig{m122Histogram}.

\begin{figure}[htb]
\centering
\subfloat[Histogram of weights and the normal distribution with the same mean and standard deviation.]{
\includegraphics[width = 6.5cm]{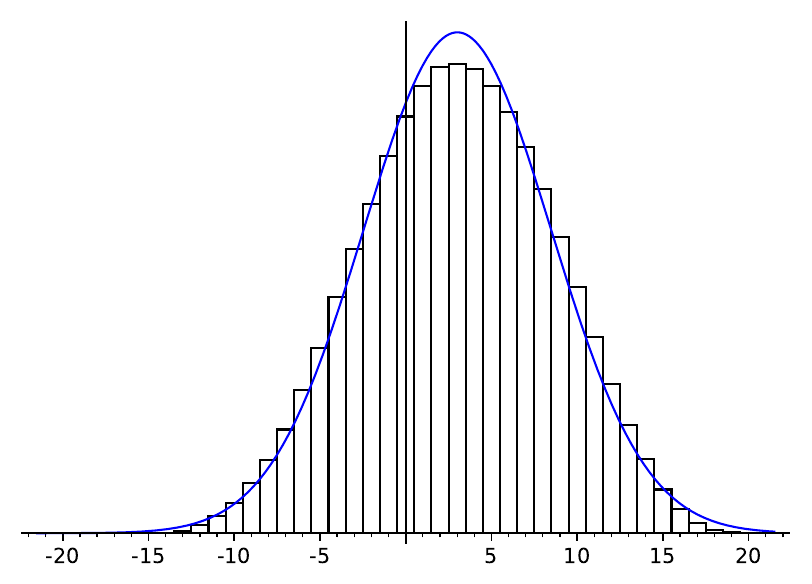}
\label{Fig:m122(4,1)Dist}
}
\subfloat[Cohomology fractal.]{
\includegraphics[height = 5cm]{Figures/FiguresMatthiasVer2/m122_4_-1/without_evil}
\label{Fig:m122(4,1)CohomFrac}
}
\caption{Statistics for a cohomology fractal of \texttt{m122(4,-1)} for a square region with field of view $20^\circ$ and $R=e^2$.}
\label{Fig:m122Histogram}
\end{figure}

We fix $R = e^2$. We sample $\Phi_R$ at each point of a $1000 \times 1000$ grid within a square of a material view with field of view $20^\circ$. We chose a relatively large field of view here so that we get an ``in focus'' image of the cohomology fractal with a relatively small value of $R$. Here we are being cautious to get good data, avoiding potential problems that our implementation has with large values of $R$ as discussed in \refsec{Accumulate}. 

We histogram the resulting data with appropriate choices of bucket widths.  In \reffig{m122(4,1)Dist} we show the histogram and the normal distribution with the same mean and standard deviation for our closed example, \texttt{m122(4,-1)}. 
In \reffig{m122(4,1)CohomFrac} we show the sample data as a 1000 by 1000 pixel image. We also draw the normal distribution with the same mean and standard deviation; the data seems to fit this well. 

\begin{figure}[htb]
\centering
\subfloat[Histogram of weights and the normal distribution with the same mean and standard deviation.]{
\includegraphics[width = 6.5cm]{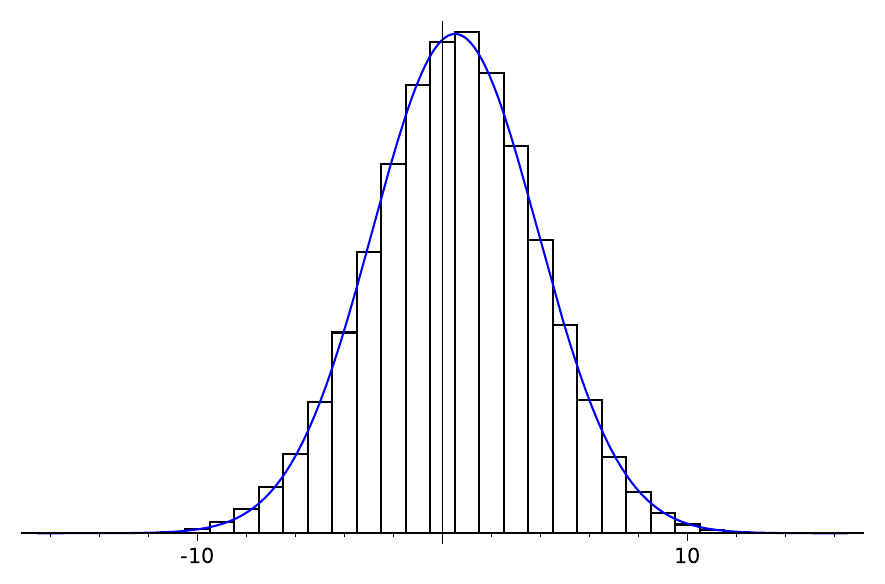}
\label{Fig:s789Dist}
}
\subfloat[Cohomology fractal.]{
\includegraphics[height = 5cm]{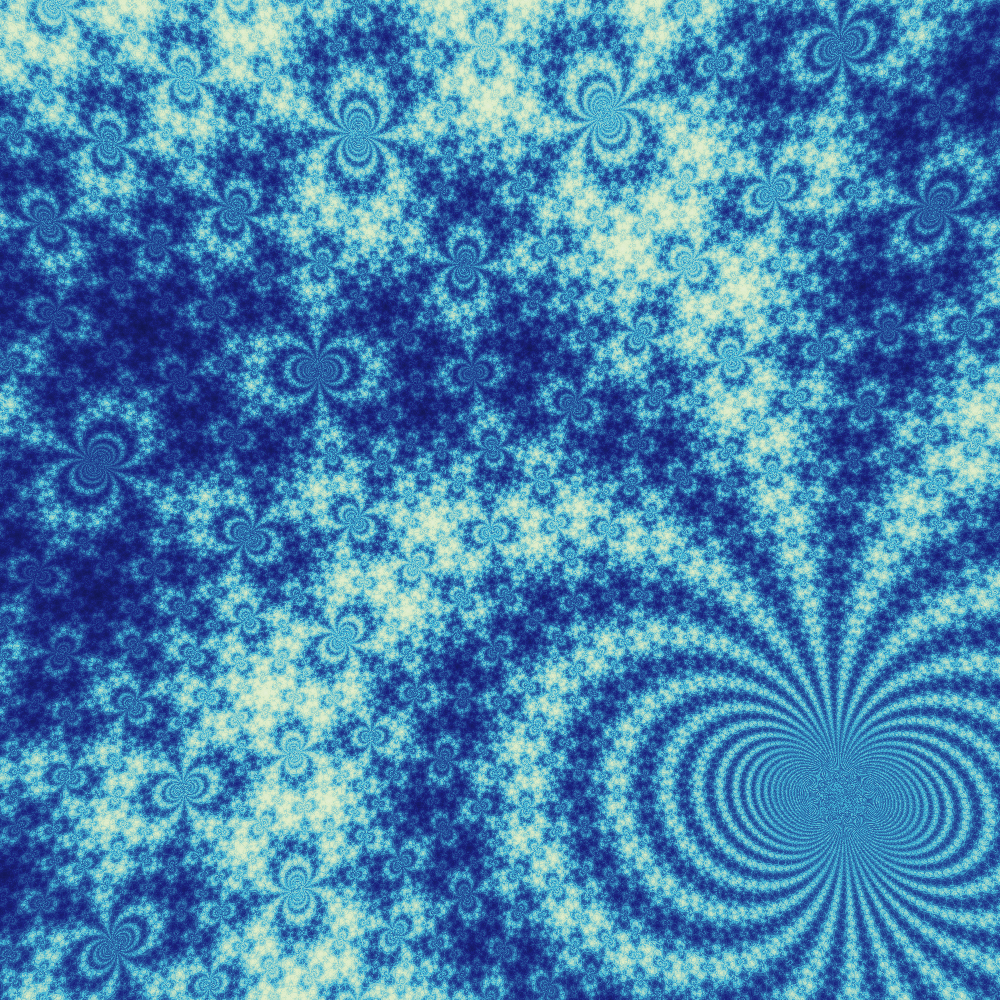}
}
\caption{Statistics for a cohomology fractal of \texttt{s789} for a class vanishing on cusp.}
\label{Fig:s789Histogram}
\end{figure}

\begin{figure}[htb]
\centering
\subfloat[Histogram of weights and the normal distribution with the same mean and standard deviation.]{
\includegraphics[width = 6.5cm]{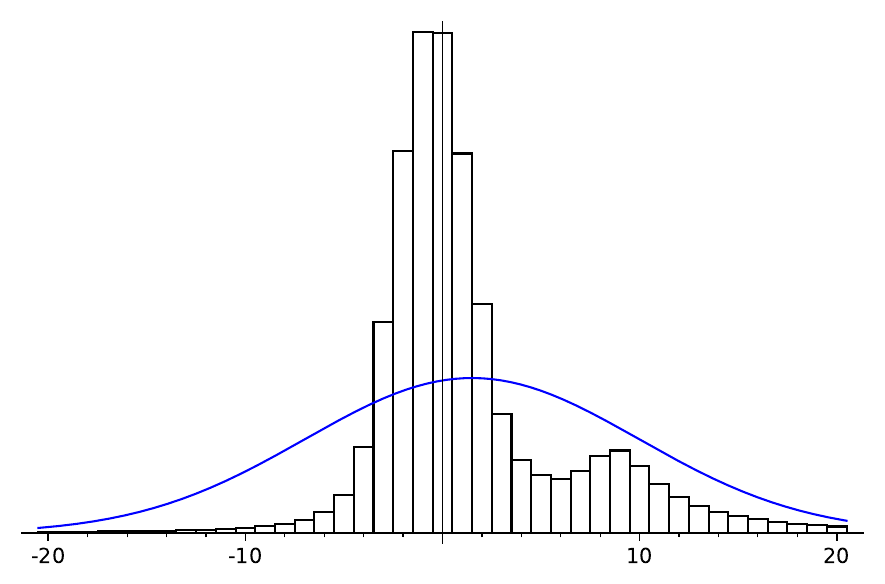}
}
\subfloat[Cohomology fractal.]{
\includegraphics[height = 5cm]{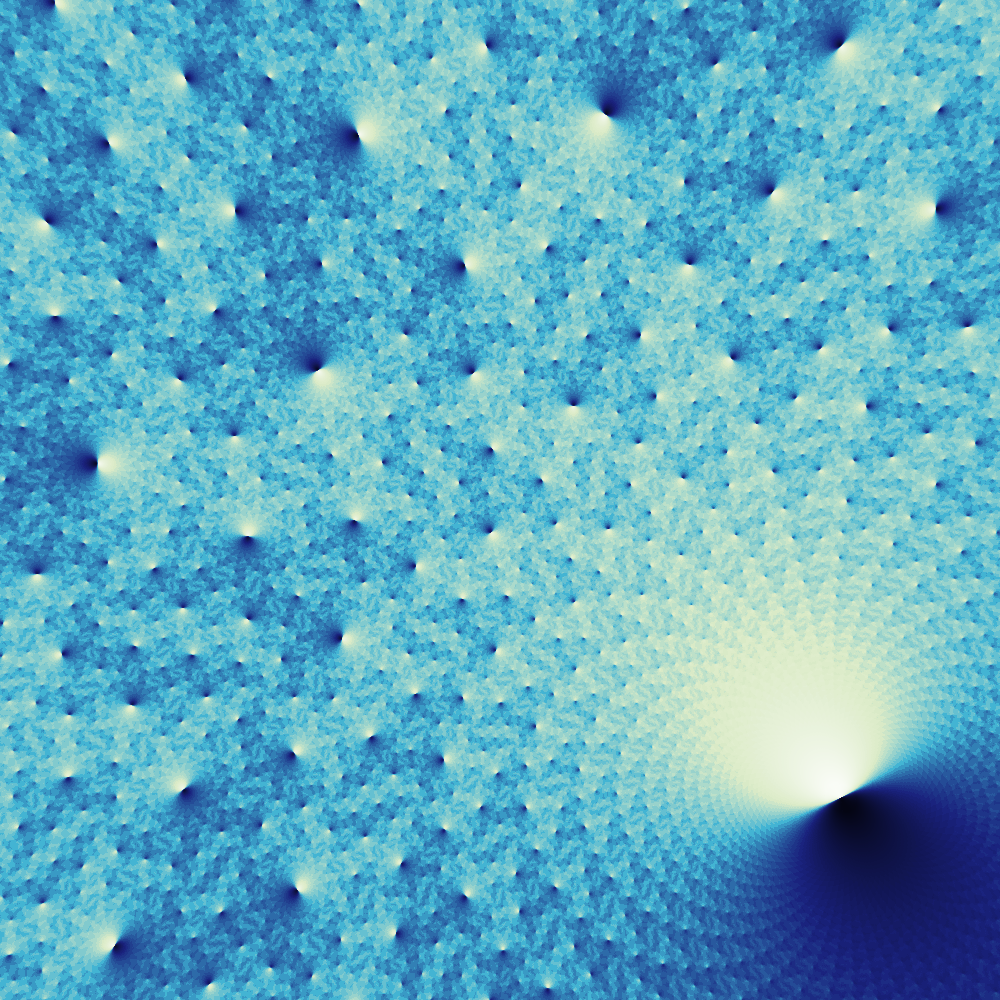}
}
\caption{Statistics for a cohomology fractal of \texttt{s789} for a class not vanishing on cusp.}
\label{Fig:s789Histogram2}
\end{figure}

We repeat this experiment back in the cusped case with \texttt{s789}. See Figures~\ref{Fig:s789Histogram} and~\ref{Fig:s789Histogram2}. Here we show the cohomology fractal for two different cohomology classes $[\omega]\in H^1(M)$. The cohomology class shown in \reffig{s789Histogram} vanishes when restricted to $\bdy M$, while in \reffig{s789Histogram2} it does not. The distribution appears to be normal when the cohomology class vanishes on $\partial M$. 
When $[\omega]$ does not vanish on $\bdy M$, something more complicated appears to be happening. One feature here is that the tails are much too long for a normal distribution. A heuristic explanation for this is that in a neighbourhood of the cusp, a geodesic ray crosses the surface repeatedly in the same direction. This allows it to gain a linear weight in logarithmic distance.

\section{The central limit theorem}
\label{Sec:CLT}

In this section, we prove a central limit theorem for the values of the cohomology fractal $\Phi_T$ across a pixel. 
That is, the distribution of the values of the scaled cohomology fractal $R_T=\Phi_T/\sqrt{T}$ 
converges to a normal distribution with mean zero. 

\subsection{Setup} 
\label{Sec:formalDefView}

We recall the framework introduced in \refsec{Views} that unifies the material, ideal, and hyperideal views.
Let $M$ be a connected, orientable, finite volume, complete hyperbolic three-manifold. 
As above we set $X=\UT{}{M}$ and $\cover{X}=\UT{}{\cover{M}}$.
We call a two-dimensional subset $D \subset \cover{X}$ a \emph{view} if it is of one of the following.%
\begin{itemize}
\item For a material view, fix a basepoint $p\in\cover{M}$ and let $D=\UT{p}{\cover{M}}$. 
Note that $D$ can be identified isometrically with $S^2$.
\item For the ideal view, fix a horosphere $H\subset\cover{M}$ and let $D$ be the set of outward normals to $H$. 
\item For the hyperideal view, fix a hyperbolic plane $H\subset\cover{M}$ and let $D$ be the set of normals to $H$ facing one of the two possible directions.
\end{itemize}

Note that $D \subset \cover{X}$ has a riemannian metric induced from the riemannian metric on $\cover{X}$. This metric also endows $D$ with an area two-form and associated measure denoted by $\zeta=\zeta_D$ and $\mu=\mu_D$. Recall that $\pi \from \cover{X} \to \cover{M}$ is the projection to the base space.

\begin{remark} 
\label{Rem:extrinsicCurvCorr}
Note that there is another riemannian metric on $D$ for the ideal and hyperideal view coming from isometrically identifying the horosphere or hyperbolic plane $H = \pi(D)$ with $\EE^2$ or with, respectively, $\HH^2$. 
Up to a constant factor, this metric is the same as the above metric.
The factor is trivial for the hyperideal view; it is $\sqrt{2}$ for the ideal view. This arises as $\sqrt{1+K_H^2}$ from the extrinsic curvature $K_H$ of $H$.
We have $K_H = 1$ so that adding $K_H$ to the ambient curvature $-1$ of $\HH^3$ gives zero, the horosphere's intrinsic curvature.
\end{remark}

In this notation, the definition of the cohomology fractal, for a given closed one-form $\omega\in\Omega^1(M)$ and basepoint $b\in\cover{M}$, becomes the following. For $v \in D$, we have
\begin{equation}
\label{eq:cohomFractView}
\Phi^{\omega, b, D}_{T}(v) = \int_0^T \omega(\varphi_t(v))\, dt + \int_{b}^{\pi(v)} \omega
\end{equation}
For the second integral, any path from $b$ to $\pi(v)$ in $\cover{M}$ can be chosen as $\omega$ is closed. This integral is constant in $v$ for the material view since $\pi(D)=p$. Choosing $W\in\Omega^0(\cover{M})$ so that $dW=\cover{\omega}$ and $W(b) = 0$, we can simply write
\[
\Phi_T(v)=\Phi^{\omega,b,D}_{T} (v) = W \circ \pi \circ \varphi_T (v)
\]

The central limit theorem will apply to probability measures $\nu_D$ that are absolutely continuous with respect to the area measure $\mu_D$ on $D$. 
We use the usual notation $\nu_D \ll \mu_D$ for absolute continuity.
By the Radon-Nikodym theorem, this is equivalent to saying that the measure $\nu_D$ is given by 
\begin{equation}
\label{Eqn:RN}
\nu_D(U) = \int_U h\cdot d\mu_D
\end{equation}
where $h\geq 0$ is measurable with $\int_D h\cdot d\mu_D = 1$.

\begin{remark}
\label{Rem:MeasuresForms}

In \refsec{Pixel}, we will switch from measures $\nu_D$ to forms $\eta_D$, for the following reason.
Here, in \refsec{CLT} we follow the well-established notation of \cite{Sinai60,zweimuellerInfMeasurePreserving2007}.
However, the transformation laws in \refsec{Pixel} are better stated in the language of forms.
In both cases we consider probability measures or two-forms that are ``products'': 
namely of a suitable function $h\from D\to\RR$ with the area measure $\mu_D$ or two-form $\zeta_D$ respectively. 
The function $h$ should be thought of as an indicator (or a kernel) function for a pixel.
\end{remark}

\subsection{The statement of the central limit theorem}
\label{Sec:StatementCLT}
The goal of this section is to prove the following.

\begin{theorem}
\label{Thm:CLT}
Fix a connected, orientable, finite volume, complete hyperbolic three-manifold $M$ and a closed, non-exact, compactly supported one-form $\omega\in\Omega_c^1(M)$. There is $\sigma > 0$ such that for all basepoints $b$, for all views $D$ with area measure $\mu_D$, for all probability measures $\nu_D\ll \mu_D$, and for all $\alpha \in \RR$, we have
\[
\lim_{T\to\infty} \nu_D\left[ v\in D : \frac{\Phi_T(v)}{\sqrt{T}} \leq \alpha \right] = \int_{-\infty}^\alpha \frac{1}{\sigma\sqrt{2\pi}} e^{-(s/\sigma)^2/2}\, ds
\]
where $\Phi_T=\Phi^{\omega,b,D}_{T}$ is the associated cohomology fractal.
\end{theorem}

Let us recall some notions from probability to clarify what this means. 
Let $(P, \nu)$ be a probability space. 
For each $T \in \RR_{>0}$, let $R_T \from P \to \RR$ be a measurable function. 
For each $T$, the probability measure $\nu$ on $P$ induces a probability measure $\nu \circ R_T^{-1}$ on $\RR$ telling us how the values of the random variable $R_T$ are distributed when sampling $P$ with respect to $\nu$. 
Let $\psi$ be a probability measure on $\RR$. 

\begin{definition}
\label{Def:ConvergeInDistribution}
We say that the random variables $R_T$ \emph{converges in distribution} to $\psi$ 
if the measures $\nu \circ R_T^{-1}$ converge in measure to $\psi$. 
This is denoted by $\nu \circ R_T^{-1} \Rightarrow \psi$. 
\end{definition}

Here, by the Portmanteau theorem, we can use any of several equivalent definitions of weak convergence of measures. 
We are only interested in the case where $\psi$ is absolutely continuous with respect to Lebesgue measure on $\RR$; 
that is, we can write $\psi(V) = \int_V p(x)\, dx$ for any measurable $V \subset \RR$.  
Note that here $p \from \RR \to \RR_{\geq 0}$ is the \emph{probability density function} for $\psi$.  
Convergence in distribution $\nu \circ R_T^{-1} \Rightarrow \psi$ is then equivalent to saying that for all $\alpha$ we have 
\[
\lim_{T\to \infty} \nu\left[ x\in P: R_T(x) \leq \alpha \right] = \int_{-\infty}^\alpha p(s) \cdot ds
\]
We define 
\[
n_\sigma(s)=\frac{1}{\sigma\sqrt{2\pi}} e^{-(s/\sigma)^2/2} \quad\mbox{and}\quad \psi_\sigma(V) = \int_V n_\sigma(x)\, dx
\]
The latter is the \emph{normal distribution} with mean zero and standard deviation $\sigma$.
\begin{example}
\newcommand{\head}{\mathrm{head}}
\newcommand{\tail}{\mathrm{tail}}
The process of flipping coins can be modelled as follows. Set $P=\{\head, \tail\}^\NN$. We define a measure $\nu_P$ on $P$ as follows.
Given any prefix $v$ of length $n$, the set of all infinite words in $P$ starting with $v$ has measure $2^{-n}$.
Let $S_i\from P\to \RR$ be the random variable $S_i(w) = \pm 1$ as $w_i$ is heads or tails respectively. Define $\Sigma_N = S_0 + S_1 + \cdots + S_{N-1}$. The classical central limit theorem states that
\[
\nu_P \circ R_N^{-1} \Rightarrow \psi_1\quad\mbox{where}\quad R_N=\frac{\Sigma_N}{\sqrt{N}}\from P\to\RR
\qedhere
\]
\end{example}
We can now restate \refthm{CLT} as
\[
\nu_D \circ R_T^{-1}\Rightarrow\psi_\sigma \quad\mbox{where}\quad R_T=\frac{\Phi_T}{\sqrt{T}}\from D\to\RR
\]

\subsection{Sinai's theorem} 
\label{Sec:Sinai}

Our proof of \refthm{CLT} starts with Sinai's central limit theorem for geodesic flows~\cite{Sinai60}. 
We use the following version of Sinai's theorem which is adopted from \cite[Theorem~VIII.7.1 and subsequent Nota Bene]{FranchiLeJan}. 
This applies to functions that are not derivatives in the following sense. Recall that $X = \UT{}{M}$.

\begin{definition} 
\label{Def:lieDerivative}
Let $f\from X\to\RR$ be a smooth function. 
We say that $f$ is a \emph{derivative} if there is a smooth function $F\from X\to\RR$ such that
\[
f(v)= \left.\frac{dF(\varphi_t(v))}{d t}\right|_{t=0} \qedhere
\] 
\end{definition}

\noindent 
Let $\mu_X = \mu_{\Haar}/\mu_{\Haar}(X)$ be the normalised Haar measure.

\begin{theorem}[Sinai-Le Jan's Central Limit Theorem] 
\label{Thm:sinai}
Fix a connected, orientable, finite volume, complete hyperbolic three-manifold $M$. 
Let $f \from X = \UT{}{M} \to \RR$ be a compactly supported, smooth function with $\int_X f \cdot d\mu_X = 0$. 
Assume $f$ is not a derivative. 
Let 
\[
R_T(v) =\frac{\int_0^T f(\varphi_t(v))\,dt}{\sqrt{T}}
\]
Then there is a $\sigma>0$ such that $\mu_X \circ R_T^{-1} \Rightarrow \psi_\sigma$. \qed
\end{theorem}

In fact, the constant $\sigma$ appearing in \refthm{sinai} is the square root of the \emph{variance} of $f$ which Franchi--Le Jan denote by $\calV(f)$. 
They give a formula for $\calV(f)$ in \cite[Theorem~VIII.7.1]{FranchiLeJan} and state that $\calV(f)$ vanishes if and only if $f$ is a derivative.

\begin{remark}
To relate \refthm{sinai} to~\cite[Theorem~VIII.7.1]{FranchiLeJan}, note that Franchi--Le Jan think of $f$ as a function on the frame bundle of $\cover{M}$ that is both $\Gamma$ and $\SO(2)$--invariant.
Since $f$ is smooth and compactly supported, it satisfies the hypotheses of their theorem.
Note that they also require $f$ to not be a derivative (denoted by $\mathcal{L}_0h$, see \cite[(VIII.1)]{FranchiLeJan}) of a function $h$ but allow $h$ to be a function on the frame bundle. 
However, if an $\SO(2)$--invariant $f$ is the derivative of a function $h$ on the frame bundle, it is also the derivative of an $\SO(2)$--invariant function on the frame bundle.
\end{remark}

We deduce \refthm{CLT} from Sinai's theorem in three steps.
\begin{enumerate}
\item 
\refthm{sinaiOne} generalises Sinai's theorem to arbitrary probability measures $\nu_X \ll \mu_X$ on the five-dimensional $X = \UT{}{M}$.
\item 
\refthm{sinaiTwo} restricts from $X$ to the two-dimensional view $D$ using a measure $\nu_D \ll \mu_D$.
\item 
Finally, we show that the term $\int_{b}^{\pi(v)}\omega$ from \eqref{eq:cohomFractView} can be added to obtain $\Phi_T$.
\end{enumerate}

\subsection{Generalising Sinai's Theorem}

We begin with a definition.

\begin{definition}
\label{Def:ConvInProb}
Let $(P, \mu)$ be a finite measure space. For $n \in \NN$, let $Q_n\from P\to \RR$ be a measurable function. 
We say that \emph{$Q_n$ converges to zero in probability} and write $\mu \circ Q_n^{-1}\to 0$ if for all $\varepsilon >0$ we have 
\[
\lim_{n\to\infty} (\mu\circ Q_n^{-1}) ((-\infty, -\varepsilon) \cup (\varepsilon,\infty)) \to 0 \qedhere
\]
\end{definition}

The following result is called \emph{strong distributional convergence}, see \cite[Proposition~3.4]{zweimuellerInfMeasurePreserving2007}.

\begin{theorem} 
\label{Thm:strongErgodicConvergence}
Let $(P, \mu)$ be a finite measure space and $T \from P \to P$ be an ergodic, measure-preserving transformation. 
For all $n\in\NN$, let $R_n \from P \to \RR$ be a measurable function. 
Let $Q_n = R_n\circ T-R_n$ and suppose that $\mu\circ Q_n^{-1}\to 0$.  
Let $\psi$ be a probability measure on $\RR$. 
If we have $\nu\circ R_n^{-1}\Rightarrow \psi$ for some probability measure $\nu\ll \mu$, then we have $\nu\circ R_n^{-1}\Rightarrow \psi$ for all probability measures $\nu\ll \mu$. \qed
\end{theorem}

\begin{remark}
We have specialised Zweim\"uller's Proposition~3.4 in \cite{zweimuellerInfMeasurePreserving2007} to finite measure spaces. 
To obtain Zweim\"uller's result for $\sigma$--finite measure spaces $(P,\mu)$, we need to replace the requirement $\mu\circ Q_n^{-1}\to 0$ by the following weaker requirement denoted by $Q_n\xrightarrow{\mu} 0$ in \cite[Footnote~3]{zweimuellerInfMeasurePreserving2007}: for all probability measures $\nu\ll \mu$ we have $\nu\circ Q_n^{-1}\to 0$. 
To see that $Q_n\xrightarrow{\mu} 0$ is weaker, we can use the following standard result: 
for any $\nu \ll \mu$ we have 
\[
\sup \{ \nu(A) : \mbox{$A$ measurable with $\mu(A)\leq \varepsilon$} \}
                \to 0\quad\mbox{as}\quad\varepsilon\to 0
\]
The requirements $\mu\circ Q_n^{-1}\to 0$ and $Q_n\xrightarrow{\mu} 0$ are equivalent if $\mu$ is finite.
\end{remark}

Using this, we now give our first variant of Sinai's theorem.

\begin{theorem} 
\label{Thm:sinaiOne}
With the same hypotheses as in \refthm{sinai}, we have the following. There is a $\sigma>0$ such that for any probability measure $\nu_X \ll \mu_X$ we have
$\nu_X \circ R_T^{-1} \Rightarrow \psi_\sigma$.
\end{theorem}

\begin{proof}
By \refthm{sinai}, there is a $\sigma$ such that $\mu_X \circ R_T^{-1} \Rightarrow \psi_\sigma$ as $T\to\infty$. 
Note that the random variables in \refthm{strongErgodicConvergence} are indexed by $n\in\NN$ instead of $T\in\RR$ but it is easy to see that sequential convergence and convergence in distribution are equivalent. 
In other words, it is suffices to show that for any sequence $(T_n)$ with $T_n\to\infty$ the random variables  $S_n=R_{T_n}$ satisfy $\nu_X \circ S_n^{-1}\Rightarrow \psi_\sigma$. 
In order to apply \refthm{strongErgodicConvergence}, we need the following two claims.

\begin{claim}
The time-one map $\varphi_1$ for the geodesic flow is ergodic.
\end{claim}

\begin{proof}
By~\cite[Theorem~V.3.1]{FranchiLeJan} the geodesic flow is mixing.
It follows that the time-one map $\varphi_1$ is also mixing, and thus ergodic.
\end{proof}

\begin{claim}
Let $Q_n = S_n \circ \varphi_1-S_n$. 
Then, $\nu_X \circ Q_n^{-1} \to 0$.
\end{claim}

\begin{proof}
We will prove a stronger statement: $\| Q_n \|_\infty \to 0$.

\begin{align*}
\left| Q_n(v)\right| 
& =  \left| \frac{\int_1^{T_n+1}     f(\varphi_t(v))\, dt}{\sqrt{T_n}} - \frac{\int_0^{T_n} f(\varphi_t(v))\, dt}{ \sqrt{T_n}}\right| \\
& =  \left| \frac{\int_{T_n}^{T_n+1} f(\varphi_t(v))\, dt}{\sqrt{T_n}} - \frac{\int_0^1     f(\varphi_t(v))\, dt}{ \sqrt{T_n}}\right| \\
& \leq   \frac{2 \| f\|_\infty}{\sqrt{T_n}} \qedhere
\end{align*}
\end{proof}

We can now finish the proof of \refthm{sinaiOne}. 
We apply \refthm{strongErgodicConvergence} with $R_n$ replaced by $S_n$ and $T$ replaced by $\varphi_1$.
\end{proof}

\subsection{Coordinates}
\label{Sec:coordinatesView}

Given a view $D$, we introduce coordinates for a neighbourhood of $D$ in $\cover{X}=\UT{}{\cover{M}}\homeo\UT{}{\HH^3}$ as follows; 
it may be helpful to consult \reffig{localCoordinates}. 
Fix $v \in \cover{X}$.  
If $v$ is close enough to $D$ in $\cover{X}$, then there is an $\xu \in D$ such that the rays emanating from $\xu$ and $v$ converge to the same ideal point in $\bdy_\infty\cover{M}$. 
Consider the set $H = H^\sfs(\xu) \subset \cover{X}$ such that
\begin{itemize}
\item $\xu \in H$,
\item $\pi(H)$ is a horosphere, and
\item $H$ are the ``inward pointing'' normals to $\pi(H)$.
\end{itemize}
Let $\xs$ be the intersection of $H$ with the line through $v$.
Let $\xf$ be the signed distance from $\xs$ to $v$ along this line. 
The triple 
\[
(\xu, \xf, \xs) 
\qquad
\mbox{with $\xu \in D$, $\xf \in \RR$, and $\xs \in H^\sfs(\xu)$}
\]
determines the vector $v \in \cover{X}$ uniquely. 
In an abuse of notation, we will simply write $v=(\xu, \xf, \xs)$.

\begin{figure}[htbp]
\centering
\subfloat[Coordinates.]{
\labellist
\small\hair 2pt
\pinlabel {$H$} at 173 210
\pinlabel {$\xu$} at 157 174
\pinlabel {$\xs$} at 114 162
\pinlabel {$\xf$} at 76 175
\pinlabel {$v$} at 90 204
\pinlabel {$v_s$} at 113 128
\pinlabel {$D$} at 223 135
\endlabellist
\includegraphics[width = 0.47\textwidth]{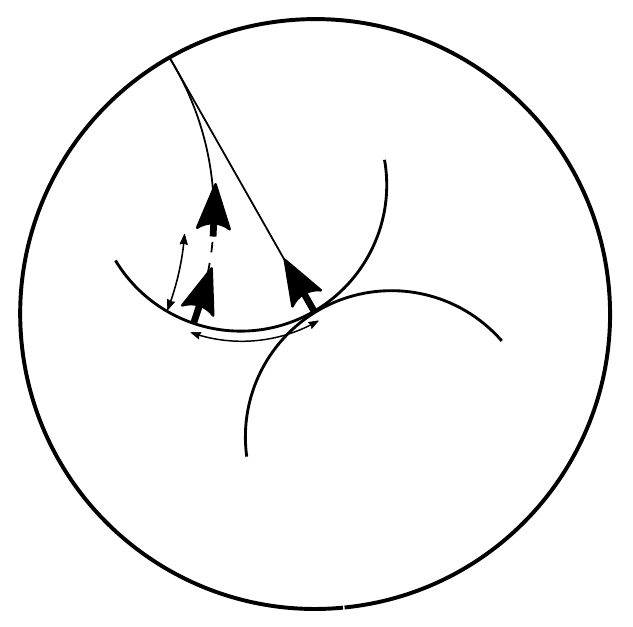} 
\label{Fig:localCoordinates}}
\hspace{-0.035\textwidth}
\subfloat[Flow.]{
\labellist
\small\hair 2pt
\pinlabel {$H$} at 207 190
\pinlabel {$\xu$} at 136 142
\endlabellist
\includegraphics[width = 0.47\textwidth]{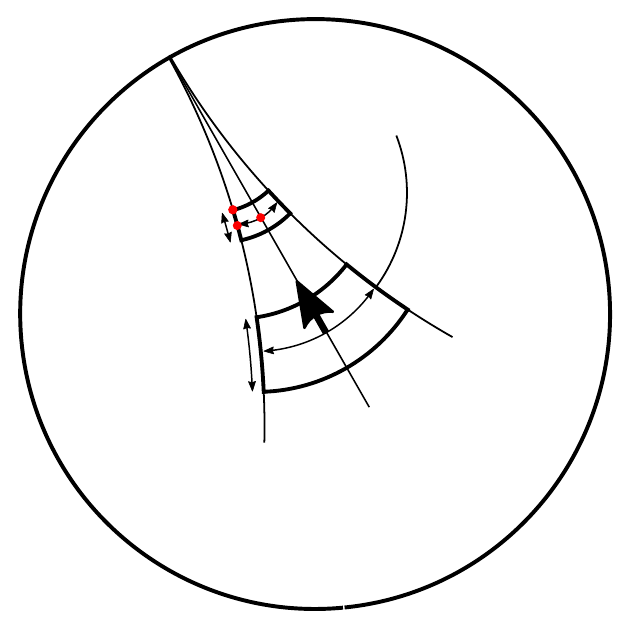}
\label{Fig:StableDirection}}
\caption{Coordinates for $\cover{X}=\UT{}{\cover{M}}=\UT{}{\HH^3}$ and flow of a box $(\xu, (-\varepsilon,\varepsilon), B^\sfs_\varepsilon(\pi(\xu)))$.}
\end{figure}

Suppose $N$ is a submanifold. 
Let $d_N(p,q)$ denote the length of the shortest curve in $N$ connecting $p$ and $q$. Given $\xu\in D$, let 
\[
B^\sfs_\varepsilon(\xu) = \{ \xs\in H : d_H(\xu, \xs) \leq \varepsilon\} \quad\mbox{where}\quad H = H^\sfs(\xu)
\]
Let
\[
D_\varepsilon = \{ (\xu, \xf, \xs) : \xu\in D, \xf \in (-\varepsilon, \varepsilon), \xs \in B^\sfs_\varepsilon(\xu) \} \subset \cover{X}
\]

\begin{remark} 
\label{Rem:product}
Note that the subscripts appearing in the coordinates $v=(\xu, \xf, \xs)$ refer to the unstable, flow, and stable foliations.
\begin{itemize}
\item The points $(\xu, 0, \xu)$ give a copy of $D$, which is unstable. 
\item If we fix $\xu$ and $\xs$, and vary $\xf$, then we obtain a geodesic flow line.
\item Also, if we fix $\xu$ and $\xf$, and vary $\xs$, then we obtain a stable horosphere. 
\end{itemize}

Note that each $H = H^\sfs(\xu)$ is isometric to a (pointed) copy of $\CC$.  
However, the coordinates above do not live in a geometric product $D \cross \RR \cross \CC$.  
They instead form a smooth fibre bundle over $D$.  (In the material case, the view $D$ is a copy of $S^2$.  
If we factor away the flow direction from our coordinates, what remains is isomorphic to the non-trivial bundle $\mathrm{T} S^2$.)
Thus we will only locally appeal to a ``product structure'' on these coordinates.
\end{remark}

The following lemma is deduced from the exponential convergence inside of stable leaves. 
See \reffig{StableDirection}.

\begin{lemma}  
\label{Lem:distFlowStable}
For all $t\geq 0$, $1>\varepsilon>0$ and all $v=(\xu, \xf, \xs) \in D_\varepsilon$, we have
\[
d_\cover{X}(\varphi_t(\xu),\varphi_t(v)) \leq 2 \varepsilon
\]
\end{lemma}

\begin{proof}
In our coordinates we have $\varphi_t(\xu) = (\xu, t, \xu)$ and $\varphi_t(v) = (\xu, \xf + t, \xs)$.  
We take $H=H^\sfs(\xu)$.  Then we have
\begin{align*}
d_\cover{X}((\xu,t,\xu), (\xu,t,\xs)) & \leq d_{\varphi_t(H)}((\xu,t,\xu), (\xu,t,\xs))  \\
 & = e^{-t} d_H((\xu,0,\xu), (\xu,0,\xs)) \\
 & \leq e^{-t} \varepsilon \leq \varepsilon \\
\intertext{and}
d_\cover{X}((\xu,t,\xs), (\xu,t+\xf,\xs)) &=  |\xf| \leq \varepsilon
\end{align*}
Applying the triangle inequality gives the result. 
\end{proof}

\subsection{Proof of the central limit theorem} \label{Sec:proofCLT}

We now use these coordinates to continue with the proof of \refthm{CLT}.

\begin{lemma} 
\label{Lem:almostEqualRandomVariables}
Let $(P,\nu)$ be a probability space. 
Let $R_T, S_T\from P\to \RR$ be a pair of one-parameter families of measurable functions. 
Assume that $\nu\circ R_T^{-1}\Rightarrow \psi$ where $\psi$ is a probability measure on $\RR$ with bounded probability density function $p\from\RR\to\RR_{\geq 0}$. 
Assume that there is a monotonically growing family $U_T\subset P$ of measurable sets such that $\left\| (R_T - S_T)|U_T\right\|_\infty\to 0$ and $\nu(P-U_T)\to 0$. 
Then $\nu\circ S_T^{-1}\Rightarrow \psi$.
\end{lemma}

\begin{proof}
Fix $\alpha$  and let 
\[
P_T = \nu \left[x\in P : S_T(x) \leq \alpha \right]\quad\mbox{and}\quad Q_T = \nu \left[x\in P : S_T(x) > \alpha \right]
\]

We need to show that for every $\varepsilon > 0$, there is a $T_0$ such that for all $T\geq T_0$ we have 
\[
\int_{-\infty}^{\alpha} p(s) \cdot ds - \varepsilon \leq P_T \leq   \int_{-\infty}^{\alpha} p(s) \cdot ds + \varepsilon
\]
We only deal with the second inequality since the first inequality can be derived in an analogous way using $P_T + Q_T = 1$. 
We have the following estimate.
\[
P_T \leq \nu\left[ x\in P:R_T(x) \leq \alpha + \left\| (R_T - S_T) |U_T\right\|_\infty \right] + \nu(P-U_T)
\]
Fix $\varepsilon > 0$. 
Let $\delta=\varepsilon / (3 \|p\|_\infty)$. 
By hypothesis, we have for all large enough $T$
\[
P_T \leq \nu\left[ x\in P:R_T(x) \leq \alpha + \delta \right] + \frac{\varepsilon}{3}
\]

Because $\nu \circ R_T^{-1} \Rightarrow \psi$, we furthermore have for all large enough $T$
\begin{align*}
 P_T &\leq \frac{\varepsilon}{3} + \int_{-\infty}^{\alpha + \delta} p(s) \cdot ds + \frac{\varepsilon}{3}\\
  & \leq  \frac{\varepsilon}{3} + \int_{-\infty}^{\alpha} p(s) \cdot ds + \|p \|_\infty \delta + \frac{\varepsilon}{3} = \int_{-\infty}^{\alpha} p(s) \cdot ds  + \varepsilon \qedhere
\end{align*}

\end{proof}

We return to the case of interest where $f$ is given by a one-form $\omega$.

\begin{lemma} 
\label{Lem:lieDerivativeNonVanishing}
Let $\omega \in \Omega^1(M)$ be closed but not exact. 
Then $\omega\from X\to\RR$ is not a derivative in the sense of \refdef{lieDerivative}.
\end{lemma}

\begin{proof}
We prove the contrapositive: 
that is, if $\omega$ is a derivative in the sense of \refdef{lieDerivative} then $\omega = d W$ for a function $W \from M \to \RR$.  
Fix a basepoint $p \in M$.  
We define $W(q) = \int_\gamma \omega$.  Here $\gamma$ is a path from $p$ to $q$.  
All that is left is to show that $W$ is well-defined. 

So, suppose that $\gamma'$ is another path from $p$ to $q$. 
Thus $z = \gamma - \gamma'$ is a cycle.  
Let $z^*$ be the geodesic representative of $z$.  
Since $\omega$ is closed we have $\int_z \omega = \int_{z^*} \omega$. 
Since $\omega$ is a derivative we have $\int_{z^*} \omega = 0$ and we are done.
\end{proof}

\begin{theorem} 
\label{Thm:sinaiTwo}
With the same hypotheses as in \refthm{CLT}, we have the following. 
There is $\sigma > 0$ such that for all views $D$ with area measure $\mu_D$, for all probability measures $\nu_D\ll \mu_D$, and for all $\alpha\in\RR$, we have
\[
\lim_{T\to\infty} \nu_D\left[ v\in D : R_T(v) \leq \alpha \right] = \int_{-\infty}^\alpha \frac{1}{\sigma\sqrt{2\pi}} e^{-(s/\sigma)^2/2}\, ds
\]
where $R_T(v) = \int_0^T \omega(\varphi_t(v)) dt / \sqrt{T}$.
\end{theorem}

\begin{proof}
The one-form $\omega$ is not a derivative by \reflem{lieDerivativeNonVanishing}. 
Taking $f=\omega$, let $\sigma$ be as in \refthm{sinaiOne}.

Fix a probability measure $\nu_D\ll\mu_D$. We define a measure $\nu_\cover{X}$ on $\cover{X}$ using the coordinates $v=(\xu, \xf, \xs)$ by taking the product of, in order,
\begin{itemize}
\item $\nu_D$
\item the Lebesgue measure on $\RR$ restricted to $[-1,1]$
\item the Lebesgue measure on $\CC\cong H^\sfs(\xu)$ restricted to the unit disk $B^{\sfs}_1(\xu)$
\end{itemize}
We scale $\nu_\cover{X}$ to be a probability measure. 
Note that the Lebesgue measure on $H^\sfs(\xu)$ does not depend on the isometric identification of $\CC$ with $H^\sfs(\xu)$. 
Thus $\nu_\cover{X}$ is well-defined.

By summing over fundamental domains, the probability measure $\nu_\cover{X}$ descends to a probability measure $\nu_X\ll \mu_X$ on $X$. Given that $R_T \from \cover{X} \to \RR$ is $\pi_1(M)$--invariant, \refthm{sinaiOne} yields $\nu_\cover{X} \circ R_T^{-1}\Rightarrow \psi_\sigma$.

Note that $\nu_\cover{X}$ is supported in the closure of $D_1$ (as defined before \refrem{product}). We have a projection $p\from D_1 \to D$ where $p(\xu,\xf,\xs)=(\xu,0,\xu)$. By construction, we have $\nu_\cover{X}(p^{-1}(U)) = \nu_D(U)$ for any measurable set $U\subset D$.

\begin{claim} 
\label{Clm:projConv}
We have
\[
\nu_\cover{X}\circ S_T^{-1}\Rightarrow \psi_\sigma\quad\mbox{where}\quad S_T = R_T\circ p
\]
\end{claim}

\begin{proof}
We take $P=D$ and we take $U_T=D$ for all $T$. Applying \reflem{almostEqualRandomVariables}, it is left to show that $\| R_T - S_T \|_\infty \to 0$. Let $W \from \cover{M} \to \RR$ be a primitive of $\cover{\omega}$. That is $dW = \cover{\omega}$. 
In an abuse of notation, we abbreviate $W \circ \pi \from \cover{X} \to \RR$ as $W$. 
Recall that $v = (\xu, \xf, \xs)$. 
We can now write
\[
R_T(v) = \frac{W(\varphi_T(v  )) - W(v  )}{\sqrt{T}}
\quad \mbox{and} \quad
S_T(v) = \frac{W(\varphi_T(\xu)) - W(\xu)}{\sqrt{T}} 
\]
Since $1/\sqrt{T} \to 0$, it is sufficient to show that both of 
\[
W(\varphi_T(v)) - W(\varphi_T(\xu)) \quad \mbox{and} \quad W(v) - W(\xu)
\]
are bounded by twice the Lipschitz constant of $W$. This follows from \reflem{distFlowStable} when replacing $\varepsilon$ by $1$ and setting $t$ to either $T$ or $0$.
\end{proof}
Fix $\alpha$. 
\refthm{sinaiTwo} follows from
\begin{align*}
\nu_D[v \in D: R_T(v) \leq \alpha] & = \nu_D[v \in D: S_T(v) \leq \alpha] \\
                                   &= \nu_\cover{X}(p^{-1}(\{v \in D: S_T(v) \leq \alpha\}))\\
                                   &= \nu_\cover{X}[v \in D_1:S_T(v)\leq \alpha] 
\end{align*}
converging to $\int^\alpha_{-\infty} n_\sigma(s)\, ds$ by \refclm{projConv}.
\end{proof}

\begin{proof}[Proof of \refthm{CLT}]
Note that \refthm{sinaiTwo} shows convergence in distribution for 
\[
R_T(v) = \frac{\int_0^T \omega(\varphi_t(v))\,dt}{\sqrt{T}}
\]
However we need to show convergence in distribution for $Q_T(v)= \Phi_T(v)/\sqrt{T}$, the difference being
\[
\Delta(v) = R_T(v) - Q_T(v) 
          = \frac{\int_b^{\pi(v)} \omega}{\sqrt{T}}=\frac{\int_b^{\pi(u)} \omega + \int_{\pi(u)}^{\pi(v)} \omega}{\sqrt{T}}
\]
where $u \in D$ is a fixed basepoint.
Thus, we need to show that \reflem{almostEqualRandomVariables} applies when taking $P = D$. 
Denote the constant $|\int_b^{\pi(u)} \omega|$ by $C$. 
Let $C'$ be a bound on the absolute value of $\omega\from X\to\RR$. 
It is convenient to let 
\[
U_T = \{ v \in D : d_D(u, v) \leq \sqrt[4]{T} \}
\]
Then, $\| \Delta | U_T \|_\infty \leq (C + C' \sqrt[4]{T}) /\sqrt{T} \to 0$ for $v\in U_T$. 
Since $U_T$ exhausts $D$ and $\nu$ is a finite measure, $\nu(D-U_T)\to 0$.
\end{proof}

\section{The pixel theorem}
\label{Sec:Pixel}

In this section, we prove that the cohomology fractal gives rise to a distribution at infinity. That is, integrating against the cohomology fractal then taking the limit as $R$ tends to infinity, gives a continuous linear functional on smooth, compactly supported functions.

\subsection{Motivation}

Throughout the paper, we have drawn many images of cohomology fractals, always depending on a visual radius $R$. 
The obvious question is whether there is a limiting image as $R$ tends to infinity. 

It turns out that the answer critically hinges on the question of what a pixel is. 
As we showed in \refthm{NoPicture}, thinking of a pixel as a sampled point does not work.  
After realising this, our next thought was that the cohomology fractal might converge to a signed measure $\mu$.
We managed to prove this for squares (as well as for regions with piecewise smooth boundary).  
However our proof does not generalise to arbitrary measurable sets.  
See \refsec{RemarkMeasureHard} for a discussion. 

We finally arrived at the notion of thinking of a pixel as a smooth test function; see~\cite{Smith95}.
The cohomology fractal now assigns to a pixel its weighted ``average value''; 
in other words, we obtain a well-defined \emph{distribution}. 
This distribution satisfies various transformation laws; 
these describe how it changes as we alter the chosen cocycle, basepoint, or view.  
To prove these we rely heavily on the \emph{exponential mixing} of the geodesic flow. 

\subsection{Background and statement}
\label{Sec:StatementPixel}

Before stating the theorem we establish our notation.
We define $\omega$, $b$, $D$, and $T$ as in \refsec{formalDefView}. 
However, as mentioned in \refrem{MeasuresForms}, 
we switch from using the area measure $\mu_D$ to the area two-form $\zeta_D$ and 
from a probability measure $\nu_D$ to a compactly supported two-form $\eta_D$.  
To obtain $\eta = \eta_D\in\Omega^2_c(D)$, we set $\eta_D = h\cdot \zeta_D$; 
here $h\in\Omega^0_c(D)$ is compactly supported and smooth. 
That is, $h$ is Hodge dual to $\eta$.

The function $h$ should be thought of as the kernel function for a pixel.  
The discussion below could be phrased completely in terms of $h$.
However, using $\eta$ allows us to neatly express the transformation laws between different views.

\begin{definition}
\label{Def:Distribution}
For a compactly supported two-form $\eta \in \Omega^2_c(D)$, we define
\[
\Phi^{\omega,b,D}(\eta) = \lim_{T\to\infty} \int_D \Phi^{\omega,b,D}_{T} \cdot \eta \qedhere
\]
\end{definition}

As we shall see, $\Phi^{\omega,b,D}$ is a \emph{distribution}: 
a continuous linear functional on $\Omega^2_c(D)$. We recall the topology on $\Omega^2_c(D)$ in the proof of \refthm{Pixel}.
We will use $\int_{D}$ to denote the \emph{canonical distribution} $\eta \mapsto \int_{D} \eta$.

To give a transformation law between views $D$ and $E$, we will need a way to relate one to the other.
Recall that $\cover{M}$ is isometric to $\HH^3$; thus we have $\bdy_\infty \cover{M} \homeo \CP^1$. 
As $t$ tends to infinity, the flow $\varphi_t$ takes a unit tangent vector $v \in D$ to some point $\varphi_\infty(v) \in \bdy_\infty\cover{M}$.
This induces a conformal embedding $i_D$ of $D$ into $\bdy_\infty \cover{M}$. 
We define $i_E$ similarly. 
We take $i_{E,D} = i_D^{-1} \circ i_E$ where it is defined.  
This is a \emph{conformal isomorphism} from (a subset of) $E$ to (a subset of) $D$.
We can now state the main result of this section. 

\begin{theorem}[Pixel theorem] 
\label{Thm:Pixel}
Suppose that $M$ is a connected, orientable, finite volume, complete hyperbolic three-manifold.  
Fix 
a closed, compactly supported one-form $\omega \in \Omega_c^1(M)$, 
a basepoint $b \in \cover{M}$, and
a view $D$.
\begin{enumerate}
\item 
\label{Itm:wellDefDist}
Then $\Phi^{\omega, b, D}$ is well-defined and is a distribution.
\item 
\label{Itm:cocycleIndep}
Given $\omega' \in \Omega_c^1(M)$ with $[\omega]=[\omega']$, there is a constant $C$ so that we have
\[
\Phi^{\omega', b, D} - \Phi^{\omega, b, D}= C \cdot \int_D
\]
\item 
\label{Itm:basePointPixelThm}
Given another basepoint $b' \in \cover{M}$, we have
\[
\Phi^{\omega, b', D} - \Phi^{\omega, b, D}= \left[ \int_b^{b'} \omega \right] \cdot \int_D
\]
\item 
\label{Itm:viewTransform}
Given another view $E$ and a two-form $\eta \in \Omega^2_c(\image(i_{E,D}))$, we have
\[
\Phi^{\omega,b,D}(\eta) = \Phi^{\omega, b, E}(i_{E,D}^* \, \eta)
\]
\end{enumerate}
\end{theorem}

The last property gives us a distribution at infinity as follows. 

\begin{corollary}
\label{Cor:BoundaryDistribution}
Suppose that $M$ is a connected, orientable, finite volume, complete hyperbolic three-manifold.  
Fix 
a closed, compactly supported one-form $\omega \in \Omega_c^1(M)$ and 
a basepoint $b \in \cover{M}$.
Then there is a distribution $\Phi^{\omega,b}$ on $\bdy_\infty \cover{M}$ so that, 
for any view $D$ and any $\eta \in \Omega^2_c(D)$, we have
\[
\pushQED{\qed}
\Phi^{\omega,b,D}(\eta) = 
\Phi^{\omega, b}((i_D^{-1})^* \eta) \qedhere
\popQED
\]
\end{corollary}

\subsection{Proof of the pixel theorem}
\label{Sec:PixelProof}

We now describe some background necessary to prove the theorem.
Throughout this section, we fix a connected, orientable, finite volume, complete hyperbolic three-manifold $M$.
Again we take $X = \UT{}{M}$ and $\cover{X} = \UT{}{\cover{M}}$.
Let 
\begin{equation} 
\label{Eqn:yt}
Y^{\omega,D}_t(\eta) = \int_D (\omega \circ \varphi_t) \cdot \eta
\end{equation}
Fubini's theorem implies that
\begin{equation} 
\int_D \Phi^{\omega,b,D}_{T} \cdot \eta =  \int_D \Phi^{\omega,b,D}_{0} \cdot \eta + \int_0^T Y^{\omega,D}_t(\eta)\, dt
\end{equation}
so much of the proof boils down to obtaining exponential decay of $Y_t=Y^{\omega,D}_t(\eta)$.  

In this section we will use the Haar measure $\mu_{\Haar}$ for integrals over $X = \UT{}{M}$.
We will use the shorthand $dv = d\mu_{\Haar}(v)$ throughout.

We also need to introduce the Sobolev norm $\sob_m = \sob_{m,\infty}$ for smooth functions $f$ on homogeneous spaces.
First consider functions $f \from G \to \RR$ where $G$ is a Lie group.
Fix a basis for the Lie algebra of $G$; we think of the elements in this basis as left-invariant vector fields on $G$.
The Sobolev norm $\sob_{m}(f)$ is the maximum of all $L_\infty$--norms of functions obtained by differentiating $f$ up to $m$ times using these vector fields in any order.
Suppose that $\Gamma$ and $H$ are (respectively) discrete and compact subgroups of $G$.
The Sobolev norm of $f \from \Gamma \backslash G \slash H \to \RR$ is the Sobolev norm of the lift of $f$ to $G$.

As usual, we have $M = \Gamma \backslash \HH^3 \homeo \Gamma \backslash \PSL(2,\CC) \slash \PSU(2)$.
Likewise, we have $X = \UT{}{M} \homeo \Gamma \backslash \PSL(2,\CC) \slash \PSO(2)$.
For any of the three views, material, ideal, or hyperideal, we can also express $D$ in this fashion.
For example, in the material view we have $D \homeo S^2 \homeo \PSU(2) \slash \PSO(2)$.

For $\eta \in \Omega^2_c(D)$, define $\sob_m(\eta) = \sob_m(h)$ where $h$ is the Hodge dual of $\eta$.
Note that the Sobolev norm depends on our choice of basis;
however, changing the basis changes the resulting norm by a bounded factor and thus only changes the constant $C$ in the following lemma.

\begin{lemma}
\label{Lem:expoDecay}
Let $M$ be a connected, orientable, finite volume, complete hyperbolic three-manifold.
There is a constant $m \in \NN$ such that the following is true.
Fix a view $D$ in $M$ and a smooth, compactly supported function $f \from X = \UT{}{M} \to \RR$ with $\int_X f(v)\, dv = 0$.
Fix a compact set $K \subset D$.
There are constants $C > 0$ and $c > 0$ such that for all two-forms $\eta \in \Omega_K^2(D)$ 
supported in $K$ and for all $t\geq 0$, we have
\[
\left| \int_D (f \circ \varphi_t) \cdot \eta \right| \leq Ce^{-ct} \sob_m(\eta)
\]
\end{lemma}

To prove this, we use the exponential decay of correlation coefficients for geodesic flows.   
This is a much studied area. 
We will rely on \cite{KelmerOh} because they explicitly give the dependence of the decay on the Sobolev norms of the functions involved. 
For hyperbolic, finite volume three-manifolds, their theorem can be simplified to the following.

\begin{theorem}
\label{Thm:mixing}
Let $M$ be a connected, orientable, finite volume, complete hyperbolic three-manifold. Then there exists  $m\in\mathbb{N}$, $C>0$ and $c>0$ with the following property. For any smooth functions $f, g \from X = \UT{}{M} \to \RR$ with $\int_{X} f(v)\, dv = 0$ and for all $t\geq 0$, we have
\begin{equation}
\left|\int_{X} f(\varphi_t(v)) g(v)\, dv\right| \leq C e^{-ct} \sob_{m}(f) \sob_{m}(g)
\end{equation}
\end{theorem}

\begin{proof}
The more general \cite[Theorem~3.1]{KelmerOh} relates to \refthm{mixing} as follows. 
They integrate over the frame bundle $\Gamma\backslash\PSL(2,\CC)$ using the Bowen-Margulis-Sullivan-measure. However, we can think of a function $X\to\RR$ as an $\PSO(2)$--invariant function on the frame bundle and the BMS-measure is simply the Haar measure in the case of a hyperbolic, finite volume three-manifold $\Gamma\backslash \HH^3$. 
Note that  \cite[Theorem~3.1]{KelmerOh} requires the functions $f$ and $g$ to be supported on a unit neighbourhood of the preimage of the convex core of $M$. 
However, for finite volume $M$, the convex core is just $M$. 
Furthermore, conventions for the Sobolev norm $\sob_m$ differ in whether to take the sum or maximum of the $L_\infty$--norms of derivatives; however the resulting norms are equivalent because they differ by a constant factor.
\end{proof}

Let $f\from X\to\RR$ be a compactly supported, smooth function and $\cover{f} = f\circ \pi\from \cover{X}\to\RR$ its lift.
To prove \reflem{expoDecay} using \refthm{mixing}, we construct test functions $h_\varepsilon \from \cover{X}\to\RR$ that tend to the given two-form $\eta\in\Omega^2_c(D)$ in the sense that 
\begin{equation}
\label{Eqn:ytmix}
Y^{f, D}_t(\eta) = \int\limits_\cover{X}(\cover{f} \circ \varphi_t)\cdot \eta
\end{equation}
can be approximated by 
\begin{equation} 
\label{Eqn:blurredIndicator}
Y^{f, D}_{t,\varepsilon}(\eta) = \int\limits_\cover{X} \cover{f}(\varphi_t(v)) h_\varepsilon(v) \, dv
\end{equation}

Note that there are several incompatibilities between $\eta$ and $h_\varepsilon$:
\begin{enumerate}
\item $\eta\in\Omega^2_c(D)$ is a two-form but $h_\varepsilon$ has to be a function.
\item $D\subset\cover{X}$ but the integral \refthm{mixing} is over $X$.
\item $D$ is two-, not five-dimensional.
\end{enumerate}
The first issue is solved by using the Hodge dual $h\in\Omega^0_c(D)$. That is, $\eta = h\cdot \zeta$ where $\zeta=\zeta_D$ is the area form on $D$.

For the second issue, we reformulate \refthm{mixing} as follows:
\begin{theorem} 
\label{Thm:mixing2}
Let $M$ be a connected, orientable, finite volume, complete hyperbolic three-manifold.
There is a constant $m\in\NN$ such that the following is true.
Fix a smooth, compactly supported function $f\from X = \UT{}{M} \to \RR$ with $\int_X f(v)\,dv = 0$ and a compact set $K\subset \cover{X}=\UT{}{\cover{M}}$.
There exists $C>0$ and $c>0$ such that for all smooth functions $g\from\cover{X}\to\RR$ supported in $K$ and all $t\geq 0$, we have
\begin{equation}
\left|\int_\cover{X} \cover{f}(\varphi_t(v)) g(v) \, dv\right| \leq C e^{-ct} \sob_{m}(g)
\end{equation}
\end{theorem}

\begin{proof}
Note that
\[
\int_\cover{X} \cover{f}(\varphi_t(v)) g(v) \, dv = \int_X f(\varphi_t(v)) g_{\sum}(v) \, dv
\]
where $g_{\sum}(v)$ is the sum of all $g(\cover{v})$ where $\cover{v}\in\cover{X}$ is a preimage of $v\in X$.
Since $K$ is covered by a finite number of copies of a fundamental domain of $M$, the sum $g_{\sum}(v)$ has a finite and bounded number of terms.
Since $f$ has compact support, $\sob_m(f)$ is finite and the result follows from \refthm{mixing}.
\end{proof}

To address the third issue, we make the following definition.
\begin{definition}
\label{Def:Bump}
Define an \emph{$\varepsilon$--bump function} by
\[
b_\varepsilon(x)= \begin{cases}
\exp\left(\frac{1}{(x/\varepsilon)^2 - 1}\right) & \mbox{if}\quad |x|<\varepsilon, \\
0 & \mbox{otherwise}
\end{cases}
\]
and set $B=\left[\int_{-\infty}^{\infty} b_1(x)\,dx\right] \left[\int_0^{\infty} b_1(r) 2\pi r dr\right]$. 
\end{definition}

We again use the coordinates on $\cover{X}$ already introduced in \refsec{coordinatesView}.
Recall that $H=H^\sfs(\xu)$. We define $v_s=d_H(\xu,\xs)$. Again, see \reffig{localCoordinates}. We now define
\begin{equation} 
\label{Eqn:defHEps}
h_\varepsilon(\xu,\xf,\xs)=h(\xu) \cdot \frac{b_\varepsilon(\xf)b_\varepsilon(v_s)}{B \varepsilon^3}
\end{equation}

Using Fubini's theorem, we can now write 
\begin{align*}
Y^{f,D}_{t,\varepsilon}(\eta) & = Y^{f,D}_{t,\varepsilon}(h\cdot \zeta) \\
& = \int_D \int_\RR \int_{H^\sfs(\xu)} \cover{f}(\varphi_t(v)) h_\varepsilon(v) J_D(v)\, d\xs d\xf d\xu
\end{align*}
where $v=(\xu,\xf,\xs)$, where $d\xu$ and $d\xs$ are using the area measures on $D$ and $H^\sfs(\xu)$, respectively, and where $J_D$ is the smooth function such that 
\begin{equation}
\label{Eqn:ScaleFactor}
d\mu_X = J_D(v)\, d\xs d\xf d\xu
\end{equation}

Note that, by construction, $J_D$ is invariant under isometries fixing $D$. In particular, $J_D(\xu) = J_D(\xu,0,\xu)$ is a positive constant. We set $J_0 = J_D(\xu)$.

Defining
\begin{equation} 
\label{Eqn:SliceBlurredIndicator}
Z^{f,D}_{t,\varepsilon}(\xu) = \int_\RR \int_{H^\sfs(\xu)}  \cover{f}(\varphi_t(v)) J_D(v) \frac{b_\varepsilon(\xf)b_\varepsilon(v_s)}{B \varepsilon^3} \, d\xs d\xf,
\end{equation}
where again $v=(\xu,\xf,\xs)$, we can write
\[
Y^{f,D}_{t,\varepsilon}(\eta) = Y^{f,D}_{t,\varepsilon}(h\cdot \zeta) =
\int_D Z^{f,D}_{t,\varepsilon}(\xu) h(\xu) \, d\xu
\]

\begin{lemma} 
\label{Lem:approxStableAndFlow}
For any smooth, compactly supported function $f\from X\to\RR$ and any view $D$, there is a $C>0$ such that for all $t\geq 0$, for all $1>\varepsilon>0$, and for all $\xu\in D$, we have
\[
\left| \cover{f}(\varphi_t(\xu)) \cdot J_0 
                 - Z^{f,D}_{t,\varepsilon}(\xu) \right| \leq C\varepsilon
\]
\end{lemma}

\begin{proof}
Note that the support of $h_\varepsilon$ is contained in the closure of $D_\varepsilon$ and that the second factor in \eqref{Eqn:defHEps} is normalised such that $$\int_\RR\int_{H^\sfs(\xu)} \frac{b_\varepsilon(\xf) b_\varepsilon(v_s)}{B\varepsilon^3} \, d\xs d\xf = 1$$ for all $\xu\in D$.
Thus, it suffices to show that there is a $C>0$ such that for all $t\geq 0$, for all  $1 > \varepsilon > 0$, and for all $v = (\xu,\xf,\xs) \in D_\varepsilon$, we have
\begin{equation} 
\label{Eq:lemApprox}
\left| \cover{f}(\varphi_t(\xu)) \cdot J_0 
               - \cover{f}(\varphi_t(v)) \cdot J_D(v) \right| \leq C \varepsilon
\end{equation}

\begin{claim} 
Let $f \from X \to \RR$ be a smooth, compactly supported function. 
Let $L=L(f)$ be a Lipschitz constant for $f$ and let $\cover{f}\from\cover{X}\to\RR$ be its lift. 
For all $t \geq 0$, for all $1 > \varepsilon > 0$, and for all $v=(\xu, \xf, \xs) \in D_\varepsilon$, we have
\[
\left| \cover{f}(\varphi_t(\xu)) - \cover{f}(\varphi_t(v))\right| \leq 2L\varepsilon \enspace \mbox{and} \enspace \left|\cover{f}(\varphi_t(\xu))\right| \leq \| f \|_\infty
\]
\end{claim}

\begin{proof}
The first inequality follows from \reflem{distFlowStable}. 
The second is by definition.
\end{proof}

\begin{claim}
The function $J_D$ has a finite Lipschitz constant $L'$ when restricted to $D_1$. Thus, for all $1>\varepsilon>0$ and for all $v=(\xu,\xf,\xs)\in D_\varepsilon$, we have
\[
\left| J_0 - J_D(v) \right| \leq 2L'\varepsilon
\]
\end{claim}

\begin{proof}
Since $J_D$ is invariant under isometries preserving $D$, we can assume that $\xu$ is fixed. Then, the domain $(\xu,(-1,1),B^{\sfs}_1(\xu))$ has compact closure and thus $J_D$ has a finite Lipschitz constant $L'$ when restricted to it. The claim now follows from \reflem{distFlowStable}.
\end{proof}

Using the above two claims, the left hand side of \eqref{Eq:lemApprox} is bounded by $2L\varepsilon \cdot J_0 +\|f\|_\infty \cdot 2L'\varepsilon  + 2L\varepsilon \cdot 2L' \varepsilon$. Thus setting $C=2L J_0 + 2L' \|f\|_\infty + 4L L'$ suffices to prove \reflem{approxStableAndFlow}.
\end{proof}

\begin{lemma} 
\label{Lem:approxStableFlow}
Let $M$ be a connected, orientable, finite volume, complete hyperbolic three-manifold and $f\from X\to\RR$ a smooth, compactly supported function. Fix a view $D$ and a compact set $K\subset D$. Let $\zeta=\zeta_D$ be the area form on $D$. There is a $C>0$ such that for all smooth $h\from D\to\RR$ supported in $K$ and for all $t\geq 0$, $1>\varepsilon>0$, we have
\[
\left| Y^{f,D}_{t}(h\cdot \zeta) J_0 - Y^{f,D}_{t,\varepsilon}(h\cdot \zeta) \right| \leq C\varepsilon \| h\|_\infty
\]
\end{lemma}

\begin{proof}
We have
\[
Y^{f,D}_{t}(h\cdot \zeta) J_0 - Y^{f,D}_{t,\varepsilon}(h\cdot \zeta) =  \int_D \left( \cover{f}(\varphi_t(\xu)) J_0 - Z^{f,D}_{t,\varepsilon}(\xu)\right) h(\xu) \, d\xu 
\]
which is bounded by
\[
C\varepsilon \int_D |h(\xu)| \, d\xu
\]
by \reflem{approxStableAndFlow}. 
The result now follows since $K$ has finite area.
\end{proof}

\begin{lemma} 
\label{Lem:sobNormBound}
Let $M$ be a connected, orientable, finite volume, complete hyperbolic three-manifold. Fix $m\in\NN$. Fix a view $D$ and a compact set $K\subset D$. There is a $C>0$ such that for all smooth $h\from D\to\RR$ supported in $K$ and all $1>\varepsilon>0$, we have
\[
\sob_m(h_\varepsilon) \leq C \varepsilon^{-(m+3)} \sob_m(h)
\]
where $h_\varepsilon$ is as defined in \eqref{Eqn:defHEps}.
\end{lemma}

\begin{proof}
We estimate $\sob_m(h_\varepsilon)$ by using that $h_\varepsilon$ is separable as defined in \eqref{Eqn:defHEps}. In suitable coordinates, the second factor can be written as 
\[
g_\varepsilon\from\RR^3\to\RR,\quad
(x, y, z)\mapsto \frac{b_\varepsilon(x) b_\varepsilon(\sqrt{y^2+z^2})}{B\varepsilon^3}
\]
We have $g_\varepsilon(u)=g_1(u/\varepsilon) / \varepsilon^{3}$ so $\sob_m(g_\varepsilon) = \varepsilon^{-(m+3)} \sob_m(g_1)$. 
Using that all $h_\varepsilon$ are supported in a common compact set, the lemma follows from the following fact about Sobolev norms.

Recall that a Sobolev norm requires a choice of vector fields that pointwise span the tangent space of the manifold (or a bundle over the manifold when using the Sobolev norm defined earlier by lifting a function $f\from \Gamma \backslash G \slash H\to \RR$ to $G$). However, any two such choices yield Sobolev norms that differ by a bounded factor when restricting to a small enough neighbourhood or compact set. In particular, up to a bounded factor, we can estimate the Sobolev norm $\sob_m(h_\varepsilon)$ by Sobolev norms using local coordinates in which $h_\varepsilon$ is separable. 
\end{proof}

\begin{proof}[Proof of \reflem{expoDecay}]
Let $m$ be as in \refthm{mixing2}. Fix a smooth, compactly supported function $f$, a view $D$, and a compact $K\subset D$.
\refthm{mixing2} states that there is a $C_0$ and $c_0$ such that for all smooth $h\from D\to\RR$ supported in $K$ and all $1 > \varepsilon > 0$, we may set $g=h_\varepsilon$ and have
\[
\left| Y^{f,D}_{t,\varepsilon}(h\cdot \zeta) \right| = \left| \int_\cover{X} \cover{f}(\varphi_t(v)) h_\varepsilon(v) \, dv \right| \leq C_0 e^{-c_0t} \sob_m(h_\varepsilon)
\]
Applying \reflem{approxStableFlow} to the left-hand side and \reflem{sobNormBound} to the right hand-side, there are $C_1$ and $C_2$ such that 
\[
\left| Y^{f,D}_{t}(h\cdot \zeta) J_0 \right| \leq C_1 \varepsilon \|h\|_\infty + C_0 e^{-c_0 t} C_2 \varepsilon^{-(m+3)} \sob_m(h)
\]
Since this holds for all $0 < \varepsilon < 1$, we can set $\varepsilon = e^{-c_0t/2(m+3)}$ and obtain 
\[
\left| Y^{f,D}_{t}(h\cdot \zeta)\right| \leq \frac{C_1 + C_0 C_2}{J_0} e^{-c_0t/2(m+3)} \sob_m(h) \qedhere
\]
\end{proof}

\begin{proof}[Proof of \refthm{Pixel}]
Recall that 
\begin{align}
\label{Eqn:Sum}
\Phi^{\omega,b,D}(\eta)=\int_0^\infty Y^{\omega,D}_t(\eta) \, dt + \int_D \left[\int_b^{\pi(v)} \omega\right] \cdot \eta
\end{align}

We now discuss the topology on $\Omega^2_c(D)$.
Many textbooks on distributions define the topology on $C^\infty_c(\RR^n)$.
This suffices in the ideal and hyperideal cases, where the view $D$ is diffeomorphic to $\RR^2$.
When $D$ is a two-sphere, we instead rely on the general theory of distributions on smooth manifolds.
We point the reader to \cite[Section~2.3]{BanCrainic} where such distributions are called ``generalised sections''. 
In particular, the definition of $\calD(M;E)$ in \cite{BanCrainic} gives the correct topology on $\Omega^0_c(D)\cong C^\infty_c(D)$ and thus on $\Omega^2_c(D)\cong\Omega^0_c(D)$.

Using \cite[Corollary~2.2.1]{BanCrainic} and the definition of the topology on $\Omega^2_c(D)$, it is not hard to see that statement~\refitm{wellDefDist} of the theorem is equivalent to the following claim.

\begin{claim}
For any compact $K\subset D$, there is a $C>0$ such that for all $\eta\in\Omega^2_K(D)$, the integrals in 
$\Phi^{\omega,b,D}(\eta)$ exist and we have 
\[
|\Phi^{\omega,b,D}(\eta)| \leq C \sob_m(\eta)
\]
\end{claim}

\begin{proof}
We apply \reflem{expoDecay} to $f = \omega \from X \to \RR$.  
The lemma implies that there is a $C_0 > 0$ such that for all $\eta \in \Omega^2_K(D)$,
the integral $\int_0^\infty Y^{\omega,D}_t(\eta)\, dt$ exists and is bounded by $C_0 \sob_m(\eta)$.
Let $C_1=\int_K \zeta_D$ be the area of $K$ and $C_2$ be the maximum of $\int_b^{\pi(v)} \omega$ over $v\in K$.
Recall that $h \in \Omega^0_K(D)$ denotes the Hodge dual to $\eta=h\cdot \zeta_D$. 
Then, by \refeqn{Sum}, we have the following.
\[
\left| \Phi^{\omega,b,D}(\eta) \right| \leq C_0 \sob_m(\eta) + C_1  C_2 \| h\|_\infty \leq (C_0 + C_1 C_2) \sob_m(\eta) \qedhere
\]
\end{proof}

Let $W\from \cover{M}\to\RR$ be a primitive such that $dW=\omega$.
We again abbreviate $W \circ \pi \from \cover{X} \to \RR$ as $W$.
We have
\begin{align}
\Phi^{\omega,b,D}(\eta)&=\lim_{T\to\infty} \int_{v\in D} \left(W(\varphi_T(v)) - W(b)\right)  \cdot \eta(v) \nonumber \\
& = \left[\lim_{T\to\infty} \int_D  (W\circ \varphi_T) \cdot \eta\right] - W(b) \cdot \int_D \eta
\label{eq:splitOffCanonical}
\end{align}

Statement~\refitm{cocycleIndep} is equivalent to the following claim.

\begin{claim}
If $[\omega] = 0$, there is a $C$ such that $\Phi^{\omega,b,D} = C \cdot \int_D $
\end{claim}

\begin{proof}
Since $\omega$ is a coboundary, its primitive descends to a map $W \from M \to \RR$.
Add a constant to $W$ to arrange $\int_M W \, dm = 0$; 
here we integrate with respect to the volume measure.
Now taking $f = W$, \reflem{expoDecay} now applies to \eqref{eq:splitOffCanonical} and we have
$\Phi^{\omega,b,D} = -W(b) \cdot \int_D$.
\end{proof}

Statement~\refitm{basePointPixelThm} follows from from \refrem{DependenceOnb}.

It is left to show Statement~\refitm{viewTransform}. 
Recall that we defined maps $i_D \from D\to\partial_\infty\cover{M}$ and $i_E\from E\to\partial_\infty\cover{M}$.  
Let $U=\image(i_D)\cap \image(i_E) \subset \partial_\infty\cover{M}$. 
Furthermore, let us define two distributions on $U$ by
\[
\Phi^{\omega,b}_{\leftarrow D}(\delta) = \Phi^{\omega,b,D}(i_D^*\delta)
\qquad\mbox{and}\qquad 
\Phi^{\omega,b}_{\leftarrow E}(\delta) = \Phi^{\omega,b,E}(i_E^*\delta)
\]

For every $\eta \in \Omega^2_c(\image(i_{E,D}))$, there is a compactly supported two-form $\delta$ on $U$ such that $\eta = i_D^* \delta$ and $i^*_{E,D} \eta = i_E^* \delta$. Thus, it is enough to show that
\begin{equation}
\label{Eqn:limViewSame}
\Phi^{\omega,b}_{\leftarrow D} = \Phi^{\omega,b}_{\leftarrow E}
\end{equation}

\begin{figure}\centering
\labellist
\small\hair 2pt
\pinlabel $s$ [B] at 89 179
\pinlabel $\varphi_T(i_D^{-1}(N))$ [Br] at 48 175
\pinlabel $\varphi_{T+\Delta T_p}(i_E^{-1}(N))$ [Bl] at 131 175
\pinlabel $i_D^{-1}(s)$ [tr] at 81 80
\pinlabel $i_E^{-1}(s)$ [tr] at 112 73
\pinlabel $\Delta T_s$ [bl] at 116 88
\pinlabel $v_s$ [b] at 104 79
\pinlabel $v$ [b] at 117 116
\pinlabel $H$ [Bl] at 128 115
\endlabellist
\includegraphics[height=8cm]{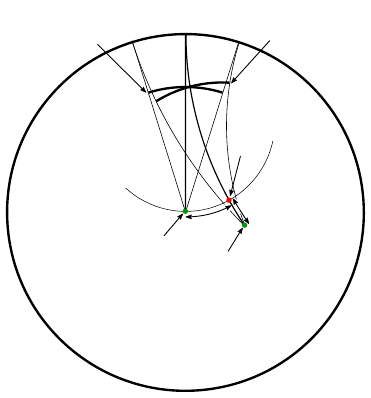}
\caption{The regions $\varphi_T(i_D^{-1}(N))$ and $\varphi_{T+\Delta T_p}(i_E^{-1}(N))$ for two material views. The distance between corresponding points in the two regions is bounded by a constant for all large enough $T$. The constant can be made arbitrarily small by making the neighbourhood $N$ about $p$ small enough.} 
\label{Fig:convergenceViews}
\end{figure}
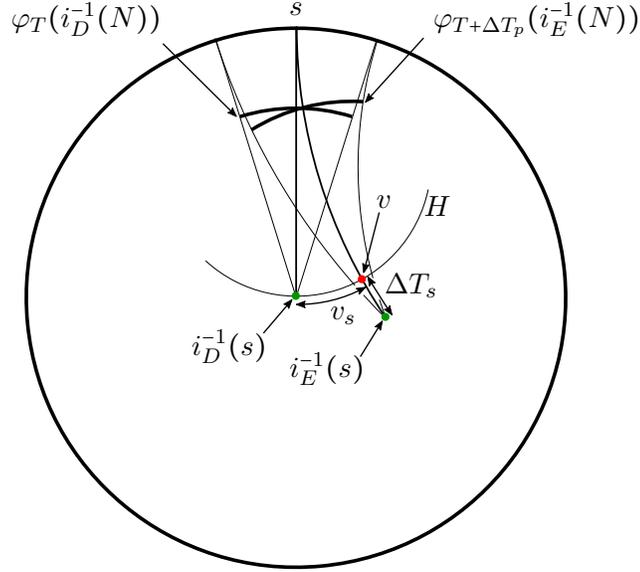

\begin{claim} 
\label{Clm:convergenceDistance}
For every $p\in U$ and $\varepsilon>0$, there is a neighbourhood $N\subset U$ of $p$, a number $\Delta T\in\RR$ and $T_0\geq 0$ such that for all $s\in N$ and $T\geq T_0$, we have
\[|W(\varphi_T(i_D^{-1}(s))) - W(\varphi_{T+\Delta T}(i_E^{-1}(s)))| \leq \varepsilon\]
\end{claim}

\begin{remark}
Note that the left hand side can be made arbitrary small by shrinking $N$. However, for fixed $N$ increasing $T$ might not suffice. This is why we will need to use a partition of unity argument to verify \eqref{Eqn:limViewSame}.
\end{remark}

\begin{proof}[Proof of \refclm{convergenceDistance}]
It might be helpful to consult \reffig{convergenceViews}.
Let $L$ be a Lipschitz constant of $\cover{W}\from X\to\RR$.

Suppose that $s$ lies in $U \subset S^2$.
Let $\gamma_D$ and $\gamma_E$ be the oriented, pointed, geodesics through $i_D^{-1}(s)$ and $i_E^{-1}(s)$ respectively. 
Both forward converge to $s$.
Let $H$ be the horosphere centered at $s$ containing $\pi(i_D^{-1}(s))$.
Note that $\gamma_E$ intersects $H$ orthogonally, say at $v$.
Let $\Delta T_s$ be the signed distance, along $\gamma_E$, from $\pi(i_E^{-1}(s))$ to $v$. 
Let $v_s = d_H(v, \pi(i_D^{-1}(s)))$.

Recall that $p$ lies in $U \subset S^2$. 
Recall that $\varphi_T$ is the geodesic flow. 
Exponential convergence and the triangle inequality give the following estimate:
\[
d_{\cover{X}}(\varphi_T(i_D^{-1}(s)), \varphi_{T+\Delta T_p}(i_E^{-1}(s))) \leq \sqrt{2} v_s e^{-T} + |\Delta T_p - \Delta T_s|
\]
We refer to \refrem{extrinsicCurvCorr} to explain the factor of $\sqrt{2}$.

There is a neighbourhood $N$ of $p$ such that for all $s \in N$, we have $|\Delta T_p - \Delta T_s| \leq \varepsilon / 2L$.
Note that $v_s$ is bounded on $N$ so there is a $T_0$ such that $\sqrt{2} v_s e^{-T_0}\leq \varepsilon / 2L$. 
We deduce that
\[
d_\cover{X}(\varphi_T(i_D^{-1}(s)), \varphi_{T+\Delta T_p}(i_E^{-1}(s)))\leq \varepsilon / L
\]
Since $L$ is a Lipschitz constant for $\cover{W}$, the claim follows by setting $\Delta T = \Delta T_p$.
\end{proof}

We now define a norm $\| \cdot \|_1$ on $\Omega^2_c(U)$. 
Identify $\partial_\infty \cover{M}$ with $S^2$. 
Such an identification induces an area form $\zeta_\infty$ on $\partial_\infty \cover{M}$. 
Given $\delta \in\Omega^2_c(U)$, let $d$ be the Hodge dual, that is $\delta=d\cdot \zeta_\infty$. 
Let
\[
\| \delta \|_1 = \int_U |d| \cdot \zeta_\infty
\]
Note that $\|\delta\|_1$ does not depend on the identification of $\partial_\infty \cover{M}$ with $S^2$.

\begin{claim} 
\label{Clm:convergenceRelatedInts}
For every $p\in U$ and $\varepsilon > 0$, there is a neighbourhood $N \subset U$ of $p$ such that for all $\delta\in \Omega^2_c(N)$, we have
\[
\left| \Phi^{\omega,b}_{\leftarrow D}(\delta) - \Phi^{\omega,b}_{\leftarrow E}(\delta) \right| \leq \varepsilon \|\delta\|_1
\]
\end{claim}

\begin{proof}
Let $N$, $\Delta T$, and $T_0$ as in \refclm{convergenceDistance}. 
Add a constant to the primitive $W$ of $\omega$ such that $W(b)=0$. 
Then, the difference 
$\Phi^{\omega,b}_{\leftarrow D}(\delta) - \Phi^{\omega,b}_{\leftarrow E}(\delta)$
can be expressed using $W$ as follows.
\begin{align*}
 & \lim_{T\to\infty}\int_U (W \circ \varphi_T \circ i_D^{-1}) \cdot \delta - \lim_{T\to\infty}\int_U (W\circ \varphi_T\circ i_E^{-1}) \cdot \delta = \\
 & \lim_{T\to\infty}\int_U (W\circ \varphi_T\circ i_D^{-1}) \cdot \delta - \lim_{T\to\infty}\int_U (W\circ \varphi_{T+\Delta T}\circ i_E^{-1}) \cdot \delta = \\
 & \lim_{T\to\infty}\int_U ( W\circ \varphi_T\circ i_D^{-1} - W\circ \varphi_{T+\Delta T}\circ i_E^{-1}) \cdot \delta
\end{align*}
\refclm{convergenceDistance} implies that the integral is bounded by $\varepsilon \|\delta\|_1$, for all $T_0\geq T$. 
Thus the limit is bounded by $\varepsilon \|\delta\|_1$.
\end{proof}

We now verify \eqref{Eqn:limViewSame}.
Fix a smooth, compactly supported $\delta\in\Omega^2_c(U)$.
Also fix $\varepsilon>0$. 
For each $p\in U$, pick a neighbourhood $N$ as in \refclm{convergenceRelatedInts}.
The support $\supp(\delta)$ can be covered by finitely many of these neighbourhoods.
Consider a partition of unity with respect to this finite cover.
By multiplying $\delta$ by the partition functions, we obtain smooth two-forms $\delta_1,\dots, \delta_n$.
Let $d_i$ be the Hodge dual of $\delta_i$. 
Since $\sum |d_i| = |d|$ (pointwise) we have $\sum \|\delta_i\|_1 = \|\delta\|_1$.
Thus \refclm{convergenceRelatedInts} implies
\[
\left| \Phi^{\omega,b}_{\leftarrow D}(\delta) - \Phi^{\omega,b}_{\leftarrow E}(\delta) \right| \leq \varepsilon \|\delta\|_1
\]
Since $\varepsilon$ was arbitrary, we are done.
\end{proof}

\begin{remark}
McMullen~\cite{McMullen19} suggests another approach to \refthm{Pixel}.
Arrange matters so that $W$, the primitive of $\omega$, is harmonic with respect to the hyperbolic metric.
Prove that $W$ has suitably bounded growth as it approaches $\bdy_\infty \HH^3$, in terms of the euclidean metric in the Poincar\'e ball model. 
Now prove and then apply an appropriate version (of part (iii) implies part (i)) of~\cite[Theorem~1.1]{Straube84}.
\end{remark}


\subsection{The cohomology fractal as a measure} 
\label{Sec:RemarkMeasureHard}

We do not know whether the cohomology fractal converges to a signed measure; 
that is we do not know whether or not 
\[
\lim_{T\to\infty} \int_U \Phi^{\omega,b,D}_T \cdot d\mu_D
\]
converges for every measurable $U\subset D$. 
Instead, we have the following partial result.

\begin{theorem}[Square pixel theorem] 
\label{Thm:SquarePixel}
Suppose that $M$ is a connected, orientable, finite volume, complete hyperbolic three-manifold.  
Fix 
a closed, compactly supported one-form $\omega \in \Omega_c^1(M)$, 
a basepoint $b \in \cover{M}$, and
a view $D$. 
Suppose that $U \subset D$ is bounded. Suppose further that $\partial U$ has finite length. Then the following limit exists
\[
\lim_{T\to\infty} \int_U \Phi^{\omega,b,D}_T \cdot d\mu_D
\]
\end{theorem}

\begin{remark}
Recall that the pictures in \reffig{RagainstNumSamples} were generated by uniformly sampling across a pixel square. 
\refthm{SquarePixel} finally proves that this technique (with enough samples) will give accurate images in some non-empty range of visual radii.
Note that the earlier \refthm{Pixel} is not sufficient; 
it required a smooth filter function.
\end{remark}

\begin{remark}
It seems difficult to generalise \refthm{SquarePixel} to measurable subsets.  
Our proof does not apply, for example, to an open set $U \subset D$ bounded by several Osgood arcs~\cite{Osgood03}. 
\end{remark}

Before sketching a proof of \refthm{SquarePixel}, we discuss \emph{mollification}.
Suppose that $h\from D \to \RR$ is a bounded measurable function with compact support. 
We define $h_\varepsilon\from D \to \RR$ by setting 
\begin{equation}
\label{Eqn:hEpsUnstable}
h_\varepsilon(u) = 
\frac{1}{B_\varepsilon} \int_D h(v) \cdot b_\varepsilon(d_D(u, v)) \cdot d\mu_D(v) 
\end{equation}
Here $B_\varepsilon$ is a normalisation factor such that 
$\int_D h \cdot d\mu_D=\int_D h_\varepsilon \cdot d\mu_D$. 
We call $h_\varepsilon$ the \emph{$\varepsilon$--mollification (in the unstable direction)} of $h$.

\begin{lemma} 
\label{Lem:technicalSquarePixel}
Let $h\from D\to\RR$ be a bounded measurable function with compact support. 
Assume that there are constants $c>0$ and $C>0$ such that for all $1 > \varepsilon >0$, we have
\[
\| h - h_\varepsilon  \|_1 \leq C\varepsilon^c
\]
Then the following limit exists
\[
\lim_{T\to\infty} \int_D \Phi^{\omega,b,D}_T \cdot h \cdot d\mu_D
\]
\end{lemma}

\begin{proof}[Proof sketch]
We will show that $Y^{\omega,D}_t(h\cdot d\mu_D) = \int_D (\omega \circ \varphi_t) \cdot h\cdot d\mu_D$ decays exponentially.
Let $\eta_\varepsilon = h_\varepsilon \cdot \zeta_D$ and $Y^{\omega,D}_{t,\varepsilon}(h\cdot d\mu_D) = \int_D (\omega \circ \varphi_t) \cdot \eta_\varepsilon$.
Regarding $\omega$ as function $X\to\RR$, we have
\[
\left|Y^{\omega,D}_t(h\cdot d\mu_D) - Y^{\omega,D}_{t,\varepsilon}(h\cdot d\mu_D)\right| \leq \| h - h_\varepsilon  \|_1 \cdot \| \omega\|_\infty
                          \leq C\varepsilon^c \|\omega\|_\infty
\]
Since $h$ is bounded, the Sobolev norm of the mollification $h_\varepsilon$ of $h$ can be estimated from the Sobolev norm of the mollification kernel. 
This can be done similarly to \reflem{sobNormBound}.
Thus, there is a $C_0>0$ such that for all $1>\varepsilon > 0$, we have $\sob_m(\eta_\varepsilon) \leq C_0 \varepsilon^{-(m+2)}$.
Taking $f = \omega$ and $\eta = \eta_\varepsilon$, \reflem{expoDecay} states that there is a $C_1>0$ and $c_1>0$ such that for all $t\geq 0$ and $1>\varepsilon>0$, we have
\[
\left|Y^{\omega,D}_{t,\varepsilon}(h\cdot d\mu_D)\right| \leq C_1 e^{-c_1t} \varepsilon^{-(m+2)}
\]
Thus, we have
\[
\left|Y^{\omega,D}_{t}(h\cdot d\mu_D)\right| \leq C \varepsilon^c \| \omega\|_\infty + C_1e^{-c_1 t} \varepsilon^{-(m+2)}
\]
We obtain exponential decay when setting $\varepsilon = e^{-c_1 t / 2 (m+2)}.$
\end{proof}

\begin{proof}[Proof sketch of \refthm{SquarePixel}]
The theorem follows from \reflem{technicalSquarePixel} by taking $h$ to be the indicator function $\chi_U$. 
It is left to show that there is a $C$ such that $\| h - h_\varepsilon\|_1 \leq C \varepsilon$.

Note that $h-h_\varepsilon$ is bounded by one, thus it is enough to show that the area where $h-h_\varepsilon$ is non-trivial is bounded by $C\varepsilon$.
Tthe area of this neighbourhood is bounded above, up to multiplication by a universal constant, by the length of $\partial U$ multiplied by $\varepsilon$.
\end{proof}

\section{Questions and projects}
\label{Sec:Questions}

\begin{question}
Suppose that $F$ is a surface in a hyperbolic three-manifold $M$. 
\refthm{CLT} tells us that the standard deviation $\sigma$ is a topological invariant of the pair $(M,F)$. What are the number theoretic (or other) properties of $\sigma$?  Fixing $M$, does $\sigma$ ``see'' the shape of the Thurston norm ball?
\end{question}

\begin{question}
Suppose that $F$ is a fibre of a closed, connected hyperbolic surface bundle $M$.  In \refprop{LightDark} we showed that approximations of the Cannon--Thurston map are (components of) level sets of the cohomology fractal.  Is there some more precise sense in which the Cannon--Thurston map $\Psi$ is a ``level set'' of the distributional cohomology fractal $\Phi^{F, b}$?
\end{question}

\begin{question}
Can the cohomology class $[\omega]$ be recovered from the distributional cohomology fractal $\Phi^{\omega, b}$?
\end{question}

\begin{question}
\reffig{SubPixEvolution} suggests that in the example of \texttt{m122(4,-1)} the mean has settled down at around $R=10$ for a pixel size of $0.1^\circ$.
We also see this in \reffig{RagainstNumSamples} in that there is hardly any difference between the images at $R = 10$ and $R=12$ with $128\times 128$ samples.
\refthm{SquarePixel} tells us that given enough samples we can produce an accurate picture of the distributional cohomology fractal.
Can one calculate effective bounds that would allow us to produce a provably correct image?
\end{question}

\begin{question}
In \reffig{RagainstNumSamples},  the image with $1 \times 1$ samples and $R = 8$ is very similar to the image with $128 \times 128$ samples and $R = 12$. 
However, we do not understand how an image generated with only one sample per pixel can so closely approximate the limiting object.
The manifold \texttt{m122(4,-1)} is small; as a result, perhaps the geodesic flow mixes rapidly enough?  
Does this fail in larger manifolds?
\end{question}

\begin{question}[Mark Pollicott]
We consider lowering the dimension of $F$ and $M$ by one.  Let $F$ be a non-separating curve in a closed, connected hyperbolic surface $M$. Fix a point $p \in M$.  Let $P$ be a ``pixel'' -- that is, a closed arc in $\UT{p}{M} \homeo S^1$ with centre $c_P$ and radius $r_P$. The distributional cohomology fractal $\Phi^F$ exists and in fact gives a ``signed measure'' to each such pixel $P$ (see \refsec{RemarkMeasureHard}). We define a function $\mu^F\from S^1\cross (0,\pi]\to \RR$ by setting $\mu^F(c_P, r_P)=\Phi^F(P)$.  What does the graph of $\mu^F$ look like?  For example, what happens if we fix $c_P$ and allow the radius to vary?
How does the graph behave as $r_P$ approaches zero?
\end{question}

\begin{question}
The histogram in \reffig{m122(4,1)Dist} is low near the mean.  Increasing $R$ (within the range that we trust our experiments, see \refsec{Accumulate}) reduces, but does not remove, this gap.  Why is it there?  (This does not seem to happen in \reffig{s789Dist}, where the surface is closed but the manifold is not.)
\end{question}

\begin{question}
\label{Que:Cusped}
Consider the experiment shown in \reffig{s789Histogram2}.  
Here the support of the cocycle $\omega$ is not compact.  
We see that the distribution of the cohomology fractal, over a pixel, appears to not be normal. 
Further experiments show that it depends sensitively on the choice of pixel. 
Can one verify rigorously that it is not normal?

We suspect that some version of ``subtracting the largest excursion'' 
(see the remarks immediately before \cite[Theorem~1]{DiamondVaaler86}) 
will yield a more reasonable distribution.
\end{question}

\begin{question}
\refthm{Pixel} applies to cocycles $\omega$ with compact support.
Consider a cusped manifold and a cocycle $\omega$ such that the pullback of $[\omega]$ to the cusp torus is non-trivial.
In this case, the Sobolev norm of $\omega$ is infinite;
thus \refthm{mixing} does not apply. 
Are there modifications, perhaps as indicated in \refque{Cusped}, 
so that we again obtain a distribution at infinity for the cohomology fractal?
\end{question}

\begin{question}
In the fibred case, what is the relationship between the cohomology fractal and the lightning curve? See \refsec{Lightning}.
\end{question}

We end with some ideas for future software projects.

\begin{project}
One could use material triangulations in the closed case to draw an approximation $\Psi_D$ to the Cannon--Thurston map, following \refalg{CTApprox}.  By \refprop{LightDark}, these match the cohomology fractal.  Motivated by Figures~\ref{Fig:m122_4_-1} and~\ref{Fig:Orbifold}, we anticipate that $\Psi_D$ will look significantly different from Cannon--Thurston maps in the cusped case; the ``mating dendrites'' that approximate $\Psi$ have bounded branching at all points. 
\end{project}

\begin{project}
In \refsec{Cone}, we discussed cohomology fractals for incomplete structures along a line in Dehn surgery space. As discussed in \refsec{NumericalInstability}, all of these suffer from numerical defects along the incompleteness locus $\Sigma_s$. These defects are visible (although small) in \reffig{Bending}.
When the slope $s$ is integral, we use material triangulations (\refsec{MaterialTriangulations}) to remove these defects. For general $s$, material triangulations are not available.  Instead, one could \emph{accelerate} through tubes about $\Sigma_s$. That is, we modify the cellulation of the manifold by truncating each tetrahedron, replacing the lost volume with a solid torus ``cell'' around $\Sigma_s$. 
\end{project}

\appendix
\section{Notation}
\label{App:Notation}
For the convenience of the reader, we list some of the notation used in the paper.  

\newcolumntype{L}[1]{>{\raggedright\let\newline\\\arraybackslash\hspace{0pt}}p{#1}}
\newcolumntype{C}[1]{>{\centering\let\newline\\\arraybackslash\hspace{0pt}}p{#1}}
\newcolumntype{R}[1]{>{\raggedleft\let\newline\\\arraybackslash\hspace{0pt}}p{#1}}

\begin{longtable}{L{0.13\textwidth}L{0.81\textwidth}}
\renewcommand{\arraystretch}{1.3}
\texttt{m004} & SnapPy notation (for the figure-eight knot complement). 
We use $\texttt{m004(p,q)}$ to denote a $(p, q)$ Dehn filling. \\
$M$ & connected, oriented, finite volume hyperbolic three-manifold.  \\
$\calT$ & triangulation of $M$, see \refsec{Triangulations}.  \\
$F$ & connected, oriented, finite-area hyperbolic surface.   \\
$\cover{F}$, $\cover{M}$ & universal covers. \\
$\bdy_\infty \cover{F}$ & ideal boundary of $\cover{F}$. \\
$\bdy_\infty \cover{M}$ & ideal boundary of $\cover{M}$. \\
$p$, $q$, $b$ & points of $M$ or $\cover{M}$.\\
$\Psi$ & Cannon--Thurston map (\refthm{CannonThurston}).\\
$\Psi_D$ & approximation to the Cannon--Thurston map (\refalg{CTApprox}).\\
$\UT{}{}$ & unit tangent bundle. \\
$u$, $v$ & unit vectors in $\UT{}{M}$ or $\UT{}{\cover{M}}$. \\
$\pi$ & unit tangent bundle map (or ratio of circumference to diameter of circle in the euclidean plane).  \\
$\omega$ & one-cocycle for M. We often replace $\omega$ by a (Poincar\'e dual) surface $F$ in our notation.\\

$R$, $T$ & radius and a time respectively.  
These serve the same purpose; 
we use one or the other as we are thinking geometrically (\refdef{Basic})
or dynamically (\refdef{OneForm}). \\

$\Phi_R^{\omega, p}$ & the cohomology fractal at radius $R$ (\refdef{Basic}). We suppress superscripts which are not relevant to the discussion at hand. Several variants of this notation follow.\\
$\Phi_R^{F, p}$ & (\refdef{Dual}). \\
$\Phi_R^{\omega, b, p}$ & (\refrem{Basepoint}). \\
$\Phi_R^{\omega, b, D}$ & (\refdef{View} and Equation~\ref{eq:cohomFractView}).\\
$\Phi_R^{F, b, D}$ & (\refsec{Experiments}). \\

$\cover{\omega}$ & lift of $\omega$ to $\cover{M}$ (\refdef{UniversalCover}).   \\
$W$ & primitive for $\cover{\omega}$ (\refdef{UniversalCover} and \refsec{formalDefView}).  \\
$\varphi_t$ & geodesic flow for time $t$ (\refdef{OneForm}).\\
$D$, $E$ & views in $\UT{}{\cover{M}}$ (\refsec{Views}). \\
$U$ & subdomain of $D$: for example, a pixel. \\

$X$, $\cover{X}$ & abbreviations for $\UT{}{M}$ and $\UT{}{\cover{M}}$ in Sections~\ref{Sec:CLT} and~\ref{Sec:Pixel}. \\
$\mu_D$ & induced area measure on $D$ (\refsec{formalDefView}). \\
$\zeta_D$ & area form on $D$ (so $\zeta_D = d\mu_D$). \\
 $h$ & measurable (\refeqn{RN}) or smooth (\refsec{StatementPixel}) test function from $D$ to $\RR$.\\ 
 $h_\varepsilon$ & $\varepsilon$--mollification of $h$ in $\cover{X}$ (\refeqn{defHEps}) or in the unstable direction (\refeqn{hEpsUnstable}).\\
$\nu \ll \mu$ & $\nu$ (always a probability measure) absolutely continuous with respect to $\mu$ (\refsec{formalDefView}).  
We decorate with the measure space (as a subscript) when needed for context. \\

 $R_T, S_T$ & random variable, usually $R_T=\Phi_T/\sqrt{T}$ (Sections~\ref{Sec:StatementCLT} and~\ref{Sec:proofCLT}).\\ 
 $\sigma$ & standard deviation (square root of the variance). \\
$\psi_\sigma$ & normal distribution with standard deviation $\sigma$ (\refsec{StatementCLT}).\\
$n_\sigma$ & probability density function for $\psi_\sigma$ (\refsec{StatementCLT}). \\
$\Rightarrow$ & convergence in measure (\refdef{ConvergeInDistribution}). \\
 $\to$ & convergence in probability (\refdef{ConvInProb}).  \\
$\mu_X$ & Haar measure on $X$ scaled to be a probability measure (\refsec{Sinai}). \\
$\xu,\xf, \xs$ & coordinates for the unstable, flow, and stable manifolds in a neighbourhood of a view (\refsec{coordinatesView}).\\
$H^\sfs(\xu)$ & stable manifold through $\xu$ (\refsec{coordinatesView}).\\
$d_N$ & induced distance on a submanifold $N$ (\refsec{coordinatesView}).\\
$B^\sfs_\varepsilon(\xu)$ & $\varepsilon$--ball in $H^\sfs(\xu)$ (\refsec{coordinatesView}).\\
$D_\varepsilon$ & product $\varepsilon$--neighbourhood of the view $D$ (\refsec{coordinatesView}).\\

$\eta$ & two-form in $\Omega^2_c(D)$ (\refsec{StatementPixel}).\\  
$\Phi^{\omega, b, D}$ & the distributional cohomology fractal (\refdef{Distribution}).\\
$\Phi^{\omega, b}$ & the cohomology fractal at infinity (\refcor{BoundaryDistribution}).\\
$\int_D$ & canonical distribution (\refsec{StatementPixel}).\\
$i_D$ & conformal embedding of $D$ into $\bdy_\infty \cover{M}$ (\refsec{StatementPixel}). \\
 $i_{D,E}$ & conformal isomorphism from (a subset of) $E$ to (a subset of) $D$  (\refsec{StatementPixel}). \\
 
 $Y^{\omega,D}_t(\eta)$ & slice of the integral defining $\Phi^{\omega, b, D}(\eta)$ (\refeqn{yt}).\\
 $Y^{f,D}_t(\eta)$ & $Y^{\omega,D}_t(\eta)$ generalised to $f\from X=\UT{}{M}\to\RR$ (\refeqn{ytmix}).\\
 $Y^{f,D}_{t,\varepsilon}(\eta)$ & mollified slice  (\refeqn{blurredIndicator}). \\
 $Z^{f,D}_{t,\varepsilon}(\xu)$ & mollified sample  (\refeqn{SliceBlurredIndicator}). \\
 $\sob_m$ & Sobolev norm (\refsec{PixelProof}).\\
 $f$, $g$ &  functions in the mixing theorems (\reflem{expoDecay}, \refthm{mixing}, and \refthm{mixing2}). \\
 $b_\varepsilon, B$ & bump function and normalising factor  (\refdef{Bump}).\\
 $J_D$ & ratio between $d\mu_X$ and $d\xs d\xf d\xu$ (\refeqn{ScaleFactor}).\\
 $J_0$ & equals $J_D(\xu, 0, \xu)$ (\refeqn{ScaleFactor}).\\
\end{longtable}


\renewcommand{\UrlFont}{\tiny\ttfamily}

\let\originalpath\path
\renewcommand{\path}[1]{\tiny\ttfamily\originalpath{#1}}

\let\originalthebibliography\thebibliography
\let\originalendthebibliography\endthebibliography
\renewenvironment{thebibliography}[1]{%
    \begin{originalthebibliography}{{MSW02}}}{\end{originalthebibliography}}

\bibliographystyle{alphaNoUrl}
\bibliography{cohomology_fractals_exp_math}
\end{document}